\documentclass[11pt,leqno]{amsart}

\usepackage{enumerate}
\usepackage{graphicx}
\usepackage{epsfig}
\usepackage{amssymb,amsmath,amsthm,amsfonts}
\usepackage {latexsym}

\usepackage{bbm}

\allowdisplaybreaks

\newcommand{\nc}{\newcommand}
\nc{\les}{\lesssim}
\nc{\nit}{\noindent}
\nc{\nn}{\nonumber}
\nc{\D}{\partial}
\nc{\diff}[2]{\frac{d #1}{d #2}}
\nc{\diffn}[3]{\frac{d^{#3} #1}{d {#2}^{#3}}}
\nc{\pdiff}[2]{\frac{\partial #1}{\partial #2}}
\nc{\pdiffn}[3]{\frac{\partial^{#3} #1}{\partial{#2}^{#3}}}
\nc{\abs}[1] {\lvert #1 \rvert}
\nc{\cAc}{{\cal A}_c}
\nc{\cE}{{\cal E}}
\nc{\cF}{{\cal F}}
\nc{\cP}{{\cal P}}
\nc{\cV}{{\cal V}}
\nc{\cQ}{{\cal Q}}
\nc{\cGin}{{\cal G}_{\rm in}}
\nc{\cGout}{{\cal G}_{\rm out}}
\nc{\cO}{{\cal O}}
\nc{\Lav}{{\cal L}_{\rm av}}
\nc{\cL}{{\cal L}}
\nc{\cB}{{\cal B}}
\nc{\cZ}{{\cal Z}}
\nc{\cR}{{\cal R}}
\nc{\cT}{{\cal T}}
\nc{\cY}{{\cal Y}}
\nc{\cX}{{\cal X}}
\nc{\cXT}{{{\cal X}(T)}}
\nc{\cBT}{{{\cal B}(T)}}
\nc{\vD}{{\vec \mathcal{D}}}
\nc{\efield}{\mathcal{E}}
\nc{\vE}{{\vec \efield}}
\nc{\vB}{{\vec \mathcal{B}}}
\nc{\vH}{{\vec \mathcal{H}}}
\nc{\ty}{{\tilde y}}
\nc{\tu}{{\tilde u}}
\nc{\tV}{{\tilde V}}
\nc{\Pc}{{\bf P_c}}
\nc{\bx}{{\bf x}}
\nc{\bX}{{\bf X}}
\nc{\bXYZ}{{\bf XYZ}}
\nc{\bY}{{\bf Y}}
\nc{\bF}{{\bf F}}
\nc{\bS}{{\bf S}}
\nc{\dV}{{\delta V}}
\nc{\dE}{{\delta E}}
\nc{\TT}{{\Theta}}
\nc{\dPsi}{{\delta\Psi}}
\nc{\order}{{\cal O}}
\nc{\Rout}{R_{\rm out}}
\nc{\eplus}{e_+}
\nc{\eminus}{e_-}
\nc{\epm}{e_\pm}
\nc{\eps}{\varepsilon}
\nc{\vnabla}{{\vec\nabla}}
\nc{\G}{\Gamma}
\nc{\w}{\omega}
\nc{\mh}{h}
\nc{\mg}{g}
\nc{\vphi}{\varphi}
\nc{\tlambda}{\tilde\lambda}
\nc{\be}{\begin{equation}}
\nc{\ee}{\end{equation}}
\nc{\ba}{\begin{eqnarray}}
\nc{\ea}{\end{eqnarray}}

\nc{\g}{\gamma}
\nc{\ol}{\overline}

\newtheorem{theorem}{Theorem}[section]
\newtheorem{lemma}[theorem]{Lemma}
\newtheorem{prop}[theorem]{Proposition}
\newtheorem{corollary}[theorem]{Corollary}
\newtheorem{defin}[theorem]{Definition}
\newtheorem{rmk}[theorem]{Remark}

\nc{\pT}{\partial_T}
\nc{\pz}{\partial_z}
\nc{\pt}{\partial_t}
\nc{\la}{\langle}
\nc{\ra}{\rangle}
\nc{\infint}{\int_{-\infty}^{\infty}}
\nc{\halfwidth}{6.5cm}
\nc{\figwidth}{10cm}
\newcommand{\f}{\frac}

\nc{\nlayers}{L} \nc{\nsectors}{M}
\nc{\indicator}{\mathbf{1}}
\nc{\Rhole}{R_{\rm hole}}
\nc{\Rring}{R_{\rm ring}}
\nc{\neff}{n_{\rm eff}}
\nc{\Frem}{F_{\rm rem}}
\nc{\R}{\mathbb R}
\nc{\Z}{\mathbb Z}
\nc{\DD}{\Delta}
\nc{\cD}{\mathcal D}
\nc{\lnorm}{\left\|}
\nc{\rnorm}{\right\|}
\nc{\rnormp}{\right\|_{\ell^{p,\eps}}}
\nc{\rar}{\rightarrow}
\begin{document}

\begin{abstract}

	We investigate dispersive estimates for the
	Schr\"odinger operator $H=-\Delta +V$ with $V$ is
	a real-valued decaying potential when there
	are zero energy resonances and eigenvalues in four
	spatial dimensions.  If there is a zero energy
	obstruction, we establish the low-energy expansion
	$$
		e^{itH}\chi(H) P_{ac}(H)=O(1/(\log t)) A_0+
		O(1/t)A_1+O((t\log t)^{-1})A_2+ O(t^{-1}(\log t)^{-2})A_3.	
	$$
	Here $A_0,A_1:L^1(\R^n)\to L^\infty (\R^n)$,
	while
	$A_2,A_3$ are operators between logarithmically
	weighted spaces, with $A_0,A_1,A_2$ finite rank
	operators, further the operators are independent
	of time.  We show that similar expansions
	are valid for the solution operators to Klein-Gordon
	and wave equations.  Finally, we show that
	under certain orthogonality conditions, if there is
	a zero energy eigenvalue one can recover the
	 $|t|^{-2}$ bound as an
	operator from $L^1\to L^\infty$.  Hence, recovering
	the same dispersive bound as the free evolution
	in spite of the zero energy eigenvalue.
	
\end{abstract}

\title[Decay estimates for Schr\"odinger, Klein-Gordon and wave equations]{\textit{Decay estimates for 
four dimensional Schr\"odinger, Klein-Gordon and wave equations with 
obstructions at zero energy}}

\author[Green, Toprak]{William~R. Green and Ebru Toprak}
\thanks{The second author was partially supported by NSF grant DMS-1201872 and DMS-1501041 and would like to thank Burak Erdogan for the financial support. }

\address{Department of Mathematics\\
Rose-Hulman Institute of Technology \\
Terre Haute, IN 47803, U.S.A.}
\email{green@rose-hulman.edu}
\address{Department of Mathematics \\
University of Illinois \\
Urbana, IL 61801, U.S.A.}
\email{toprak2@illinois.edu}

\maketitle

\section{Introduction}

The free Schr\"odinger evolution on $\R^n$, $e^{-it\Delta}$ maps $L^1(\R^n)$ to $L^\infty(\R^n)$ with norm bounded by
$C_n |t|^{-n/2}$.  This can be seen by the triangle
inequality and the representation 
$$
	e^{-it\Delta}f(x)= \frac{1}{(4\pi i t)^{\frac{n}{2}}} \int_{\R^n} e^{-i|x-y|^2/4t}f(y)\, dy.
$$
In this paper we study the dispersive
properties of the operator $e^{itH}$ where $H=-\Delta+V$
is a Schr\"odinger operator perturbed by a real-valued,
decaying potential $V$.  Formally, this defines the
solution operator to the perturbed Schr\"odinger equation
\begin{align}\label{Schr eqn}
	iu_t+Hu=0, \qquad u(x,0)=f(x).
\end{align}
That is, the solution to
\eqref{Schr eqn} may be expressed as $u(x,t)=e^{itH}f(x)$.

Quantifying the dispersive properties of the solution operator is a well-studied problem.  In general, 
with $P_{ac}$ the projection onto the absolutely continuous
spectral subspace of $L^2(\R^n)$ associated to the
Schr\"odinger operator $H$,
the
dispersive estimates are expressed as
\begin{align}\label{dispgoal}
	\|e^{itH}P_{ac}(H)\|_{L^1\to L^\infty}\les |t|^{-\f n2}.
\end{align}
One requires the projection onto the absolutely continuous spectrum
 as the perturbed Schr\"odinger operator
often possesses point spectrum, for which large time decay
cannot occur.  Under weak pointwise assumptions on the
potential, say $|V(x)|\les \la x\ra^{-\beta}$ for some
$\beta>1$, we have $\sigma(H)=\sigma_{ac}(H)\cup \sigma_p (H)$.  Here the absolutely continuous spectrum
$\sigma_{ac}(H)=[0,\infty)$, and the point spectrum consists
of a finite collection of non-positive eigenvalues, 
\cite{RS1}.  One may alternatively seek to quantify the
dispersive properties in terms of Strichartz norms,
$L^p$ bounds, or micro-local estimates.  In this paper,
we focus on proving point-wise bounds for the
evolution when there are zero energy obstructions.
The obstructions
can be related to solutions to $H\psi=0$.  If $\psi \in L^2(\R^n)$, there is a zero energy eigenvalue, while $\psi\notin L^2(\R^n)$ is a resonance
if it is in a different space which depends on the dimension.
When $n=4$, $\psi$ is a resonance if $\la \cdot \ra^{0-}\psi \in L^{2}(\R^4)$.

Local dispersive estimates were first studied treating
$e^{itH}P_{ac}$ as an operator between weighted $L^2(\R^n)$
spaces.  The study was begun by Rauch in \cite{Rauch} on 
exponentially weighted spaces when $n=3$.  Jensen and Kato \cite{JenKat} for $n=3$, and Jensen \cite{Jen,Jen2}
for $n>3$ proved estimates on polynomially weighted $L^2$ spaces that decay at a rate of $|t|^{-\f n2}$.
Murata, \cite{Mur}, studied local dispersive estimates for a wide class of Schr\"odinger-like equations.  In these
works, it was shown that threshold obstructions can effect
the time decay of the solution operator even though
$P_{ac}(H)$ explicitly projects away from the zero
energy eigenspace.

In recent years, there has been much interest in global
dispersive estimates, in which one seeks to bound
$e^{itH}P_{ac}(H)$ as an operator from $L^1(\R^n)$ to
$L^\infty (\R^n)$.  The study began with the seminal paper
of Journ\'e, Soffer and Sogge \cite{JSS}, with much of
the recent work having its roots in the approach of Rodnianski and
Schlag, \cite{RodSch}.  If zero energy is regular, that
is if there are no zero-energy eigenvalues
of resonances, \eqref{dispgoal} has been established in all dimensions, see 
\cite{Wed,GS,Sc2,GV,CCV,EG1,Gr}, a more thorough
history may be found in \cite{Scs}.

When there is an obstruction at zero energy, either a
resonance or an eigenvalue, in general, the time decay
is slower.  The effect of these threshold obstructions is
dimension-specific, we note \cite{GS,ES2,Yaj,golde,EG2,Bec,EGG,GGodd,GGeven,Top} in which the effects
were studied in all dimensions $n\geq 1$ in the sense of
global dispersive estimates.

In this article we provide refined dispersive
bounds for the Schr\"odinger operator with zero-energy
resonances and/or eigenvalues in dimension $n=4$.  
We recall the result of Erdo\smash{\u{g}}an, Goldberg
and the first author, \cite{EGG}, it was shown that
if zero is not regular, one has the low energy expansion
$$
	\|e^{itH}\chi(H)P_{ac}(H)-F_t\|_{L^1\to L^\infty}\les |t|^{-1}, \qquad |t|>2
$$
where $F_t$ is a finite rank operator satisfying
$\|F_t\|_{L^1\to L^\infty} \les (\log |t|)^{-1}$ for
$|t|>2$.  We improve these results in several directions.
First of all, we provide a more detailed expansion the
evolution with at most three slowly decaying terms with an error term that is integrable for large time.
In particular, we show
that as $t\to \infty$,
the non-integrable time decay is ``small" in the 
sense of being attached to finite-rank operators.

To state the results, we define the functions
$\log^+(x)= \chi_{\{x>1\}} \log x$ and
$w(x)=1+\log^+|x|$, and $\varphi(t)$ is a function that satisfies the bound
$\varphi(t)=O(1/\log t)$ for $t>2$.  Further,
$\chi$ is an smooth, even cut-off function supported in $[-2\lambda_1,2\lambda_1]$ for a fixed sufficiently small $\lambda_1>0$ and it is equal to one if  $|\lambda| \leq \lambda_1$.
We also define 
the logarithmically weighted $L^p$ spaces
$$
	L^1_{w^k}(\R^4)=\{f \, :\, \int_{\R^4}w^k(x) |f(x)| \, dx<\infty
	\}, \qquad
	L^\infty_{w^{-k}}(\R^4)=\{ g\, :\, \|w^{-k}(\cdot)g\|_\infty <\infty \}.
$$

\begin{theorem}\label{thm:major}
	
	Suppose that $|V(x)|\les \la x\ra^{-\beta-}$.
	If zero energy is not regular, 
	for $t>2$,
	$$
		e^{itH}\chi(H) P_{ac}(H)=\varphi(t) A_0+
		O(1/t)A_1+O((t\log t)^{-1})A_2+ O(t^{-1}(\log t)^{-2})A_4,
	$$	
	with $A_0:L^1\to L^\infty$ is a finite rank
	operator, $A_1:L^1\to L^\infty$,
	$A_2:L^{1}_w\to L^{\infty}_{w^{-1}}$ finite rank operators,
	and $A_4:L^1_{w^3}\to L^{\infty}_{w^{-3}}$,
	provided $\beta>0$ is large enough.  
	Furthermore, the operators $A_1,A_2$ are independent of
	time.  	In particular,
	\begin{enumerate}[i)]
	\item If there is a resonance but no eigenvalue at zero and $\beta>4$, the above expansion
	is valid.
	
	\item If there is an eigenvalue but no resonance at zero and $\beta>8$, the above expansion is valid with
	$A_0=0$.
	
	\item If there is a resonance and an eigenvalue at zero and $\beta>8$, the above expansion is valid.
	
\end{enumerate}

\end{theorem}

The polynomially weighted $L^p$ spaces are defined by
$$
  L^{p,\sigma}(\R^n)=\{f\, : \|\la x\ra^{\sigma} f\|_p<\infty\}.
$$
To prove this theorem, 
we employ an interpolation argument
between the results of \cite{EGG} and the
three parts of the following
theorem which we prove in Sections~\ref{sec:first},
\ref{sec:second} and \ref{sec:third} respectively.

\begin{theorem}\label{thm:main}
	
	Suppose that $|V(x)|\les \la x\ra^{-\beta-}$.
	\begin{enumerate}[i)]
	\item If there is a resonance but no eigenvalue at zero,
		then if $\beta>4$ for $t>2$, 
	$$
		e^{itH}\chi(H) P_{ac}(H)=\varphi(t) A_0+
		O(1/t)A_1+O((t\log t)^{-1})A_2+ O(t^{-1}(\log t)^{-2})A_3+O(t^{-1-})A_4
	$$
	with $A_0:L^1\to L^\infty$ a rank one
	operator, $A_1:L^1\to L^\infty$,
	$A_2:L^{1}_w\to L^{\infty}_{w^{-1}}$ finite rank operators, $A_3:L^1\to L^\infty$
	and $A_4:L^{1,\f12}\to L^{\infty,-\f12}$.  
	
	\item If there is an eigenvalue but no resonance at zero,
	then if $\beta>8$, for $t>2$, 
	$$
		e^{itH}\chi(H) P_{ac}(H)=
		O(1/t)A_1+O(t^{-1-})A_4
	$$
	with $A_1:L^1\to L^\infty$ a finite rank operators,
	and $A_4:L^{1,\f12}\to L^{\infty,-\f12}$.
	
	\item If there is a resonance and an eigenvalue at zero,
	then if $\beta>8$ for $t>2$, 
	$$
		e^{itH}\chi(H) P_{ac}(H)=\varphi(t) A_0+
		O(1/t)A_1+O((t\log t)^{-1})A_2+ O(t^{-1}(\log t)^{-2})A_3+O(t^{-1-})A_4
	$$
	with $A_0,A_1,A_3:L^1\to L^\infty$,
	$A_2:L^{1}_w\to L^{\infty}_{w^{-1}}$ finite rank operators,
	and $A_4:L^{1,\f12}\to L^{\infty,-\f12}$.
	
\end{enumerate}
Furthermore, the operators $A_1,A_2,A_3$ are independent of time.  

\end{theorem}

The operators  $A_3,A_4$ need not be finite rank.  However,
their
contribution is integrable for large time.
To establish Theorem~\ref{thm:major}, we recall that
in \cite{EGG}, it was shown that
\begin{align}\label{EGG bound} 
	e^{itH}\chi(H) P_{ac}(H)=\varphi(t) A_0+O(1/t),	
\end{align} 
with $A_0:L^1\to L^\infty$ finite rank, and the error
term is understood as an operator mapping $L^1\to L^\infty$ which is not finite rank.
In Theorem~\ref{thm:main}, we show
$$
	e^{itH}P_{ac}(H)-\varphi(t)A_0
	-\frac{w(x)B_1w(y)}{t\log t}
	-\frac{B_2}{t(\log t)^2}=O\left(\frac{\la x\ra^{\f12}
	\la y\ra^{\f12}}{t^{1+}}\right).
$$
Here $B_1,B_2$ are finite rank and bounded independent of $x,y$.  We can
subtract them off of the bound \eqref{EGG bound} to 
conclude that
$$
	e^{itH}P_{ac}(H)-\varphi(t)A_0
	-\frac{w(x)B_1w(y)}{t\log t}
	-\frac{B_2}{t(\log t)^2}=O\left(
	\frac{w(x)w(y)}{t}\right).
$$
So that, 
$$
	e^{itH}P_{ac}(H)-\varphi(t)A_0
	-\frac{w(x)B_1w(y)}{t\log t}
	-\frac{B_2}{t(\log t)^2}=O\left(\frac{w(x)w(y)}{t}
	\min\bigg( 1, \frac{\la x\ra^{\f12}
		\la y\ra^{\f12}}{t^{0+}}\bigg)
	\right).
$$
Using $\min(1,\f ab)\les (\log a)^2/(\log b)^2$ if $a,b>2$,
we can obtain an error term bounded by $t^{-1}(\log t)^{-2}$ 
as an operator between
logarithmically weighted spaces.

We can combine the low energy estimates, for which the
spectral parameter $\lambda$ is in a sufficiently
small neighborhood of zero, proven above with large energy
estimates from, for example \cite{YajEven} or \cite{CCV}. 
To apply these
results, one requires additional regularity on the potential, which is expected from the counterexample
constructed in \cite{GV}.  This counterexample showed that
the high energy portion of the evolution need not satisfy
the desired dispersive bound if $V$ does not have
$\frac{n-3}{2}$ continuous derivatives when $n>3$.
The smoothness when $n=4$ may be
expressed in terms of a weighted Fourier transform as in \cite{YajEven}, or
explicitly requiring $V\in C^{\f12+}(\R^4)$ as in \cite{CCV}.  We note that the
goal of \cite{YajEven} was not to prove dispersive bounds directly, but was instead
concerned with the $L^p$-boundedness of the wave operators, which are defined by
$$
  W_{\pm}=s-\lim_{t\to\pm \infty} e^{itH}e^{it\Delta}.
$$
The $L^p$ boundedness of the wave operators allows one to deduce bounds on the
perturbed operator from the free operator $H_0=-\Delta$.  That is, for any Borel
function $f$,
$$
  f(H)P_{ac}=W_\pm f(-\Delta) W_\pm^*.
$$
If $W_\pm$ is bounded on $L^\infty$, then $W_\pm^*$ is bounded on $L^1$ and one
can deduce the $|t|^{-\f n2}$ dispersive bound for the evolution by  using the natural bound for the free equation.
It is known that the if zero is not regular, the range of
$p$ for which the wave operators are bounded 
shrinks from $1\leq p\leq \infty$ when zero is
regular to $\frac{4}{3}<p<4$ when there is an eigenvalue but no resonance at zero, \cite{JY4}.  We expect that the range of $p$ can
be expanded to $1\leq p<4$
in light of the recent works \cite{YajNew,GGwaveop,YajNew2}.\footnote{During the review period of this paper, the first author and Goldberg proved this result, see \cite{GGwaveop4}.}

A second direction in which we improve the known results
is to establish a global dispersive bound that matches the
natural $|t|^{-2}$ decay of the free evolution even in
the presence of a zero-energy eigenvalue.
In \cite{GGeven,GGodd} 
$L^1\to L^\infty$ dispersive bounds for the perturbed evolution
were established in the presence of zero energy eigenvalues with the full
$|t|^{-\f n2}$ time decay, assuming additional orthogonality conditions between the
zero energy eigenspace and the potential.  In these papers, 
dimensions $n\geq 5$ were studied, where zero energy resonances do not occur.  
We show in Section~\ref{sec:t-2 decay} that such a
bound holds in dimension $n=4$.  The more complicated
structure of zero energy obstructions when $n=4$ leads to
significant technical difficulties in the analysis, even when there
is not a zero energy resonance.

\begin{theorem}\label{thm:evalcancel}
 
  Assume that $|V(x)|\les \la x\ra^{-12-}$, and that zero is an eigenvalue of $H=-\Delta+V$, but
not a resonance.  Further, suppose that $\int_{\R^4} V\psi \, dx=0$ 
and $\int_{\R^4} x_jV\psi \, dx=0$ for each $\psi\in$Null $H$ and all $1\leq j\leq 4$.  Then,
$$
  \|e^{itH}\chi(H) P_{ac}(H)\|_{L^1\to L^\infty} \les |t|^{-2}.
$$

\end{theorem}

One can alternatively state the orthogonality hypotheses as $P_eVx,P_eV1=0$ where
$P_e$ is the projection onto the zero energy eigenspace.

The perturbed resolvent operators are defined by 
$$R_V^{\pm}(\lambda^2) = R_V^{\pm}(\lambda^2 \pm i0 )= \lim_{\epsilon \to 0^+}(H - ( \lambda^2 \pm i\epsilon))^{-1}.$$

By Agmon's well-known limiting absorptions principle, \cite{agmon}, these limits are well-defined as bounded  operators between weighted $L^2$ spaces. Treating $e^{itH}\chi(H)P_{ac}(H)$ as an element of functional calculus, Stone's formula yields the representation
  \begin{align}
\label{stone}
 \ e^{itH}  \chi(H) P_{ac}(H) f(x) =\frac1{2\pi i}  \int_0^{\infty} e^{it\lambda^2} \lambda  \chi( \lambda )  [R_V^+(\lambda^2)-R_V^{-}(\lambda^2)]  f(x) \,d\lambda. 
 \end{align}
 
   Here the difference of the resolvents provides the absolutely continuous spectral measure.   
  As usual (cf. \cite{RodSch,GS,Sc2,EG2,EGG}) the proofs of Theorem~\ref{thm:main} and Theorem~\ref{thm:evalcancel}  relies on the formula \eqref{stone} and the expansion of the spectral density $[R_V^+(\lambda^2)-R_V^{-}(\lambda^2)]$ around zero energy which varies depending on which kind of obstruction one has at zero energy.

The final direction in which we improve the known results
is to prove dispersive bounds for a wide class of
wave-like equations.
The 
Klein-Gordon equation
\begin{align}
	u_{tt}-\Delta u+m^2 u+Vu=0, \qquad u(0)=f, \qquad\partial_t u(0)=g \label{eq:KG}
\end{align}
is formally solved by
\begin{align}\label{eq:KGsoln}
	u(x,t)=\cos(t\sqrt{H+m^2})f(x)
	+\frac{\sin(t\sqrt{H+m^2})}{\sqrt{H+m^2}}g(x).
\end{align}
The formal solution
is valid for the wave equation, when $m^2=0$.  We restrict
our attention to $m^2\geq 0$ so that the unperturbed
operator $-\Delta+m^2$ is positive.  

In the
free case, when $V=0$, one has the natural dispersive
bounds
\begin{align*} 
	\|u(\cdot, t)\|_{\infty} \les |t|^{-\f32}(\|f\|_{W^{k+1,1}}+\|g\|_{W^{k,1}})
\end{align*}
for $k>\f32$ in four dimensions.  The derivative loss
on the initial data is strictly a high-energy phenomenon, see \cite{Gr2,EGG}. 
$L^\infty$ bounds on solutions to the wave equation has
been studied, \cite{BS,Beals,Cucc,DP,BeGol,Gr2}.  The dispersive
nature of the Klein-Gordon has been studied in various
senses \cite{MSW,KK,GLS}.  The effect of threshold
eigenvalues and resonances for wave-like equations has
been studied in dimensions $n\leq 4$, see
\cite{KS,Gr,EGG}.

We prove low-energy dispersive bounds by taking advantage
of the representation
\begin{multline*}
	\cos(t\sqrt{H+m^2})P_{ac}
	+\frac{\sin(t\sqrt{H+m^2})}{\sqrt{H+m^2}}P_{ac}\\
	=\frac{1}{\pi i} \int_0^\infty\bigg(\cos(t\sqrt{\lambda^2+m^2})+  \frac{\sin(t\sqrt{\lambda^2+m^2})}{\sqrt{\lambda^2+m^2}}\bigg) \lambda
	[R_V^+(\lambda^2)-R_V^-(\lambda^2)]\, d\lambda.
\end{multline*}
This allows us to extend our analysis of the spectral
measure and the perturbed resolvents we develop for
the Schr\"odinger evolution in Sections~\ref{sec:first},
\ref{sec:second} and \ref{sec:third} below.

\begin{theorem}\label{thm:KG}

Let $m^2>0$ then the results of Theorem~\ref{thm:main} are valid if
the operator $e^{itH}$ is replaced by either
$\cos(t\sqrt{H+m^2})$ or $\frac{\sin(t\sqrt{H+m^2})}{\sqrt{H+m^2}}$.

\end{theorem}

  Note that Theorem~\ref{thm:KG} is stated for $m^2>0$. When $m^2=0$,  \eqref{eq:KGsoln} corresponds to the solution of the wave equation. This formula suggests a similar statement for the solution of the wave equation. However, the behavior of  $ \frac{\sin(t \sqrt {\lambda^2+m^2})} {\sqrt{\lambda^2+m^2}} $ when $m=0$  for small $\lambda$ is singular.  In \cite{EGG}, 
  to ensure integrability as $\lambda \to 0$
  we had to use
    $$ 
    \bigg|\frac{\sin(t\lambda)}{\lambda}\bigg| \les |t|  \quad \textrm{whereas} \quad \bigg| \frac{\sin(t \sqrt {\lambda^2+m^2})} {\sqrt{\lambda^2+m^2}}\bigg| \les \frac{1}{m}\les 1\quad \textrm{for} \quad m>0. $$

 Considering this fact together with $|\cos(\lambda t)| \les 1$, the results of Theorem~\ref{thm:KG} apply 
 to the operator $\cos(t\sqrt{H})$, but 
 the bounds for the sine operator must all be multiplied 
 by $t$. 

As in our analysis for the Schr\"odinger evolution, we consider only the low energy
portion of the evolution.  High energy bounds, with a loss of $\f52+$ derivatives of the initial data, 
for the wave equation are proven in
\cite{CV} under the assumption that $V\in C^{\f12}(\R^4)$, and 
$|\partial_x^\alpha V(x)|\les \la x\ra^{-5/2-}$ for $|\alpha|\leq \frac{1}{2}$.

\section{Resolvent expansions around zero}\label{sec:exp}

Much of this discussion appears in \cite{EGG}, we include
and expand upon it for completeness.  
The notation
$$
f(\lambda)=\widetilde O(g(\lambda))
$$
denotes
$$
\frac{d^j}{d\lambda^j} f = O\big(\frac{d^j}{d\lambda^j} g\big),\,\,\,\,\,j=0,1,2,3,...
$$
Unless otherwise specified, the notation refers only to derivatives with respect to the spectral variable $\lambda$.
If the derivative bounds hold only for the first $k$ derivatives we  write $f=\widetilde O_k (g)$.  
If we write $f(\lambda)=\widetilde O_k (\lambda^j)$, it 
should be understood that differentiation is comparable
to division by $\lambda$.  That is, $|\frac{d^\ell}{d\lambda^\ell}f(\lambda)|\les \lambda^{j-\ell}$, even for $\ell>j$.
In this
paper we use that notation for operators as well as
scalar functions, the meaning should be clear from
context.

Most properties of the low-energy expansion for $R_V^\pm(\lambda^2)$ are
inherited in some way from the free resolvent
$R_0^\pm(\lambda^2) = (-\Delta-(\lambda^2\pm i0))^{-1}$.  In this section we gather facts
about $R_0^\pm(\lambda^2)$ and examine the  relationship between $R_V^\pm(\lambda^2)$
and $R_0^\pm(\lambda^2)$.

Recall that the free resolvent in four dimensions has the integral kernel
\begin{align}\label{resolv def}
	R_0^\pm(\lambda^2)(x,y)=\pm\frac{i}{4}\frac{\lambda}{2\pi |x-y|} H_1^\pm(\lambda|x-y|)
\end{align}
where $H_1^\pm$ are the Hankel functions of order one:
\begin{align}\label{H0}
H_1^{\pm}(z)=J_1(z)\pm iY_1(z).
\end{align}
From the series expansions for the Bessel functions, see \cite{AS}, as $z\to 0$ we have
\begin{align}
	J_1(z)&=\frac{1}{2}z-\frac{1}{16}z^3
	+\widetilde O_2(z^5),\label{J0 def}\\
	Y_1(z)&=-\frac{2}{\pi z}+\frac{2}{\pi}\log (z/2)J_1(z)+b_1z+b_2 z^3 +\widetilde O_2(z^5)
	\label{Y0 def}\\
         &= -\frac{2}{\pi z}+\frac{1}{\pi}z\log (z/2)
	+b_1z-\frac{1}{8\pi}z^3\log (z/2)+b_2 z^3
	+\widetilde O_2(z^5\log z).\label{Y0 def2}
\end{align}
Here $b_1,b_2\in \R$.
Further, for $|z|\gtrsim 1 $, we have the representation (see, {\em e.g.}, \cite{AS})
\begin{align}\label{JYasymp2}
	&H_1^\pm(z)= e^{\pm iz} \omega_\pm(z),\,\,\,\, |\omega_{\pm}^{(\ell)}(z)|\lesssim (1+|z|)^{-\frac{1}{2}-\ell},\,\,\,\ell=0,1,2,\ldots.
\end{align}
This implies, among the various expansions we develop, that  
(with $r=|x-y|$)
\begin{align}\label{R0repr}
	R_0^{\pm}(\lambda^2)(x,y)=r^{-2}\rho_-(\lambda r)+
	r^{-1}\lambda e^{\pm i\lambda r}\rho_+(\lambda r).
\end{align}
Here $\rho_-$ is supported on $[0,\frac{1}{2}]$, 
$\rho_+$ is supported on $[\frac{1}{4},\infty)$ 
satisfying the estimates 
$|\rho_-(z)|\les 1$  and
$\rho_+(z)=\widetilde O(z^{-\frac{1}{2}})$.

To obtain expansions for $R_V^\pm(\lambda^2)$ around zero energy  we utilize the symmetric resolvent identity.
Define $U(x)=1$ if $V(x)\geq 0$ and $U(x)=-1$ if $V(x)<0$, and let $v=|V|^{1/2}$, $w=Uv$ so that $V=Uv^2=wv$. 
Then the formula 
\be\label{res_exp}
R_V^\pm(\lambda^2)=  R_0^\pm(\lambda^2)-R_0^\pm(\lambda^2)vM^\pm(\lambda)^{-1}vR_0^\pm(\lambda^2),
\ee
is valid for $\Im (  \lambda) > 0$, where $M^\pm(\lambda)=U+vR_0^\pm(\lambda^2)v$.  

Note that the statements of Theorem~\ref{thm:main} control operators from $L^1(\R^4)$ to $L^\infty(\R^4)$,
while our analysis of $M^\pm(\lambda^2)$ and its inverse will be conducted in $L^2(\R^4)$.
The free resolvents are not locally $L^2(\R^n)$ when
$n>3$.  This requires us to iterate the standard resolvent
identities, as we show that iterated resolvents have better
local integrability. 
To use the 
symmetric resolvent identity, we need two resolvents on
either side of $M^{\pm}(\lambda)^{-1}$.
Accordingly, from the standard
resolvent identity we have:
\begin{align}
	R_V^\pm(\lambda^2)=&R_0^\pm(\lambda^2)\label{resolvid1}
	-R_0^\pm(\lambda^2)VR_0^\pm(\lambda^2) 
	+R_0^\pm(\lambda^2)V R_V^\pm(\lambda^2)
	VR_0^\pm(\lambda^2).
\end{align}
Combining this with \eqref{res_exp}, we have 
\begin{align}
	R_V^{\pm}(\lambda^2)=&R_0^{\pm}(\lambda^2)\label{bs finite}
	-R_0^{\pm}(\lambda^2)VR_0^{\pm}(\lambda^2)
	+R_0^{\pm}(\lambda^2)VR_0^\pm(\lambda^2)VR_0^{\pm}(\lambda^2)\\
	&-R_0^{\pm}(\lambda^2)VR_0^\pm(\lambda^2)vM^\pm(\lambda)^{-1}vR_0^\pm(\lambda^2)VR_0^{\pm}(\lambda^2).\label{bs tail}
\end{align}
Provided $V(x)$ decays sufficiently,  $[R_0^{\pm}(\lambda^2)VR_0^\pm(\lambda^2)v](x,\cdot)\in L^2(\R^4)$ uniformly in $x$, and that $M^\pm(\lambda)$ is invertible in $L^2(\R^4)$.  We recall the following lemma
from \cite{EGG},

\begin{lemma}\label{lem:locL2}

	If $|V(x)|\les \la x\ra^{-\beta-}$ for
	some $\beta>2$, then for any 
	$\sigma>\max(\frac{1}{2},3-\beta)$ we have
	$$
		\sup_{x\in\R^4}
		\| [ R_0^{\pm}(\lambda^2)V R_0^{\pm}(\lambda^2)]
		(x,y)\|_{L^{2,-\sigma}_y}\les \la \lambda
		\ra.
	$$
	Consequently $\lnorm R_0^\pm(\lambda^2)VR_0^\pm(\lambda^2)v\rnorm_{L^2\to L^\infty} \les \la \lambda \ra$.
\end{lemma}

To invert  $M^{\pm}(\lambda)$ in $L^2$ under various spectral assumptions on the zero energy we need  several different expansions for $M^{\pm}(\lambda)$. The  following operators arise naturally in these expansions 
(see \eqref{J0 def}, \eqref{Y0 def}):
\begin{align}
	G_0f(x)&=-\frac{1}{4\pi^2}\int_{\R^4} 
	\frac{f(y)}{|x-y|^2}\,dy=(-\Delta)^{-1} f(x) , \label{G0 def}\\
	G_1f(x)&=-\frac{1}{8\pi^2}\int_{\R^4} \log(|x-y|) f(y)\, dy,\label{G1 def}\\
	\label{Gj def}
	G_j f(x) &= \begin{cases} 
     c_j \int_{\R^4} |x-y|^j f(y)
	\, dy & \textrm{for } j=2,4,\dots \\
c_j \int_{\R^4}|x-y|^{j-1}
	\log(|x-y|)f(y) \, dy & \textrm{for } j=3,5,\dots \\
   \end{cases}	
\end{align}

Here $c_j$ are certain real-valued constants, the
exact values are unimportant for our analysis.  We will use $G_j(x,y)$ to denote the integral kernel of the operator $G_j$.  In 
addition, the following functions appear  naturally,
\begin{align}
\label{gi def}
g^+_j(\lambda)=\overline{g_j^-(\lambda)}&=\lambda^{2j}(a_j\log(\lambda)+z_j), \qquad j=1,2,3,\dots
    \end{align}
Here  $a_j\in \R\setminus\{0\}$ and $z_j\in \mathbb C
\setminus \R$.  In our expansions, we move any imaginary
parts to these functions of the spectral variable so that
the operators all have real-valued kernels.
To gain a more detailed low energy expansion 
than in \cite{EGG}, we go further into the expansions of
the resolvents, $M^\pm(\lambda)$ and $M^{\pm}(\lambda)^{-1}$ respectively.

 We also define the operators
 \begin{align}
\label{TandP} T:= M^\pm(0) = U+vG_0v,\,\,\,\,\,\,\,\,\,P:= \|V\|_1^{-1} v\la v, \cdot\ra.
\end{align} 
Recall the definition of the Hilbert-Schmidt
norm of an operator $K$ with kernel $K(x,y)$,
$$
	\| K\|_{HS}:=\bigg(\iint_{\R^{2n}}
	|K(x,y)|^2\, dx\, dy
	\bigg)^{\f12}
$$

\begin{lemma}\label{lem:M_exp}

	Assuming that $v(x)\les
	\la x\ra^{-\beta}$. If $\beta>2$, then we have 
	\begin{align}\label{Mexp0}
	 M^{\pm}(\lambda)=T+M_0^\pm(\lambda), \qquad \qquad
	 \sum_{j=0}^1\|\sup_{0<\lambda<\lambda_1} \lambda^{ -2+j+}  \partial_\lambda^j 
	 M_0^{\pm}(\lambda)\|_{HS}\les 1,
	\end{align}
	and
	\begin{multline}\label{Mexp1}
		M^{\pm}(\lambda) =T+\|V\|_1 g_1^{\pm}(\lambda) P
		+\lambda^2 vG_1v+M_1^{\pm}(\lambda),\\
		 \sum_{j=0}^2\|\sup_{0<\lambda<\lambda_1} \lambda^{ -2+j-}  \partial_\lambda^j 
		 	 M_1^{\pm}(\lambda)\|_{HS}\les 1. 
	\end{multline}
	If $\beta> 4 $, we have
	\begin{multline}\label{Mexp2}
		M^{\pm}(\lambda) =T+\|V\|_1 g_1^{\pm}(\lambda) P
		+\lambda^2 vG_1v+g_2^\pm(\lambda) vG_2v
		+ \lambda^4 vG_3v
		+M_2^{\pm}(\lambda),\\
		\sum_{j=0}^2\|\sup_{0<\lambda<\lambda_1} 
		\lambda^{ -4+j-}  \partial_\lambda^j 
		M_2^{\pm}(\lambda)\|_{HS}\les 1. 
	\end{multline}
	If $\beta> 6 $, we have
	\begin{multline}\label{Mexp3}
		M^{\pm}(\lambda) =T+\|V\|_1 g_1^{\pm}(\lambda) P
		+\lambda^2 vG_1v+g_2^\pm(\lambda) vG_2v
		+ \lambda^4 vG_3v\\+g_3^\pm(\lambda) vG_4v+
		\lambda^6 vG_5v
		+M_3^{\pm}(\lambda), \qquad\qquad
		\sum_{j=0}^2\|\sup_{0<\lambda<\lambda_1} \lambda^{ -6+j-}  \partial_\lambda^j 
		M_3^{\pm}(\lambda)\|_{HS}\les 1. 
	\end{multline}
\end{lemma}

\begin{proof}

Using the notation introduced in \eqref{G0 def}--\eqref{gi def} in \eqref{resolv def}, \eqref{J0 def}, and \eqref{Y0 def}, we obtain (for $ \lambda|x-y|\ll 1$)
\begin{align}\label{resolv expansion1}
	R_0^{\pm}(\lambda^2)(x,y) &=G_0(x,y)+ 
	\widetilde O_2(\lambda^{2}(1+\log (\lambda|x-y|)))\\ \label{resolv expansion2}
	R_0^{\pm}(\lambda^2)(x,y)  &=G_0(x,y)+g_1^{\pm}(\lambda)
	 +\lambda^2 G_1(x,y)+  
	\widetilde O_1(\lambda^{4}|x-y|^2 \log(\lambda|x-y|)).\\ \label{resolv expansion3}
	R_0^{\pm}(\lambda^2)(x,y)  &=G_0(x,y)+g_1^{\pm}(\lambda)
	 +\lambda^2 G_1(x,y)+g_2^\pm(\lambda)G_2(x,y)+\lambda^4G_3(x,y)\\
&\qquad \qquad \qquad	 + 
	\widetilde O_2(\lambda^{6}|x-y|^4 \log(\lambda|x-y|)). \nn\\
 \label{resolv expansion4}
	R_0^{\pm}(\lambda^2)(x,y)  &=G_0(x,y)+g_1^{\pm}(\lambda)
	 +\lambda^2 G_1(x,y)+g_2^\pm(\lambda)G_2(x,y)+\lambda^4G_3(x,y)\\
&\qquad \qquad + g_3^{\pm}(\lambda)G_4(x,y)
	+\lambda^6 G_5(x,y)+
	\widetilde O_2(\lambda^{8}|x-y|^6 \log(\lambda|x-y|)). \nn
\end{align} 
The $\widetilde O_2$ notation refers to derivatives with
respect to $\lambda$ in all cases.

In light of these expansions and using the notation in \eqref{TandP}, we define $M_j^\pm(\lambda)$ by the identities
\begin{align}\label{M0 exp}
M^{\pm}(\lambda) & =U+vR_0^{\pm}(\lambda^2)v =T+  M_0^{\pm}(\lambda).\\\label{M1 exp}
	M_0^{\pm}(\lambda) &= \|V\|_1g_1^{\pm}(\lambda)P
	+\lambda^2 vG_1v  + M_1^{\pm}(\lambda).\\ \label{M2 exp}
M_1^{\pm}(\lambda) &=g_2^\pm(\lambda)vG_2v+\lambda^4vG_3v + M_2^{\pm}(\lambda).\\ \label{M3 exp}
	M_2^{\pm}(\lambda) &=g_3^\pm (\lambda) vG_4v+\lambda^6 vG_5v
	+ M_3^{\pm}(\lambda).
\end{align} 
When $\lambda |x-y|\ll 1$, the bounds follow  from \eqref{resolv expansion1}, \eqref{resolv expansion2}, \eqref{resolv expansion3},
	and \eqref{resolv expansion4}.  	
	On the other hand, when $\lambda|x-y|\gtrsim 1$, we use \eqref{resolv def} and
	\eqref{JYasymp2} to see (for any $\alpha\geq 0$ and $k=0,1,2$)
	\begin{align}\label{R0high}
		|\partial_\lambda^k R_0(\lambda^2)(x,y)|
		=\bigg|\partial_\lambda^k \bigg[\frac{\lambda e^{i\lambda|x-y|}\omega(\lambda|x-y|)}{|x-y|}  \bigg]
		\bigg| 	\les 
		(\lambda |x-y|)^{\frac{1}{2}+\alpha}|x-y|^{k-2}.
	\end{align}
	Using \eqref{resolv expansion1}, \eqref{M0 exp}, and \eqref{R0high}, and choosing $\alpha=\frac32-k$, we see that
	\begin{align*}
	M_0(\lambda)(x,y)&=\left\{ \begin{array}{lc} 
	v(x)v(y) \log|x-y|\widetilde O_1(\lambda^{2-}), & \lambda|x-y|\ll 1\\
	v(x)v(y) [G_0(x,y)+  \widetilde O_1(\lambda^{2})], &\lambda |x-y|\gtrsim 1
	\end{array}\right.\\
	& = v(x)v(y)(1+\log|x-y|)\widetilde O_1(\lambda^{2-} ).
	\end{align*}
This yields the bounds in \eqref{Mexp0} as $v(x)\les \la x\ra^{-2-}$.

	The other assertions of the lemma follow from 
	similar arguments. Taking $\alpha=\frac32-k+$ in \eqref{R0high} and using
	$$ \widetilde O_2(\lambda^{4}|x-y|^2 \log(\lambda|x-y|)) = \widetilde O_2(\lambda^{2 } (\lambda |x-y|)^{0+}  ),\,\,\,\,\,\,\text{ for } \lambda|x-y|\ll 1
	$$
	we obtain \eqref{Mexp1}, whereas taking $\alpha=\frac72-k-$ in \eqref{R0high} and using
	$$ \widetilde O_2(\lambda^{6}|x-y|^4 \log(\lambda|x-y|)) = \widetilde O_2(\lambda^{2 } (\lambda |x-y|)^{2+}  ),\,\,\,\,\,\,\text{ for } \lambda|x-y|\ll 1
	$$
	we obtain \eqref{Mexp2}.   Finally, an operator with integral kernel $v(x)|x-y|^{\gamma}v(y)$ or $v(x)|x-y|^{\gamma}\log |x-y|v(y)$, $\gamma > 0$,  is Hilbert-Schmidt provided $\beta >2+\gamma$.

\end{proof}

The following corollary
is useful.

\begin{corollary}\label{iteratedBSexp}

	We have the expansion
	\begin{align}
		R_0^\pm(\lambda^2)V R_0^\pm(\lambda^2)(x,y)&=
		K_0
		+\widetilde E^\pm_0(\lambda)(x,y),\nn
	\end{align}
	here the operators $K_j$ have real-valued kernels.
	Furthermore, the error term $\widetilde{E}^\pm_0(\lambda)$ satisfies
	$$
		\widetilde E^\pm_0(\lambda)(x,y)=
		(1+\log^-|x-\cdot|+\log^-|\cdot-y|)
		\widetilde O_{1}(\lambda^{2-}).
	$$
	Furthermore, if one wishes to have $2$ derivatives, the extended expansion
	$$
		\widetilde E^\pm_0(\lambda)(x,y)
		=g_1^\pm(\lambda) K_{1}
				+\lambda^{n-2}K_{2}
		+\widetilde E^\pm_1(\lambda)(x,y),
	$$
	satisfies the bound
	$$
		\widetilde E^\pm_1(\lambda)(x,y)=
		\la x\ra^{\f12} \la y \ra^{\f12}
		\widetilde O_{2}(\lambda^{\f52}).
	$$

\end{corollary}

\begin{proof}

	This follows from the expansions for 
	$R_0^\pm(\lambda^2)$ in Lemma~\ref{lem:M_exp}.  
	For the iterated resolvents, the desired bounds come
	from simply multiplying out the terms.  In particular,
	$$
		K_0= G_0 V G_0, \qquad
		K_1=1 VG_0+G_0V1
	$$
	where $1$ is the operator with integral kernel a
	scalar multiple of $1(x,y)=1$,
	and
	$$
		K_2=G_0 V G_1+ G_1VG_0.
	$$
	These are all real-valued, 
	operators.

\end{proof}

\begin{rmk} \label{rmk:33}
The spatially weighted bound $|\partial_\lambda^{2}\widetilde{E}_1^\pm(\lambda)(x,y)|
\les \la x \ra^{\f 12}\lambda^{\frac{1}{2}}$ is only needed both derivatives act on the
leading resolvent, $R_0^\pm(\lambda^2)(x,z_1)$, in the product.   Similarly, the upper bound $\la y \ra^{\f 12}\lambda^{\frac{1}{2}}$ is
only needed if all derivatives act on the lagging resolvent, $R_0^\pm(\lambda^2)(z_1,y)$, in the product.  All other expressions that
arise would be consistent with $\widetilde{E}_1^\pm(\lambda)$ belonging to the class 
$\widetilde{O}_{2}(\lambda^{2-})$.  A very similar bound
holds for $R_0^\pm(\lambda^2)$ rather than iterated
resolvents.
\end{rmk}

One can see
that the invertibility of $M^{\pm}(\lambda)$ as an operator on $L^2$ for small
$\lambda$ depends upon the
invertibility of the operator $T$  on $L^2$, see  \eqref{TandP}.   
We now recall the definition of resonances at zero energy, following \cite{JN,EGG}.  This definition and subsequent
discussion also appear in \cite{EGG}.

\begin{defin}\label{resondef}\begin{enumerate}
\item We say zero is a regular point of the spectrum
of $H = -\Delta+ V$ provided $T=U+vG_0v$ is invertible on $ L^2(\mathbb R^4)$.

\item Assume that zero is not a regular point of the spectrum. Let $S_1$ be the Riesz projection
onto the kernel of $ T $ as an operator on $ L^2(\mathbb R^4)$.
Then $ T +S_1$ is invertible on $ L^2(\mathbb R^4)$.  Accordingly, we define $D_0=( T +S_1)^{-1}$ as an operator
on $ L^2(\R^4)$.
We say there is a resonance of the first kind at zero if the operator $T_1:= S_1 P S_1$ is invertible on
$S_1L^2(\mathbb R^4)$.

\item  Assume that $T_1$ is not invertible on
$S_1L^2(\mathbb R^4)$. Let $S_2$ be the Riesz projection onto the kernel of $T_1$ as an operator on $S_1L^2(\R^4)$.   Then $T_1+S_2$ is invertible on
$S_1 L^2(\R^4)$.
We say there is a resonance of the second kind at zero if $S_2=S_1$.
If $S_1\neq S_2$,  we say there is a resonance of the third kind. 
\end{enumerate}
\end{defin}

\noindent
{\bf Remarks.  i)} To relate the projections $S_j$ to the type of obstruction, we use the characterization proven in
\cite{EGG}.
In particular, $S_1-S_2\neq 0$ corresponds to the existence of a resonance
at zero energy, and $S_2\neq 0$ corresponds to the existence of an
eigenvalue at zero energy.  A resonance of
the first kind indicates that there is a resonance at zero
only, for a resonance of the second kind there is
an eigenvalue at zero only, and a resonance of the third
kind means there is both a resonance and an 
eigenvalue at zero energy.  For technical reasons, we 
need to employ different tools to invert 
$M^{\pm}(\lambda)$ for the different types of resonances.
It is well-known that different types of resonances at
zero energy lead to different expansions for 
$M^{\pm}(\lambda)^{-1}$ in other dimensions, see
\cite{ES2,EG2,EGG}.
Accordingly, we will develop
different expansions for $M^{\pm}(\lambda)^{-1}$ in
the following sections.\\
{ \bf ii)} Noting that $ T $ is self-adjoint, $S_1$ is the orthogonal projection onto the kernel of $ T $, and we have
(with $D_0=( T +S_1)^{-1}$) $$S_1D_0=D_0S_1=S_1.$$
This statement also valid for $S_2$ and $(T_1+S_2)^{-1}$. \\
{\bf iii)} $S_1$ and $S_2$ are finite-rank projections in all cases.  This follows by the observation that $T$ is a compact perturbation of the invertible operator $U$, and
invoking the Fredholm alternative. 
 See Section~\ref{sec:spec} below for a
full characterization of the spectral subspaces of $L^2$ associated
to $H=-\Delta+V$.

\begin{defin}
	We say an operator $K:L^2(\R^4)\to L^2(\R^4)$ with kernel
	$K(\cdot,\cdot)$ is absolutely bounded if the operator with kernel
	$|K(\cdot,\cdot)|$ is bounded from $L^2(\R^4)$ to $L^2(\R^4)$.
\end{defin}
Note that  Hilbert-Schmidt and
finite rank operators are absolutely bounded.  In 
\cite{EGG}, it was proven that

\begin{lemma}\label{d0bounded}
 The operator $ D_0 $ is absolutely bounded on $L^2$.
\end{lemma}

To invert $M^\pm(\lambda)=U+vR_0^\pm(\lambda^2)v$  for small $\lambda$, we use Lemma 2.1  in \cite{JN}.
\begin{lemma}\label{JNlemma}
Let $A$ be a closed operator on a Hilbert space $\mathcal{H}$ and $S$ a projection. Suppose $A+S$ has a bounded
inverse. Then $A$ has a bounded inverse if and only if
$$
B:=S-S(A+S)^{-1}S
$$
has a bounded inverse in $S\mathcal{H}$, and in this case
$$
A^{-1}=(A+S)^{-1}+(A+S)^{-1}SB^{-1}S(A+S)^{-1}.
$$
\end{lemma}

We will apply this lemma with $A=M^\pm(\lambda)$ and $S=S_1$,
the orthogonal projection onto the kernel of
$T$. Thus, we need to show that $M^{\pm}(\lambda)+S_1$
has a bounded inverse in $L^2(\mathbb R^4)$ and
\begin{align}\label{B defn}
  B_{\pm}(\lambda) =S_1-S_1(M^\pm(\lambda)+S_1)^{-1}S_1
\end{align}
has a bounded inverse in $S_1L^2(\mathbb R^4)$.

The invertibility of the operator $B_\pm$ is established using
different techniques, which depend  on the type of resonance  at zero energy. 
We recall Lemma~2.9 in \cite{EGG}.

\begin{lemma}\label{M+S1inverse}

	Suppose that zero is not a regular point of the spectrum of  $H=-\Delta+V$, and let $S_1$ be the corresponding
	Riesz projection. Then for   sufficiently small $\lambda_1>0$, the operators
	$M^{\pm}(\lambda)+S_1$ are invertible for all $0<\lambda<\lambda_1$ as bounded operators on $L^2(\R^4)$.
	Further, one has (with $\widetilde g_1^\pm (\lambda)=\|V\|_1g_1^\pm(\lambda)$)
	\begin{align}\label{M plus S}
        (M^{\pm}(\lambda)+S_1)^{-1}&=
        D_0-\widetilde g_1(\lambda)D_0PD_0-\lambda^2D_0vG_1vD_0
        +\widetilde O_2(\lambda^{2+})\\\label{M plus Sv2}
        &=D_0 
        +\widetilde O_2(\lambda^{2-})
	\end{align}
	as an absolutely bounded operator on $L^2(\R^4)$ 
	provided $v(x)\lesssim \langle x\rangle^{-2-}$.

\end{lemma}

\begin{corollary}

	We have the following expansion(s) for $B^{\pm}(\lambda)$.  If $v(x)\lesssim \langle x\rangle^{-2-}$,
	\begin{align}\label{B1}
        B^{\pm}(\lambda)&=-\widetilde g_1(\lambda)S_1PS_1-\lambda^2S_1vG_1vS_1
        +\widetilde O_2(\lambda^{2+})\\\label{B2}
        &=\widetilde O_2(\lambda^{2-})
	\end{align}

\end{corollary}

The contribution of 
\eqref{bs finite}, the 
finite terms of the Born series, to the Stone formula \eqref{stone} is controlled by the following lemma.  The lemma, which is Proposition~3.1 in
\cite{GGeven}, was proven in great generality and applies
in our case by taking $n=4$. A similar proposition is
proven for odd $n$ in \cite{GGodd}.

\begin{lemma}\label{bsprop}

	If $|V(x)|\les \la x\ra^{-3-}$, then the following
	bound holds.
		$$
			\sup_{x,y\in \R^4}\bigg|
			\int_0^\infty e^{it\lambda^2} \lambda \chi(\lambda)\bigg[\sum_{k=0}^{2m+1}(-1)^k\big\{
			R_0^+(VR_0^+)^k
			-R_0^-(VR_0^-)^k\big\}\bigg](\lambda^2)(x,y)\, 
			d\lambda\bigg| \les |t|^{-2}.
		$$

\end{lemma}  

This allows us now to focus only on the more singular
part of the resolvent expansion in \eqref{resolvid1}.
As the expansion for \eqref{bs tail} depends on the
type of obstruction at zero energy, we control its
contribution to the Stone formula separately in the
next sections.

\begin{rmk}

	We note that if zero is regular, 
	if $|V(x)|\les \la x\ra^{-4-}$ the low energy dispersive bound 
	$$
		\|e^{itH}\chi(H) P_{ac}(H)\|_{L^1\to L^\infty}\les |t|^{-2}
	$$	
	holds.  This can be done by using the expansion for $M^{\pm}(\lambda)^{-1}$ obtained
	in Lemma~\ref{M+S1inverse}, specifically \eqref{M plus S} is valid when $S_1=0$.  The
	time decay is obtained by an analysis that mirrors that of the Born series.

\end{rmk}

\section{Resonance of the first kind}\label{sec:first}
In this section we develop the tools necessary to
prove the first claim of Theorem~\ref{thm:main} when
there is a resonance of the first
kind.
That is, there is a resonance but no eigenvalue at zero energy. Further, $S_1\neq 0$ and $S_2=0$, and $S_1$ is of rank one by 
Corollary~\ref{ranks}. 
  In particular,
we prove
\begin{theorem}\label{thm:res1}
	
	Suppose that $|V(x)|\les \la x\ra^{-4-}$.
	If there is a resonance of the first kind at zero,
	then for $t>2$, 
	$$
		e^{itH}\chi(H) P_{ac}(H)=\varphi(t) P_r+
		O(1/t)A_1+O((t\log t)^{-1})A_2+ O(t^{-1}(\log t)^{-2})A_3+O(t^{-1-})A_4
	$$
	with $P_r=G_0VG_0vS_1vG_0VG_0:L^1\to L^\infty$ a rank one
	operator, $A_1:L^1\to L^\infty$,
	$A_2:L^{1}_w\to L^{\infty}_{w^{-1}}$ finite rank operators, $A_3:L^1\to L^\infty$
	and $A_4:L^{1,\f12}\to L^{\infty,-\f12}$.

\end{theorem}

We note here that the operator $P_r$ can be viewed as a sort of projection onto the
canonical resonance function $\psi\in L^{2,0-}(\R^4)$, which is chosen from the one-dimensional resonance space
so that $\la v\psi,v\psi\ra=1$.  From the representation $P_r=G_0VG_0vS_1vG_0VG_0$, using
Lemma~\ref{EG:Lem} and the fact that $S_1$ is rank one, we can see that $P_r$ is a bounded
operator from $L^1$ to $L^\infty$.

 We recall
Lemma~3.2 in \cite{EGG}, 
 \begin{lemma}\label{log decay}
 
 	If $\mathcal E(\lambda)=\widetilde O_1((\lambda \log \lambda)^{-2})$, then 
 	$$
 		\bigg|\int_0^\infty e^{it\lambda^2}
 		\lambda \chi(\lambda)
 		\mathcal E(\lambda)\, d\lambda\bigg| 
 		\les \frac{1}{\log t}, \qquad t>2.
 	$$
 
 \end{lemma}

This lemma essentially defines the function $\varphi(t)$
in the statement of Theorem~\ref{thm:main}. 
To  invert $M^{\pm}(\lambda)$ using Lemma~\ref{JNlemma}, we need to compute $B_\pm(\lambda)$, we use Lemma~3.3 
in \cite{EGG}.

\begin{lemma}

	In the case of a resonance of the first kind at
	zero, under the hypotheses of Theorem~\ref{thm:res1}
	the operator
	$B_{\pm}(\lambda)$ is invertible for small $\lambda$ and
	\begin{align}\label{Binv def}
		B_{\pm}(\lambda)^{-1}= 
		f^{\pm}(\lambda)S_1,
	\end{align}
where
\begin{align}\label{f defn}
	f^{+}(\lambda)=\frac{1}{\lambda^2}
	\frac{1}{a\log\lambda+z +\widetilde O_2(\lambda^{0+})}=\overline{f^-(\lambda)} 
\end{align}
for some $a\in \R/\{0\}$ and $z\in \mathbb C / \R$.
\end{lemma}

In
particular we note that for $0<\lambda <\lambda_1$,
\begin{align}\label{f diff} 
	f^+(\lambda)-f^-(\lambda)
	 =\frac{1}{\lambda^2}\bigg(\frac{(a\log \lambda+z)
	-(a\log\lambda +\overline{z})+\widetilde O_1(\lambda^{0+})}
	{(a\log \lambda+z)(a\log\lambda +\overline{z})+\widetilde O_1(\lambda^{0+})}
	\bigg)=\widetilde{O}_1((\lambda \log \lambda)^{-2}).
\end{align} 
We are now ready to use Lemma~\ref{JNlemma} to obtain
an expansion for $M^\pm(\lambda)^{-1}$.  This 
expansion is longer than the corresponding
expansion in \cite{EGG}, which allows us to give a more detailed long-time expansion
for the evolution.  

\begin{prop}\label{pwave exp}

	If there is a resonance of the first
	kind at zero, then for small $\lambda$
	\begin{align*}
		M^{\pm}(\lambda)^{-1} = f^\pm(\lambda)S_1+K_0
		+f_1^\pm(\lambda)K_1
		+f_2^\pm(\lambda)K_2
		+\widetilde O_2\bigg(\frac{1}{(\log\lambda)^3}\bigg), 
	\end{align*}
where $K_j$ are $\lambda$ independent, finite rank, absolutely bounded operators, and
$$
	f_j^\pm (\lambda) =\widetilde O_2\bigg(\frac{1}{(\log\lambda)^j}\bigg), \qquad
	j=1,2.
$$
\end{prop}

\begin{proof}

	Using Lemma~\ref{JNlemma} and \eqref{Binv def}, we see
	that
	\begin{align*}
	M^{\pm}(\lambda)^{-1}& =(M^\pm(\lambda)+S_1)^{-1}
	+(M^\pm(\lambda)+S_1)^{-1}S_1B_{\pm}(\lambda)^{-1}S_1
	(M^\pm(\lambda)+S_1)^{-1}\\
	& =(M^\pm(\lambda)+S_1)^{-1}
	+f^\pm(\lambda) (M^\pm(\lambda)+S_1)^{-1} S_1
	(M^\pm(\lambda)+S_1)^{-1}.
	\end{align*}

	 The representation \eqref{M plus Sv2} in Lemma~\ref{M+S1inverse}   takes care of the first summand.  Using \eqref{M plus S}, and $S_1D_0=D_0S_1=S_1$, we have
	\begin{align*}
		(M^\pm(\lambda)+S_1)^{-1}S_1
		&=S_1
		-\widetilde g_1^{\pm}(\lambda)D_0PS_1
		-\lambda^2 D_0vG_1vS_1
		+\widetilde O_2(\lambda^{2+}),\\
		S_1(M^\pm(\lambda)+S_1)^{-1}
		&=S_1
		-\widetilde g_1^{\pm}(\lambda)S_1PD_0
		-\lambda^2 S_1vG_1vD_0
		+\widetilde O_2(\lambda^{2+}).
	\end{align*}
	When an error term of size $\widetilde{O}_2(\lambda^{2+})$ interacts with 
$f^\pm(\lambda)$, the product satisfies $\widetilde O_2(\lambda^{2+})f^\pm(\lambda) = \widetilde O_2(\lambda^{0+})$, 
which satisfies the desired error bound for small $\lambda$.
	For the remaining terms
	$$-\widetilde g_1^{\pm}(\lambda)f^{\pm}(\lambda)
		[D_0PS_1+S_1PD_0]
		-\lambda^2 f^{\pm}(\lambda)[D_0vG_1vS_1+
		S_1vG_1vD_0]$$
For small $\lambda$, by a Taylor expansion,
\begin{align}
	\widetilde g_1^{\pm}(\lambda) f^{\pm}(\lambda)\nn
	&=\frac{\widetilde g_1^{\pm}(\lambda)}{c_1\widetilde g_1^{\pm}(\lambda)+c_2\lambda^2+\widetilde O_2
	(\lambda^{2+})}=\bigg(\frac{1}{c_1+\frac{c_2}{a_2\log \lambda+z_2^\pm}+\widetilde O_2(\lambda^{0+})}\bigg)\\
	&=a_0+\frac{a_1}{\log \lambda +z_2^\pm}\label{fgexp}
	+\frac{a_2}{(\log \lambda +z_2^\pm)^2}
	+ \widetilde O_2((\log\lambda)^{-3}), \text{ and}\\
	\lambda^2 f^{\pm}(\lambda)&=\frac{1}{a\log \lambda
	+z^{\pm}+\widetilde O_2(\lambda^{0+})}=
	\frac{b_1}{\log \lambda +z_2^\pm}
	+\frac{b_2}{(\log \lambda +z_2^\pm)^2}
	+ \widetilde O_2((\log\lambda)^{-3}).\nn
\end{align}
Where $a_j, b_j$ are real-valued constants, the exact values are unimportant for our
analysis.
Since $S_1$ is a rank one projection, the operators 
$K_j$ are all finite rank, and hence also absolutely bounded.

\end{proof}

To attain better time-decay, we use 

\begin{lemma}\label{log decay2}

	If $\mathcal E(\lambda)=\widetilde O_2((\log \lambda)^{-k})$, then  for $k\geq 2$,
	$$
		\bigg|\int_0^\infty e^{it\lambda^2}
		\lambda \chi(\lambda)
		\mathcal E(\lambda)\, d\lambda\bigg| 
		\les \frac{1}{t(\log t)^{k-1}}, \qquad t>2.
	$$

\end{lemma}

\begin{proof}

	We first divide the integral into two pieces,
	\begin{align*}
		\int_0^\infty e^{it\lambda^2}
		\lambda \chi(\lambda)
		\mathcal E(\lambda)\, d\lambda
		=\int_0^{t^{-1/2}} e^{it\lambda^2}
		\lambda \chi(\lambda)
		\mathcal E(\lambda)\, d\lambda
		+\int_{t^{-1/2}}^\infty e^{it\lambda^2}
		\lambda \chi(\lambda)
		\mathcal E(\lambda)\, d\lambda
	\end{align*}
	For the first integral one cannot utilize the
	oscillation of the Gaussian, instead we use
	\begin{align*}
		\bigg| \int_0^{t^{-1/2}} e^{it\lambda^2}
		\lambda \chi(\lambda)
		\mathcal E(\lambda)\, d\lambda\bigg|
		&\les \int_0^{t^{-1/2}} \frac{\lambda}{(\log \lambda)^k}\, d\lambda
		\les \int_0^{t^{-1/2}} \frac{\lambda^2}{\lambda (\log \lambda)^k}\, d\lambda\\
		&\les \frac{1}{t}\int_0^{t^{-1/2}} \frac{1}{\lambda (\log \lambda)^k}\, d\lambda
		\les \frac{1}{t(\log t)^{k-1}}.
	\end{align*}
	For the second integral, we utilize the oscillation
	by integrating by parts twice
	to see
	\begin{multline*}
		\bigg| \int_{t^{-1/2}}^\infty e^{it\lambda^2}
		\lambda \chi(\lambda)
		\mathcal E(\lambda)\, d\lambda\bigg|
		= \bigg|\frac{\mathcal E(t^{-1/2})}{2it}+\frac{1}{2it}
		\int_{t^{-1/2}}^\infty e^{it\lambda^2}
		\frac{d}{d\lambda}\big(\chi(\lambda)
		\mathcal E(\lambda)\big) \, d\lambda\bigg|\\
		\les \frac{|\mathcal E(t^{-1/2})|}{t}+
		\frac{|\mathcal E'(t^{-1/2})|}{t^{\f32}}
		+\frac{1}{t^2}\int_{t^{-1/2}}^\infty
		\bigg| \partial_\lambda
		\bigg( \frac{\mathcal E'(\lambda)}{\lambda}
		\bigg) \bigg|\, d\lambda
		\\
		\les \frac{1}{t(\log t)^k}+\frac{1}{t(\log t)^{k+1}}+\frac{1}{t^2}\int_{t^{-1/2}}^{t^{-1/4}} \frac{1}{\lambda^3 |\log t|^k}\, d\lambda+
		\frac{1}{t^2} \int_{t^{-\f14}}^{\f12} \frac{1}{\lambda^3}\, d\lambda \les
		\frac{1}{t(\log t)^k}.
	\end{multline*}
	Here we used that the integral converges on
	$[\f12,\infty)$.

\end{proof}

One can similarly prove bounds with $(t\log(\log t))^{-1}$ if
$k=1$ and $1/t$ if $k=0$.

\begin{lemma}

	In the case of a resonance of the first kind,
	dim$\{\psi \in L^{2,0+} \, :\, H\psi =0\}=1$.  Furthermore,
	the integral kernel of the operator $P_r:=G_0VG_0vS_1vG_0VG_0$ satisfies the
	identity
	$$
		P_r(x,y)=\psi(x)\psi(y), \qquad 
		\textrm{where } H\psi =0, \quad  \textrm{and }
		\quad 	\la v\psi, v\psi \ra=1.
	$$

\end{lemma}

\begin{proof}

	Corollary~\ref{ranks} and the fact that $S_2=0$ in this
	case establishes the first claim.  For the second,
	we first note that
	\begin{align}\label{S1trick}
		S_1(U+vG_0v)=0 \qquad \Rightarrow \qquad
		S_1=-S_1vG_0w, \qquad \textrm{ and }\qquad
		S_1=-wG_0vS_1.
	\end{align}
	So that $G_0VG_0vS_1vG_0VG_0=G_0vS_1vG_0$.  
	Now, since $S_1$ is a one dimensional projection, we
	have 
	$
		S_1 f(x)=\phi(x) \la f, \phi \ra 
	$
	where we take $\phi\in S_1L^2(\R^4)$ such that $\|\phi \|_2=1.$
	Furthermore, Lemma~\ref{lem:resonance} gives us 
	that $\phi=w\psi$ with $H\psi =0$.
	Noting that $(-\Delta+V)\psi=0$ is equivalent to
	$(I+G_0V)\psi=0$, we have
	\begin{align*}
		G_0vS_1vG_0 (x,y) &= G_0v(x,x_1)\phi(x_1) \phi(y_1)vG_0(y_1,y) 
		                                       = [G_0v \omega \psi ](x) [G_0v \omega \psi](y) \\
		                                 &= [G_0 V \psi ](x) [G_0 V \psi](y) = \psi(x) \psi(y)       
        \end{align*}

\end{proof}
We note that since $P_r$ is rank one, it is absolutely  bounded.
Noting that from
the expansion in Lemma~\ref{lem:M_exp} and its proof,
we have
\begin{align}\label{resolv wtd}
	R_0^\pm(\lambda^2)(x,y)=G_0(x,y)+g_1^\pm(\lambda)
	+\lambda^2 G_1(x,y)
	+\la x\ra^{\f12}\la y\ra^{\f12}
	\widetilde O_2(\lambda^{\f52}).
\end{align}
We are now ready to prove Theorem~\ref{thm:res1}.

\begin{proof}[Proof of Theorem~\ref{thm:res1}]

The proof follows from Lemma~\ref{bsprop} and the expansions in Proposition~\ref{pwave exp}.  
By \eqref{resolvid1}, and the discussion following 
Lemma~\ref{bsprop}, we need
now only bound the contribution of
\begin{align}
	[R_0^+VR_0^+vM^+(\lambda)vR_0^+VR_0^+
	-R_0^-VR_0^-vM^-(\lambda)vR_0^-VR_0^-](x,y)\label{mdiff1}
\end{align}
to the Stone formula, \eqref{stone}.  Using the algebraic fact,
\begin{align}\label{alg fact}
	\prod_{k=0}^MA_k^+-\prod_{k=0}^M A_k^-
	=\sum_{\ell=0}^M \bigg(\prod_{k=0}^{\ell-1}A_k^-\bigg)
	\big(A_\ell^+-A_\ell^-\big)\bigg(
	\prod_{k=\ell+1}^M A_k^+\bigg),
\end{align}
there are  two cases.  Either the `+/-'
difference acts on a free resolvent, or on 
$M^\pm(\lambda)^{-1}$.

By (\ref{resolv wtd}) we have
\begin{align}\label{R0diff}
	R_0^+(\lambda^2)(x,y)-R_0^-(\lambda^2)(x,y)=
	c\lambda^2+\la x\ra^{\f12}\la y\ra^{\f12}
	\widetilde O_2(\lambda^{\f52}).
\end{align}
When the `+/-' difference acts on a free resolvent,
we can write
$$
	M^\pm(\lambda)^{-1}=f^\pm(\lambda)S_1+\widetilde O_2(1).
$$
We get, in this case, contributions of the form
\begin{align}\label{C1 defn}
	\lambda^2 f^\pm(\lambda) C_1+\la x\ra^{\f12}
	\la y\ra^{\f12}\widetilde O_2(\lambda^{\f12-}),
\end{align}
here $C_1=1VG_0vS_1vG_0VG_0+G_0V1vS_1vG_0VG_0+G_0VG_0vS_1v1VG_0
+G_0VG_0vS_1vG_0V1$ is a finite rank operator.
By \eqref{fgexp} and a simple integration by parts, 
noting that $\partial_\lambda (\log \lambda)^{-1}$ is
integrable on $[0,\lambda_1]$,
we have that contribution
of $C_1$ to the Stone formula can be bounded by $t^{-1}$.
The error term's contribution to \eqref{stone} is of the
form
$$
	\la x\ra^{\f12}
		\la y\ra^{\f12}
	\int_0^\infty e^{it\lambda} \chi(\lambda)\widetilde{O}_2(\lambda^{\f32-})\, 
	d\lambda,
$$
which can be bounded by $\la x\ra^{\f12}\la y\ra^{\f12}t^{-1-}$ using
Lemma~\ref{lem:fauxIBP}.

On the other hand, if the '+/-' difference acts on 
$M^\pm(\lambda)^{-1}$, we have
\begin{align*}
	M^{+}(\lambda)^{-1} -M^{-}(\lambda)^{-1}&= [f^+-f^-](\lambda)S_1
	+[f_1^+-f_1^-](\lambda)K_1
	+[f_2^+-f_2^-](\lambda)K_2\\
	&+\widetilde O_2\bigg(\frac{1}{(\log\lambda)^3}\bigg), 
\end{align*}
There are now four subcases to consider.  
First, using the representation \eqref{resolv wtd},
if all the free resolvents contribute $G_0$, we have
to bound the contribution of
\begin{align*}
	[f^+-f^-](\lambda)P_r+\widetilde O_2(1/\log\lambda)C_2
	+\widetilde O_2\bigg(\frac{1}{(\log\lambda)^3}\bigg)
\end{align*}
Here $C_2=G_0VG_0vK_1vG_0VG_0+G_0VG_0vK_2vG_0VG_0$ is a
finite rank operator.  By Lemma~\ref{log decay}, the
first term's contribution to the Stone formula is bounded
by $1/\log t$.  The second term is bounded by $1/t$
and the final term is bounded by $t^{-1}(\log t)^{-2}$
by Lemma~\ref{log decay2}.  We note that the contribution
of this error term, when all of its surronding free resolvents contribute $G_0$ defines the operator
$A_3$ in Theorem~\ref{thm:res1}.  We can see from the
expansion for $M^{\pm}(\lambda)^{-1}$, that we cannot
expect this operator to be finite rank.

Another case to consider is when one free resolvent contributes $g_1(\lambda)$ while the other contribute
$G_0$.  Recall that $g_1(\lambda)$ comes with an operator whose integral kernel is a constant.  
In this case, we have to control
\begin{align*}
	g_1^\pm(\lambda)[f^+-f^-](\lambda) C_1+\widetilde{O}_2(\lambda^{2-}).
\end{align*}
Here $C_1$ the same operator encountered in \eqref{C1 defn}.
We note that $g_1^\pm(\lambda)[f^+-f^-](\lambda)=\widetilde O_1(1/\log \lambda)$, thus the first term's contribution to
the Stone formula is bounded by $1/t$ by Lemma~\ref{log decay}, while the contribution of the second term is
$O(t^{-1-})$ by Lemma~\ref{lem:fauxIBP}.

Another case to consider is when one free resolvent contributes $\lambda^2 G_1$ while the other contribute
$G_0$.  In this case, we have to control
\begin{align*}
	\lambda^2[f^+-f^-](\lambda) C_3+\widetilde{O}_2(\lambda^{2-}).
\end{align*}
Here $C_3=G_1VG_0vS_1vG_0VG_0+G_0VG_1vS_1vG_0VG_0+G_0VG_0vS_1vG_1VG_0
+G_0VG_0vS_1vG_0VG_1$ is a finite rank operator.
By Lemma~\ref{log decay2}, the first term's contribution 
to the Stone formula is bounded by $1/(t\log t)$.  
Due to the presence of the operator $G_1$, whose integral kernel is
$G_1(x,y)=-\frac{1}{8\pi^2} \log |x-y|$,
this bound
is understood as mapping logarithmically weighted spaces,
see the discussion around \eqref{log+ disc} below.

Finally, if the
error term in \eqref{resolv wtd} in any free resolvent
is encountered, or if less than three $G_0$'s are encountered,
its contribution is bounded by $\la x\ra^{\f12} \la y \ra^{\f12} \widetilde O_2(\lambda^{\frac{1}{2}})$, which
contributes $t^{-\f32}$ to the Stone formula as an operator
from $L^{1,\f12}$ to $L^{\infty,-\f12}$.

To close the proof, we must establish that the spatial
integrals converge.  For the contribution of
$C_1$, we note that due to the similarity of the four 
constitutent operators, and 
$$
	|1VG_0|(x,z_2)=C\int_{\R^4}\frac{V(z_1)}{|z_1-x|^2}
	\, dz_1 \les \int_{\R^4}\frac{\la z_1 \ra^{-4-}}{|z_1-x|^2}\, dz_1\les 1
$$
uniformly in $x$
by Lemma~\ref{EG:Lem}.  Similarly,
\begin{align*}
	|G_0VG_0|(z_3,y)&=C\int_{\R^4}\frac{V(z_4)}{|z_3-z_4|^2|z_4-y|^2}
	\, dz_4\\ 
	&\les \int_{\R^4}\frac{\la z_4 \ra^{-4-}}{|z_4-y|^2}\left(\frac{1}{|z_3-z_4|^{2+}}
	+\frac{1}{|z_3-z_4|^{2-}}
	\right)
	\, dz_4\les 1+\frac{1}{|z_3-y|^{0+}}.
\end{align*}
Under the assumptions on $V$, we have that
$\sup_{y\in \R^4} \|v(\cdot)(1+|y-\cdot|^{-0-})\|_2\les 1 $.  Thus,
\begin{align}
	\sup_{x,y}|1VG_0vS_1vG_0VG_0(x,y)|\les \sup_{x,y}
	\|1VG_0(x,\cdot)v\|_2 \||S_1|\|_{2\to 2} \|vG_0VG_0(\cdot, y)\|_2\les 1.
\end{align}
Similarly, one can bound the contributions of $P_r$ and
$C_2$.  For $C_3$ we must take some care to account
for the operator $G_1$.  Note that $G_1(x,z)=c \log|x-z|$,
when $G_1$ is contributed by the
leading or lagging free resolvent, we use that
$\log |x-z|=\log^-|x-z|+\log^+|x-z|$.  Since
$\log^{-}|x-z|\les |x-z|^{0-}$, we can control it as
in the previous operators.  
For $\log^+$, we note that
$\log^+$ is an increasing function and 
$|x-y|\leq |x|+|y|\leq 2\max(|x|,|y|)$.  Then, 
\begin{multline}\label{log+ disc}
	|G_1VG_0|(x,z_2)=C \int_{\R^4} \frac{\log |x-z_1| V(z_1)}{|z_1-z_2|^2}
	\, dz_1\\
	\les (1+\log^+ |x|) \int_{\R^4}
	\frac{V(z_1)}{|z_1-z_2|^2}\left(1+
	\log^+ |z_1|+\frac{1}{|z_1-x|^{0+}}\right)\, dz_1
	\les 1+\log^+|x|.
\end{multline}
Here we used the decay of the potential to control the
$\log^+|z_1|$ growth and Lemma~\ref{EG:Lem} to establish
the boundedness of the resulting integrals.  A similar
analysis holds for the polynomially weighted error terms.

\end{proof}

\section{Resonance of the second kind}\label{sec:second}

In this section we prove Theorem~\ref{thm:main} in the case of a resonance of the second 
kind, when $S_1\neq 0$, and $S_1-S_2=0$.  Recall that this means there is an eigenvalue at
zero energy, but no resonance. 
In particular,
we prove

\begin{theorem}\label{thm:res2}

	Suppose that $|V(x)|\les \la x\ra^{-8-}$.
	If there is a resonance of the second kind at zero,
	then 
	$$
		e^{itH}\chi(H) P_{ac}(H)=
		O(1/t) A_1+ O(t^{-1-})A_3 ,
		\qquad t>2.
	$$
	$A_1:L^1\to L^\infty$,
	is a finite rank operator,
	and $A_3:L^{1,\f12}\to L^{\infty,-\f12}$.

\end{theorem}

Despite the fact that the spectral measure is more
singular as $\lambda \to0$, the lack of resonances
greatly simplifies our expansions for $M^\pm(\lambda)^{-1}$.
Much of this simplification follows from the fact that
$S_1=S_2$, which by \eqref{PS2} shows that $PS_1=0$.  This eliminates 
many of the terms containing powers of $\log\lambda$ in the
expansion of the spectral measure as $\lambda \to 0$.
 
To understand the expansion for
$M^{\pm}(\lambda)^{-1}$ in this case we need  more terms in the expansion of $(M^{\pm}(\lambda)+S_1)^{-1}$ than was provided Lemma~\ref{M+S1inverse}.
From Lemma~\ref{lem:M_exp}, specifically 
\eqref{M2 exp}, we have by a Neumann series expansion
\begin{align}
	(M^{\pm}(\lambda)&+S_1)^{-1}\nn \\
	&=D_0[\mathbbm 1
	+\widetilde g_1^{\pm}(\lambda)PD_0+
	\lambda^2 vG_1vD_0+g_2^{\pm}(\lambda)vG_2vD_0+\lambda^4
	vG_3vD_0+M_2^{\pm}(\lambda)D_0]^{-1}\nn\\
	&=D_0-\widetilde g_1^{\pm}(\lambda) D_0PD_0-\lambda^2
	D_0vG_1vD_0+(\widetilde g_1^{\pm}(\lambda))^2 D_0PD_0PD_0
	\label{MS eval}\\
	&+\lambda^2 \widetilde g_1^{\pm}(\lambda)[D_0PD_0vG_1vD_0
	+D_0vG_1vD_0PD_0]-g_2^{\pm}(\lambda)D_0vG_2vD_0\nn\\
	&-\lambda^4	D_0vG_3vD_0+D_0E_2^{\pm}(\lambda)D_0\nn
\end{align}
with $E_2^{\pm}(\lambda)=\widetilde O_1(\lambda^{4+})$. 

In the case of a resonance of the second kind, we recall
that $S_1=S_2$.  
By  Lemma~\ref{vG1v kernel} below the operator $S_1vG_1vS_1$ is invertible on $S_1L^2$ (which is $S_2L^2$ in this case). We define  $D_2=(S_1vG_1vS_1)^{-1}$ as an operator on
$S_2L^2(\R^4)$.  Noting that $D_2=S_1D_2S_1$, the operator
is finite rank and hence absolutely bounded.

\begin{prop}\label{prop:Minv2}

	If there is a resonance of the second kind at zero,
	then for small $\lambda$
\begin{equation} \label{Minv2}
		M^{\pm}(\lambda)^{-1}=-\frac{D_2}{\lambda^2}
		+\frac{g_2^{\pm}(\lambda)}{\lambda^4}	K_1+K_2+\widetilde O_2(\lambda^{0+})
\end{equation}
	where $K_1, K_2$ are $\lambda$ independent, finite rank operators.

\end{prop}

We note that the statement and proof of this Proposition are found in 
\cite{EGG}, see Proposition~4.2 with an error term of
$\widetilde O_1(\lambda^{0+})$.  
Here, we need the extra derivative, but we note the
same proof follows noting that the error term from Lemma~\ref{M+S1inverse} has two
derivatives.  In \cite{EGG} only one derivative was needed to get the $t^{-1}$ bound,
we wish to gain more time decay from its contribution which necessitates the spatial weights.
The fact that $K_1,K_2$ are finite rank operators follows from the fact that $S_1$ is finite rank.

\begin{proof}[Proof of Theorem~\ref{thm:res2}]

The proof follows by bounding
the contribution of Proposition~\ref{prop:Minv2}
to the Stone formula, \eqref{stone}.  We to use cancellation between the `+' and `-'
terms in
$$
R_0^+VR_0^+vM^+(\lambda)^{-1}vR_0^+VR_0^+
- R_0^-VR_0^-vM^-(\lambda)^{-1}vR_0^-VR_0^-.
$$
As with resonances of the first kind, we use
the algebraic fact~\eqref{alg fact}.
Two kinds of terms occur in this decomposition, one featuring
the difference $M^+(\lambda)^{-1} - M^-(\lambda)^{-1}$ and
those containing a difference of free resolvents.
For the first case we use Proposition~\ref{prop:Minv2} and that   $g_2^+(\lambda) - g_2^-(\lambda) = c\lambda^4$
to obtain
\begin{align}\label{B diff r2}
	M^+(\lambda)^{-1} - M^-(\lambda)^{-1}
	=cK_1+\widetilde O_2(\lambda^{0+}).
\end{align}
Recalling \eqref{resolv wtd}, we can write
$R_0^\pm(\lambda^2)=G_0
	+\la x\ra^{\f12}\la y\ra^{\f12}
	\widetilde O_2(\lambda^{2-})$
and consider the most singular terms this difference
contributes, i.e.,
\begin{align*}
	G_0VG_0v K_1 vG_0VG_0 +\la x\ra^{\f12}\la y\ra^{\f12}	\widetilde O_2(\lambda^{0+}).
\end{align*}
The time decay of $t^{-1}$ for the first term follows from
Lemma~\ref{lem:IBP} 
and the bound of $\la x\ra^{\f12}\la y\ra^{\f12} t^{-1-}$
for the error term follows from \ref{lem:fauxIBP}.
An analysis of the spatial integrals as in
Theorem~\ref{thm:res1} noting
that $K_1$ is finite rank and hence absolutely bounded finishes the argument.
For the terms of the second kind the difference  of `+' and
`-' terms in \eqref{alg fact} acts on one of the resolvents.
As usual, the most delicate case is of the form
\begin{align}\label{eq:2ndworst}
	(R_0^+(\lambda^2)-R_0^-(\lambda^2))VR_0^+(\lambda^2)v 
	[\eqref{Minv2}]v R_0^+(\lambda^2)VR_0^+(\lambda^2).
\end{align}
Using \eqref{resolv wtd}, we have
$[R_0^+-R_0^-](\lambda^2)(x,y)=c\lambda^2+\la x\ra^{\f12}\la y\ra^{\f12}\widetilde O_2(\lambda^{\f52})$.
We then write $M^\pm(\lambda)=-D_2/\lambda^2+\widetilde O_2(\lambda^{0-})$ to see
$$
	\eqref{eq:2ndworst}=-cVG_0vD_2vG_0VG_0+
	\la x\ra^{\f12}\la y\ra^{\f12}\widetilde O_2(\lambda^{0+}).	
$$
Using Lemma~\ref{lem:IBP}, we see that the first term
contributes $t^{-1}$ to the Stone formula, while the
second term is bounded by $\la x\ra^{\f12}\la y\ra^{\f12} t^{-1-}$ using Lemma~\ref{lem:fauxIBP}.
The remaining terms 
can be bounded similarly. 
The spatial integrals are controlled as in the case of
a resonance of the first kind in 
Theorem~\ref{thm:res1}.

\end{proof}

\section{Resonance of the third kind}   \label{sec:third}
In this section we prove  Theorem~\ref{thm:main} in the case of a resonance of the third
kind, that is when $S_1\neq 0$, $S_2\neq 0$ and $S_1-S_2\neq 0$.  Recall that this means
there are both a zero energy resonance and a zero energy eigenvalue.
   In particular,
we prove
\begin{theorem}\label{thm:res3}
	
	Suppose that $|V(x)|\les \la x\ra^{-8-}$.
	If there is a resonance of the third kind at zero,
	then for $t>2$, 
	$$
		e^{itH}\chi(H) P_{ac}(H)=\varphi(t) A_0+
		O(1/t)A_1+O((t\log t)^{-1})A_2+ O(t^{-1}(\log t)^{-2})A_3+O(t^{-1-})A_4
	$$
	with $A_0,A_1:L^1\to L^\infty$,
	$A_2:L^{1}_w\to L^{\infty}_{w^{-1}}$ finite rank operators, $A_3:L^1\to L^\infty$
	and $A_4:L^{1,\f12}\to L^{\infty,-\f12}$.

\end{theorem}	
	
The expansion in \eqref{MS eval} remains valid, but in this
section we do not have that $S_1P=0$.  Using
\eqref{M2 exp} in Lemma~\ref{lem:M_exp},
we have
\begin{align}
	B^{\pm}(\lambda)&= \widetilde g_1^\pm(\lambda)
	S_1PS_1+\lambda^2 S_1vG_1vS_1 -(\widetilde g_1^\pm(\lambda))^2 S_1PD_0PS_1\nn\\ & \qquad-
	\lambda^2 \widetilde g_1^\pm(\lambda)[S_1PD_0vG_1vS_1
	+S_1vG_1vD_0PS_1] +g_2^\pm(\lambda) S_1vG_2vS_1\nn\\
	&\qquad +\lambda^4 S_1vG_3vS_1+ \widetilde O_2(\lambda^{6-})\label{B3 gammas}
	\\& =: \widetilde g_1^\pm(\lambda)
	S_1PS_1+\lambda^2 S_1vG_1vS_1 +(\widetilde g_1^\pm(\lambda))^2 \Gamma_1 +
	\lambda^2 \widetilde g_1^\pm(\lambda)\Gamma_2 +g_2^\pm(\lambda)\Gamma_3 \nn \\
	&\qquad +\lambda^4 \Gamma_4 + \widetilde O_2(\lambda^{6-}).\nn
\end{align}

 Note that, since $S_2\neq 0$ the kernel of $S_1PS_1$ is non-trivial.  We, therefore, use Feshbach
formula to invert $B^{\pm}(\lambda)$. To do that, we define the operator $\Gamma$ by $S_1=S_2+\Gamma$ and express 
 $B^{\pm}(\lambda)$ with respect
to the decomposition $S_1L^2(\R^4)=S_2L^2(\R^4)\oplus
\Gamma L^2(\R^4)$.
 We also define the finite rank operator $S$ by
\begin{align}\label{S defn}
		S:=\left[
		\begin{array}{cc}
			\Gamma & -\Gamma vG_1vD_2\\
			-D_2vG_1v\Gamma &
			D_2vG_1v\Gamma vG_1v	D_2
		\end{array}
		\right].
\end{align}

  We note that the operator $A_0$ in the statement of Theorem~\ref{thm:res3} has rank at most two. This follows
from the expansions detailed below and the fact that $S$ has rank at most two.  As in the 
previous cases, we give a refinement of the expansion in \cite{EGG}.

\begin{lemma}

In the case of a resonance of the third kind 
we have for small $\lambda$
\begin{align}\label{B-13}
	B^{\pm}(\lambda)^{-1}& =\tilde f^{\pm}(\lambda) S+\frac{D_2}{\lambda^2}+\frac{g_2^\pm(\lambda)}{\lambda^4} F_1+F_2+f_1(\lambda) F_3 + f_2(\lambda)F_4+ \widetilde O_2 ( \lambda^{0+}).
\end{align}
Here $F_j$ are $\lambda$ independent absolutely bounded operators, $\tilde f^+(\lambda)
=(\lambda^2(a\log \lambda +z))^{-1}$ with $a\in \R \setminus \{0\}$ and $z\in \mathbb C \setminus \R$,
with $\tilde f^-(\lambda)=\overline{\tilde f^+(\lambda)}$, 
and $f_j(\lambda)=\widetilde O_2((\log \lambda)^{-j})$.

\end{lemma}

\begin{proof}

Here we use the fact that $S_2P=PS_2=0$ to
see that the two smallest terms of $B^{\pm}(\lambda)$
with respect to $\lambda$, see \eqref{B3 gammas}, may be written in the block form
\begin{align}\label{B fesh}
	A^{\pm}(\lambda):=\lambda^2 \left[
	\begin{array}{cc}
		\frac{\widetilde g_1^{\pm}(\lambda)}
		{\lambda^2}\Gamma P\Gamma +   
		\Gamma vG_1v\Gamma &
		\Gamma vG_1vS_2 \\
		S_2vG_1v\Gamma &
		S_2vG_1vS_2
	\end{array}
	\right].
\end{align}


Then, by the Feshbach formula we have
\begin{align}
	A^{\pm}(\lambda)^{-1}
	&=\frac{1}{\lambda^2 h^{\pm}(\lambda)}
	\left[
	\begin{array}{cc}
		\Gamma & -\Gamma vG_1vD_2\\
		-D_2vG_1v\Gamma &
		D_2vG_1v\Gamma vG_1v	D_2
	\end{array}
	\right]
	+\frac{D_2}{\lambda^2}\\ \nn
	&=:\tilde f^{\pm}(\lambda) S+\frac{D_2}{\lambda^2}.
\end{align}
Here $\tilde{f}^{\pm}:= (\lambda^2 [a\log\lambda +z])^{-1} $ for some $a \in \R \setminus \{0\}$ and $z \in \mathbb{C}\setminus \R $. One can see \cite{EGG} for the details of this inversion.

By a Neumann series expansion, we obtain
\begin{align*}
	B^{\pm}(\lambda)^{-1}& =A^{\pm}(\lambda)^{-1}
	[\mathbbm 1+(B^{\pm}(\lambda)-A^{\pm}(\lambda))A^{\pm}
	(\lambda)^{-1}]^{-1}\\
	&=A^{\pm}(\lambda)^{-1}-A^{\pm}(\lambda)^{-1}[B^{\pm}(\lambda)-A^{\pm}(\lambda)] A^{\pm}(\lambda)^{-1} +\widetilde O_2(\lambda^{0+}).
\end{align*}
We note that $D_2S_1P=D_2S_2P=0$.  Therefore, the
operators $\Gamma_j$ in the expansion of $B^\pm(\lambda)$
in \eqref{B3 gammas} satisfy
$$
\Gamma_1 D_2=D_2\Gamma_1 =D_2\Gamma_2D_2 =0.
$$
Further recalling \eqref{gi def} for  $\widetilde{g_1}^\pm(\lambda)$ 
and $g_2^\pm(\lambda)$, for small
$\lambda$ we have 
\begin{equation} \label{fgproperties}
\begin{aligned}
\tilde f^\pm(\lambda)\widetilde{g_1}^\pm(\lambda)
& =\f{\widetilde{g_1}^\pm(\lambda)}{\lambda^2 (a\log \lambda +z)}= c_1+\f{z_1}{a\log \lambda +z}, \\
 \frac{\tilde f^\pm(\lambda)}{\lambda^2}  g_2^\pm(\lambda)&=\f{g_2^\pm(\lambda)}{\lambda^4(a\log \lambda +z)}
= c_2+\f{z_2}{a\log \lambda +z},\\
\widetilde f^\pm(\lambda) \lambda^2& =\f 1{a\log\lambda+z}  \\
 [\tilde f^\pm(\lambda)]^2 g_2^\pm(\lambda)&= \f{c_3}{a \log\lambda +z} +\f{z_3}{(a \log\lambda +z)^2} ,
\end{aligned}
\end{equation}
for some (unimportant) constants $c_j, z_j$,
establishes the claim.
\end{proof}

Using Lemma~\ref{JNlemma} and \eqref{B-13}, we have

\begin{prop}\label{pwave exp3}

	If there is a resonance of the third
	kind at zero, then for small $\lambda$
	\begin{align*}
		M^{\pm}(\lambda)^{-1} =  \tilde f^\pm(\lambda) S_1SS_1+\frac{D_2}{\lambda^2}
+\frac{g_2^\pm(\lambda)}{\lambda^4} D_2\Gamma_3D_2+ K_1+f_1(\lambda) K_2 + f_2(\lambda)K_3
+ \widetilde O_2 ( \lambda^{0+}),
	\end{align*}
where the operators are $\lambda$ independent and finite rank except for the error term.
\end{prop}

\begin{proof}[Proof of Theorem~\ref{thm:res3}]

We note that the expansion of  $M^{\pm}(\lambda)^{-1}$ is a sum of terms similar to the ones 
encountered in
Propositions~\ref{pwave exp} and \ref{prop:Minv2}.
Accordingly, the methods used in the proofs of Theorems~\ref{thm:res1} and
\ref{thm:res2} apply with only minor adjustments.

\end{proof}

\section{Eigenvalue only and $P_eVx=0$}\label{sec:t-2 decay}

We consider the evolution when there is a resonance of the third kind, that is an eigenvalue by not resonance at
zero energy, and extra
cancellation between the eigenfunctions and the potential. In particular, we show that
the evolution satisfies the same $|t|^{-2}$ dispersive bound as an operator from 
$L^1$ and $L^\infty$ as
the free evolution.  This bound is motivated by the work in \cite{GGodd,GGeven} which
proved such bounds in higher dimensions $n\geq 5$.  We 
note that the techniques of \cite{GGodd,GGeven} are not
sufficient to obtain the $|t|^{-2}$ bound when $n=4$.   In dimensions $n>4$, one has the
expansion $R_0^\pm(\lambda^2)=G_0+O(\lambda^2)$, in 
dimension $n=4$ we instead have $R_0^\pm(\lambda^2)=G_0+O(\lambda^2(1+\log (\lambda|x-y|)))$.  This small difference introduces
many technical challenges which we  overcome
in this section. 

For the purpose of obtaining this bound, one needs much longer expansion for $M^{\pm}(\lambda)^{-1}$.

\begin{lemma}\label{M+S1inverse long}

	Suppose that zero is not a regular point of the spectrum of  $H=-\Delta+V$, and let $S_1$ be the corresponding
	Riesz projection. Then for   sufficiently small $\lambda_1>0$, the operators
	$M^{\pm}(\lambda)+S_1$ are invertible for all $0<\lambda<\lambda_1$ as bounded operators on $L^2(\R^4)$.
	Further, one has 
	\begin{align}\nn
        (M^{\pm}(\lambda)+S_1)^{-1}&=
        D_0-\widetilde g_1^\pm(\lambda)D_0PD_0-\lambda^2D_0vG_1vD_0
        -g_2^\pm(\lambda)D_0vG_2vD_0\\
        &-\lambda^4 D_0vG_3vD_0
        -g_3^\pm (\lambda)D_0vG_4vD_0-\lambda^6 D_0vG_5vD_0\nn \\
        &+(\widetilde g_1^\pm(\lambda))^2
        D_0PD_0PD_0
        +\lambda^2 \widetilde g_1^\pm(\lambda)
        [D_0PD_0vG_1vD_0+D_0vG_1vD_0PD_0]\nn\\
        &+\lambda^4 \widetilde g_1^\pm(\lambda) [D_0PD_0vG_3vD_0+D_0vG_3vD_0PD_0]\nn\\
        &+\widetilde g_1^\pm(\lambda)g_2^\pm (\lambda)[D_0PD_0vG_2vD_0+D_0vG_2vD_0PD_0]
        +\lambda^4 D_0vG_1vD_0vG_1vD_0\nn\\
        &-(\widetilde g_1^\pm(\lambda))^3D_0PD_0PD_0PD_0
        +\lambda^2 g_2^\pm (\lambda)[D_0vG_1vD_0vG_2vD_0\label{M plus S long}\\
        &+    D_0vG_2vD_0vG_1vD_0]
        -\lambda^2 (\widetilde g_1^\pm(\lambda))^2[D_0PD_0PD_0vG_1vD_0\nn\\
        &+D_0PD_0vG_1vD_0PD_0+D_0vG_1vD_0PD_0PD_0]\nn\\
        &-\lambda^4 \widetilde g_1^\pm (\lambda)
        [D_0PD_0vG_1vD_0vG_1vD_0
        +D_0vG_1vD_0vG_1vD_0PD_0\nn\\
        &+D_0vG_1vD_0PD_0vG_1vD_0]
        -\lambda^6 D_0vG_1vD_0vG_1vD_0vG_1vD_0
        +\widetilde O_2(\lambda^{6+})\nn
	    \end{align}
	as an absolutely bounded operator on $L^2(\R^4)$ 
	provided $v(x)\lesssim \langle x\rangle^{-6-}$.

\end{lemma}

\begin{proof}
 The proof uses the expansion \eqref{M3 exp} for $M^{\pm}(\lambda)$ up to terms of size $\lambda ^{6}$
 with error term $M_3^\pm(\lambda)$,
 along with a Neumann series expansion that considers
 up to the `$x^3$' term. 
\end{proof}

\begin{rmk} \label{rem:cancel}

When there is an eigenvalue only, we take advantage of
the facts that $S_1=S_2$, $S_1D_0=D_0S_1=S_1$ and
$S_1P=PS_1=0$.  The effect of this is that the leading
terms in $S_1(M+S_1)^{-1}S_1$ containing
only $\widetilde g_1^\pm(\lambda)$, $(\widetilde g_1^\pm(\lambda))^2$, $(\widetilde g_1^\pm(\lambda))^3$,
$\lambda^2 \widetilde g_1^\pm(\lambda)$, 
$\widetilde g_1^\pm(\lambda) g_2^\pm(\lambda)$ and the
$\lambda^4 \widetilde g_1^\pm(\lambda) [D_0PD_0vG_3vD_0+D_0vG_3vD_0PD_0]$ all vanish. 

\end{rmk}


This observation allows us to prove the following.
	     
\begin{lemma} Suppose there is a resonance of the second kind at zero and $P_eVx=0$. If $v(x)\les \la x \ra ^{-6} $ then we have the following expansion. 
 \begin{align}\label{eq:Bin cancel} 
  B^{\pm}(\lambda) ^{-1} = {\f {D_2}{\lambda^2} } +B_1+ {\f{g_3^{\pm}(\lambda)}{\lambda^4}} B_2 +{\f{g_2^{\pm}(\lambda)}{\lambda^2}} B_3+ \widetilde{g}_1^{\pm}( \lambda) B_4 +\lambda^2 B_5+  \widetilde O_2(\lambda^{2+})  
      \end{align}
  where $B_i$ are absolutely bounded operators with real-valued kernels.
     \end{lemma}
  \begin{proof}
  Note that  by the identities $S_1D_0=D_0S_1=S_1$ and
$S_1P=PS_1=0$ in Remark~\ref{rem:cancel} many terms in \eqref{M plus S long} cancels and recalling that by
Lemma~\ref{vG1v kernel} the operator
$D_2=(S_1vG_1vS_1)^{-1}$ is bounded when $S_1=S_2$, we obtain
  \begin{align*} 
        B^{\pm}(\lambda)^{-1}& = [
        -\lambda^2S_1vG_1vS_1
        -g_2^\pm(\lambda)S_1vG_2vS_1 \\
         & \hspace{10mm}+ \lambda^4 C_1+g_3^\pm (\lambda)C_2+\lambda^2 g_2^{\pm} (\lambda) C_3 + \lambda^4 \widetilde{g}_1^{\pm}(\lambda) C_4 + \lambda^6 C_5 +\widetilde O_2(\lambda^{6+})]^{-1} \\  
         & = -\lambda^{-2} D_2  [ \mathbbm 1 + \lambda^{-2} g_2^\pm(\lambda)S_1vG_2vS_1D_2+ \lambda^2 C_1D_2\\
         & \hspace{10mm} +\lambda^{-2} g_3^\pm (\lambda)C_2D_2+ g_2^{\pm} (\lambda) C_3D_2 + \lambda^2 \widetilde{g}_1^{\pm}(\lambda) C_4D_2 + \lambda^4 C_5D_2 +\widetilde O_2(\lambda^{4+}) ]^{-1}   .
	\end{align*}
Here the operators $C_i$'s can be written explicitly, however, in our analysis it is enough to know that the decay assumption on $v(x)$ ensures the boundedness of their Hilbert-Schmidt norms.
 
To effectively invert the above expression in a 
Neumann series, we first recall  $D_2= S_1D_2S_1$ and $S_1=-wG_0vS_1$. Also, by Lemma~\ref{lem:eigenspace} we have $ P_e=G_0vS_2D_2S_2vG_0$. 

Using these and noting that in this section $S_1=S_2$ , we have 
 \begin{align} \label{eq D1}  D_2=S_1D_2S_1= wG_0vS_1D_2S_1vG_0w = w P_e w.
  \end{align} 
Further assuming $P_eVx=0$, recalling that
$G_2(x,y)=|x-y|^2=(x-y)\cdot(x-y)$, we have
   \begin{multline} \label{P_eVx cancel}
      D_2vG_2vD_2= w P_eV[x^2 -2x\cdot y + y^2] V P_e w \\ = wP_eV x^2
1V P_ew - 2wP_eV x \cdot yV P_ew + wP_eV 1y^2V P_ew =0 
   \end{multline} 
Note that using \eqref{P_eVx cancel} the term $\lambda^{-2} g_2^\pm(\lambda)D_1S_1vG_2vS_1$ is zero. Hence, we obtain  \eqref{eq:Bin cancel}. 
   \end{proof}
\begin{prop}\label{prop:Minvcanc}
 Assume $P_eVx=0$. If there is a resonance of the second kind at zero, then for small $\lambda$ we have
\begin{align} \label{eq:M inverse cancel}
M^{\pm}(\lambda)^{-1} = -\f{D_1}{\lambda^2} + M_1+ {\f{g_3^{\pm}(\lambda)}{\lambda^4}} M_2 +{\f{g_2^{\pm}(\lambda)}{\lambda^2}} M_3+ \widetilde{g}_1^{\pm}( \lambda) M_4 +\lambda^2 M_5+  \widetilde O_2(\lambda^{2+})
     \end{align}
where $M_i$ are $\lambda$ independent and finite rank operators.  
  \end{prop}
  \begin{proof} 
Using the expansion \eqref{M plus S long}  together with the fact that $PS_1=S_1P=0$ we have 

\begin{align*}
\begin{split}
(M^{\pm}(\lambda)+S_1)S_1&= S_1-\lambda^2 D_0vG_1vS_1-g_2^{\pm}(\lambda) D_0vG_2vS_1 - \lambda^4 D_0vG_3vS_1 \\
& +\lambda^2g_1^{\pm}(\lambda) D_0PD_0vG_1vS_1+\lambda^4 D_0vG_1D_0 vG_1 S_1 +\widetilde O_2(\lambda^{4+}),
\end{split}
\end{align*}
\begin{align*}
\begin{split}
S_1(M^{\pm}(\lambda)+S_1)&= S_1-\lambda^2 S_1vG_1vD_0-g_2^{\pm}(\lambda) S_1vG_2vD_0 - \lambda^4 S_1vG_3vD_0\\
& +\lambda^2g_1^{\pm}(\lambda)S_1PD_0vG_1vD_0+\lambda^4 S_1vG_1D_0 vG_1 D_0 +\widetilde O_2(\lambda^{4+}).
\end{split}
\end{align*}
The assertion follows by applying Lemma~\ref{JNlemma}.  

\end{proof}

To prove the main Theorem, we need a few
lemmas.
The following variation of stationary phase from \cite{Sc2} will be useful in the analysis.

\begin{lemma}\label{stat phase}

	Let $\phi'(\lambda_0)=0$ and $1\leq \phi'' \leq C$.  Then,
  	\begin{multline*}
    		\bigg| \int_{-\infty}^{\infty} e^{it\phi(\lambda)} a(\lambda)\, d\lambda \bigg|
    		\lesssim \int_{|\lambda-\lambda_0|<|t|^{-\frac{1}{2}}} |a(\lambda)|\, d\lambda\\
    		+|t|^{-1} \int_{|\lambda-\lambda_0|>|t|^{-\frac{1}{2}}} \bigg( \frac{|a(\lambda)|}{|\lambda-\lambda_0|^2}+
    		\frac{|a'(\lambda)|}{|\lambda-\lambda_0|}\bigg)\, d\lambda.
  	\end{multline*}

\end{lemma}

Define $\mathcal{G}_n(\pm\lambda,|x-y|)$ to be kernel of n-dimensional free resolvent operator $R_0^{\pm}(\lambda^2)$.  Then we recall Lemma~2.1 in
\cite{EG1},

\begin{lemma} \label{lem:reduction} For $n\geq 2$, the following recurrence relation holds. 
$$  \Big(\f1{\lambda} \frac{d}{d\lambda} \Big)\mathcal{G}_n(\pm\lambda,|x-y|) = \f 1{2\pi}\mathcal{G}_{n-2} (\pm\lambda,|x-y|).$$
\end{lemma}

To prove Theorem~\ref{thm:evalcancel}, we need to bound
the contribution of
$$
R_0^+VR_0^+vM^+(\lambda)^{-1}vR_0^+VR_0^+
- R_0^-VR_0^-vM^-(\lambda)^{-1}vR_0^-VR_0^-
$$
to the Stone formula, \eqref{stone}.  There are a number
of terms that arise when considering the difference with
\eqref{alg fact}, which we bound in a series of Lemmas.

\begin{lemma}\label{lem:awful}

Under the assumptions of Theorem~\ref{thm:evalcancel},
we have the bound
\begin{multline*}
	\sup_{x,y\in \R^4}\bigg| \int_{\R^{16}}\int_0^\infty
	e^{it\lambda^2} \lambda \chi(\lambda)
	\frac{[R_0^+-R_0^-](\lambda^2)(x,z_1)}{\lambda^2}V(z_1)G_0(z_1,z_2)v(z_2)D_2(z_2,z_3)v(z_3)\\
	[R_0^+(\lambda^2)-G_0](z_3,z_4)V(z_4)R_0^+(\lambda^2)(z_4,y) \, d\lambda
	\, dz_1\, dz_2\, dz_3\, dz_4 \bigg| \les |t|^{-2}.
\end{multline*}

\end{lemma}

\begin{proof}

	The proof follows from a delicate case analysis.
	The integrand can be seen to be of small enough in
	the spectral variable $\lambda$ to allow for one
	integration by parts without a boundary term.
	The $\lambda$ integral is now equal to
	\begin{align}
		\frac{1}{2it} \int_0^\infty e^{it\lambda^2} \chi(\lambda) \bigg\{&
		\bigg[\frac{\partial_\lambda (R_0^+-R_0^-)(\lambda^2)}{\lambda^2}VG_0VP_eV[R_0^+(\lambda^2)-G_0]VR_0^+(\lambda^2)	\bigg]\label{term1}	\\
		&+\bigg[\frac{(R_0^+-R_0^-)(\lambda^2)}{\lambda^2}VG_0VP_eV\partial_\lambda[R_0^+(\lambda^2)-G_0]VR_0^+(\lambda^2)\bigg]\label{term2}\\
		&+\bigg[\frac{(R_0^+-R_0^-)(\lambda^2)}{\lambda^2}VG_0VP_eV[R_0^+(\lambda^2)-G_0]V
		\partial_\lambda R_0^+(\lambda^2)\bigg]\label{term3}\\
		&-2\bigg[\frac{(R_0^+-R_0^-)(\lambda^2)}{\lambda^3}VG_0VP_eV[R_0^+(\lambda^2)-G_0]V
		R_0^+(\lambda^2)\bigg]\label{term4}
		\bigg\}\, d\lambda.
	\end{align}
	When the derivative acts on $\chi(\lambda)$, it can be
	controlled as in \eqref{term4}.
	We first consider \eqref{term1}, using 
	$R_0^+(\lambda^2)-G_0=g_1^+(\lambda)+\lambda^2 G_1
	+\la x\ra^{\f12}\la y\ra^{\f12} \widetilde O_2(\lambda^{\f52})$ on the inner resolvent, and Lemma~\ref{lem:reduction}
	yields $\partial_\lambda [R_0^+-R_0^-](\lambda^2)(x,z_1)=c\lambda  J_0(\lambda |x-z_1|)$, we need only bound
	\begin{align*}
		\frac{1}{t} \int_0^\infty e^{it\lambda^2} \chi(\lambda) &
		\bigg[\frac{\lambda J_0(\lambda |x-z_1|)}{\lambda^2}VG_0VP_eV[g_1^+(\lambda)+\lambda^2 G_1 +\widetilde O(\lambda^{\f52})]VR_0^+(\lambda^2)	\bigg].
	\end{align*}
	Here $J_0$ is the  Bessel function of order zero. Since $P_eV1=0$, the term containing $g_1^+(\lambda)$
	is immediately zero.  Writing $R_0^\pm(\lambda^2)(x,y)=G_0+(1+\log^-|x-y|)\widetilde O_1(\lambda^{2-})$  for the lagging free resolvent,
	we then need only bound
	\begin{align*}
		\frac{1}{t} \int_0^\infty e^{it\lambda^2} \chi(\lambda) &
		\lambda J_0(\lambda |x-z_1|)VG_0VP_eV G_1 VG_0+(1+\log^-|z_4-y|)\widetilde O_1(\lambda^{1+})	
		\, d\lambda
	\end{align*}
	The first term can be bounded by $|t|^{-2}$ uniformly
	 in $x,y$ by Lemma~12 of \cite{Sc2}.  The second
	 term can be bounded by Lemma~\ref{lem:fauxIBP} using
	 the observation $J_0(\lambda |x-z_1|)=\widetilde{O}_1(1)$.
	 
	For \eqref{term2}, we use that
	$\partial_\lambda [R_0^+(\lambda^2)-G_0]=c\lambda (\log \lambda +1)+2\lambda G_1+\la x\ra^{\f12}
	\la y\ra^{\f12} 	\widetilde O(\lambda^{\f32})$.
	The first term can be safely ignored due to
	$P_eV1=0$ to consider
	\begin{multline*}
		\frac{1}{t} \int_0^\infty e^{it\lambda^2} \chi(\lambda) 
		\frac{(R_0^+-R_0^-)(\lambda^2)}{\lambda^2}VG_0VP_eV\\
		[\lambda G_1+\la z_3\ra^{\f12} \la z_4 \ra^{\f12}
		\widetilde O(\lambda^{\f32})]V[G_0+
		(1+\log^-|z_4-y|)\widetilde{O}_1(\lambda^{2-})]
		\, d\lambda.
	\end{multline*}
	By writing
	\begin{align*}
		[R_0^+-R_0^-](\lambda^2)(x,z_1)=\left\{
		\begin{array}{ll}
			c_1 \lambda^2+ \widetilde O_1(\lambda^4 |x-z_1|^2) & \lambda |x-z_1|\ll 1\\
			\frac{\lambda}{|x-z_1|}e^{\pm i\lambda |x-z_1|} \widetilde O((1+\lambda |x-z_1|)^{-\f12}) & \lambda |x-z_1| \gtrsim 1
		\end{array}
		\right.,
	\end{align*}
	one can employ an approach as in the Born series.
	  First, for $\lambda |x-z_1|\ll 1 $,
	the $\lambda$ integral is of the form
	\begin{align*}
		\frac{1}{t}\int_0^\infty e^{it\lambda^2}\lambda 
		\chi(\lambda)[1+\widetilde O_1(\lambda^{0+})+
		\widetilde O_1(\lambda^2 |x-z_1|^2)]\chi(\lambda |x-z_1|)\, d\lambda.
	\end{align*}
	The first two terms are easily seen to be bounded
	by $|t|^{-2}$ by integration by parts.  For the
	last term, we note that the error term is supported
	on the set $\lambda \les |x-z_1|^{-1}$ to integrate
	by parts and bound by
	\begin{align*}
		\frac{|x-z_1|^2}{|t|^2}\int_0^{|x-z_1|^{-1}} \lambda \, d\lambda \les \frac{1}{|t|^2}.
	\end{align*}
	For $\lambda |x-z_1|\gtrsim 1$, we note that the bound follows as in Lemmas~3.6 and 3.8 of
	\cite{GGeven} by using Lemma~\ref{stat phase}.
	The bound follows as long as we can write the
	resulting integral in the form
	$$
		\int_0^\infty e^{it\lambda^2 \pm i\lambda |x-z_1|}
		a(\lambda) \, d\lambda \qquad \textrm{with} 
		\qquad 
		|a(\lambda)|\les \frac{\lambda^{\f12}}{|x-z_1|^{\f12}}, \qquad
		|a'(\lambda)|\les \frac{1}{\lambda^{\f12}|x-z_1|^{\f12}}.
	$$
	Using $|x-z_1|^{-1}\les \lambda$, these can be 
	established from the bounds above.
	
	For \eqref{term3}, we use Lemma~\ref{lem:reduction}
	and $P_eV1=0$ to bound
	\begin{multline*}
		\frac{1}{t} \int_0^\infty e^{it\lambda^2} \lambda \chi(\lambda) 
		\frac{(R_0^+-R_0^-)(\lambda^2)}{\lambda^2}VG_0VP_eV\\
		[\lambda^2 G_1+\la z_3\ra^{\f12} \la z_4 \ra^{\f12}
		\widetilde O_2(\lambda^{\f52})]VR_2(\lambda^2)(z_4,y)
		\, d\lambda
	\end{multline*}
	where $R_2$ denotes  the two-dimensional 
	Schr\"odinger free resolvent.
	Since 
	$$
		\frac{(R_0^+-R_0^-)(\lambda^2)}{\lambda^2}
		=\widetilde O_1(1),
	$$
	we are left to bound
	\begin{align*}
		\frac{1}{t} \int_0^\infty e^{it\lambda^2} \lambda \chi(\lambda) 
		\widetilde O_1(\lambda^3) R_2(\lambda^2)(z_4,y)
		\, d\lambda.
	\end{align*}
	When $\lambda |z_4-y|\gtrsim 1$, the analysis 
	of Lemma~12 in \cite{Sc2} yields the desired bound.
	When $\lambda |z_4-y|\ll 1$, we note that
	\begin{align*}
		|[J_0+iY_0](\lambda |z_4-y|)|\les 1+ |\log \lambda| +\log^-|z_4-y|, \qquad 
		|\partial_\lambda [J_0+iY_0](\lambda |z_4-y|)|&\les \frac{1}{\lambda},
	\end{align*}
	which allows us to integrate by parts a second time
	without growth in $x$ or $y$.
	
	For the final term, \eqref{term4}, we again need
	to consider cases based on the size of
	$\lambda |x-z_1|$ and consider
	\begin{multline*}
		\frac{1}{t} \int_0^\infty e^{it\lambda^2}  \chi(\lambda) \bigg[
		(R_0^+-R_0^-)(\lambda^2)VG_0VP_eV G_1VG_0 \widetilde O_1(\lambda)+ (1+\log^- |z_4-y|) \widetilde O_1(\lambda^{1+})\bigg]
		\, d\lambda
	\end{multline*}	
	The bound for the first term follows as in the bound
	for \eqref{term2}, while the bound for the
	second term follows from Lemma~\ref{lem:fauxIBP}.

\end{proof}

The following bound is proved similarly.
\begin{lemma}\label{lem:awful2}

Under the assumptions of Theorem~\ref{thm:evalcancel},
we have the bound
\begin{multline*}
	\sup_{x,y\in \R^4}\bigg| \int_{\R^{16}}\int_0^\infty
	e^{it\lambda^2} \lambda \chi(\lambda)
	\frac{[R_0^+-R_0^-](\lambda^2)(x,z_1)}{\lambda^2}V(z_1)[R_0^+(\lambda^2)-G_0](z_1,z_2)\\v(z_2)D_2(z_2,z_3)v(z_3)
	G_0(z_3,z_4)V(z_4)R_0^+(\lambda^2)(z_4,y) \, d\lambda
	\, dz_1\, dz_2\, dz_3\, dz_4 \bigg| \les |t|^{-2}.
\end{multline*}

\end{lemma}

Then, we have

\begin{prop}\label{lem:awful3}

Under the assumptions of Theorem~\ref{thm:evalcancel},
we have the bound
\begin{multline*}
	\sup_{x,y\in \R^4}\bigg| \int_{\R^{16}}\int_0^\infty
	e^{it\lambda^2} \lambda \chi(\lambda)
	\frac{[R_0^+-R_0^-](\lambda^2)(x,z_1)}{\lambda^2}V(z_1)G_0(z_1,z_2)v(z_2)D_2(z_2,z_3)v(z_3)\\
	G_0(z_3,z_4)V(z_4)R_0^+(\lambda^2)(z_4,y) \, d\lambda
	\, dz_1\, dz_2\, dz_3\, dz_4 \bigg| \les |t|^{-2}.
\end{multline*}

\end{prop}

To prove this proposition, we write
\begin{multline*}
	\frac{R_0^+-R_0^-}{\lambda^2}VG_0vD_2vG_0VR_0^+\\=
	\frac{R_0^+-R_0^-}{\lambda^2}VG_0vD_2vG_0VG_0+
	\frac{R_0^+-R_0^-}{\lambda^2}VG_0vD_2vG_0V[R_0^+-G_0].
\end{multline*}
The first term is handled by the machinery set up in
\cite{GGeven}, specifically Lemma~4.5.  The second term
requires more care.  We introduce some ideas and
techniques inspired by the two-dimenstional treatment  from \cite{Sc2,EG2}.  We first note that
\begin{align*}
	\frac{R_0^+-R_0^-}{\lambda^2}VG_0vD_2vG_0V[R_0^+-G_0]
	=\frac{R_0^+-R_0^-}{\lambda^2}VP_eV[R_0^+-G_0]
\end{align*}
and when $P_eV1=0$, we have
\begin{multline}\label{eqn:PV1trick}
	\int_{\R^8}\frac{[R_0^+-R_0^-](\lambda^2)(x,z_1)}{\lambda^2}VP_eV[R_0^+-G_0](z_4,y) \, dz_1\, dz_4\\
	=\int_{\R^8}\frac{\big\{[R_0^+-R_0^-](\lambda^2)(x,z_1)-f(\lambda, x)\big\}}{\lambda^2}VP_eV\big\{[R_0^+-G_0](z_4,y)-g(\lambda, y)
	\big\} \, dz_1\, dz_4
\end{multline}
for any funtions $f,g$ that are independent of 
$z_1$ and $z_4$ respectively.  To utilize this cancellation, we must consider the different behavior
for the resolvents for small and large arguments.  We
begin by noting \eqref{resolv def} and \eqref{H0} to see
\begin{align}
	[R_0^+-R_0^-](\lambda^2)(x,y)=c \frac{\lambda}{|x-y|}
	J_1(\lambda|x-y|)=c\lambda^2 \frac{J_1(\lambda|x-y|)}{\lambda |x-y|}:=\lambda^2
	A(\lambda|x-y|)\label{eqn:Adefn}
\end{align}
In particular, we use \eqref{eqn:PV1trick} to subtract
off $\lambda^2 A(\lambda(1+|x|))$.  We define
\begin{align}
	G(\lambda,p,q):=A(\lambda p)\chi(\lambda p)-A(\lambda q) \chi(\lambda q).
\end{align}
To control the contribution of $R_0^+-G_0$, recalling
\eqref{H0}, we defined $G(\lambda, p,q)$ to
control the contribution of
 $J_1$, we now turn to the contribution of
$Y_1$.  We have
\begin{align*}
	[R_0^+(\lambda^2)(x,y)-G_0(x,y)]
	=\frac{i \lambda }{8\pi |x-y|}J_1(\lambda|x-y|)-
	\frac{\lambda^2}{8\pi}\bigg[
	 \frac{Y_1(\lambda |x-y|)}{\lambda |x-y|}-\frac{2}{\pi \lambda^2 |x-y|^2}\bigg].
\end{align*}
This leads us to define the function
\begin{align}
	F(\lambda, p, q)=\chi(\lambda p)\bigg[\frac{Y_1(\lambda p)}{\lambda p}+\frac{2}{\pi \lambda^2 p^2}
	\bigg] -\chi(\lambda q)\bigg[\frac{Y_1(\lambda q)}{\lambda q}+\frac{2}{\pi \lambda^2 q^2}
	\bigg].
\end{align}
In addition, we define $k(x,y):=1+\log^+|y|+\log^-|x-y|$.

\begin{lemma}\label{lem:FG bounds}

	Let $p:=|x-y|$ and $q=1+|x|$, then for
	$0<\lambda<2\lambda_1 \ll 1$,
	\begin{align*}
		|\partial_\lambda ^k G(\lambda,p,q)|&\les
		\lambda^{1-k} |p-q|\les
		\lambda^{1-k} \la y \ra, \qquad k=0,1,2,\\
		|\partial_\lambda^k F(\lambda, p, q)|&\les  
		\lambda^{-k} k(x,y), \qquad k=0,1,2.
	\end{align*}

\end{lemma}

\begin{proof}

	We begin with the bounds for $G(\lambda, p,q)$.
	We note that
	by the asymptotic expansion in \eqref{J0 def}, we
	have $g(\lambda p)=A(\lambda p) \chi(\lambda p)=1+(\lambda p)^2+\widetilde O_4((\lambda p)^2)$  Thus, $g'(z)=O(1)$, and by
	the mean value theorem, we see
	\begin{align*}
		|G(\lambda, p,q)|=|g(\lambda p)-g(\lambda q)|\les 
		\lambda |p-q| |g'(c)|\les \lambda |p-q|.
	\end{align*}
	For the derivatives, we note that for $k=1,2$,
	\begin{align*}
		|\partial_\lambda^k G(\lambda, p,q)|&=
		|p^k g^{(k)}(\lambda p)-q^k g^{(k)}(\lambda q)|
		=\frac{|(\lambda p)^k g^{(k)}(\lambda p)-(\lambda q)^k g^{(k)}(\lambda q)|}{\lambda^k}
	\end{align*}
	Again, by \eqref{J0 def}, we have $|\partial_z[z^k g^{(k)}](z)|=O(1)$ and then the mean value theorem
	we have 
	\begin{align*}
		|(\lambda p)^k g^{(k)}(\lambda p)-(\lambda q)^k g^{(k)}(\lambda q)|\les \lambda \big|p-q\big| \big|\partial_z
		[z^k g^{(k)}(z)]\big|\les \lambda |p-q|
	\end{align*}

	We now turn to the bounds for $F(\lambda, p,q)$.
	The bounds follow by the expansion \eqref{Y0 def}.  Note
	that adding the term $2(\pi \lambda^2 p^2)^{-1}$ 
	exactly cancels out the singular term.  We thus have
	\begin{align*}
		\chi(\lambda p)\bigg[\frac{Y_1(\lambda p)}{\lambda p}+\frac{2}{\pi \lambda^2 p^2}\bigg]&=
		\frac{1}{\pi}\log(\lambda p/2)+b_1 \lambda p -
		\frac{1}{8\pi} (\lambda p)^2 \log (\lambda p)+
		b_2 (\lambda p)^2\\
		&\qquad + \widetilde O((\lambda p)^4 \log (\lambda p))
		:=b(\lambda p).
	\end{align*}
	From this expansion, we can see that
	\begin{align}\label{Fat0}
		F(0+,p,q)=\log \bigg(\frac{|x-y|}{1+|x|}
		\bigg)+c \les k(x,y).
	\end{align}
	Similar to the two dimensional case considered in
	\cite{EG2}, we have
	\begin{align*}
		|\partial_\lambda F(\lambda, p,q)|=|p\chi'(\lambda p)
		b (\lambda p)-q\chi'(\lambda q) b(\lambda q)
		+\partial_\lambda b(\lambda p)\chi(\lambda p)-
		\partial_\lambda b(\lambda q) \chi(\lambda q)|.
	\end{align*}
	The bound of $\lambda^{-1}k(x,y)$ can be seen by using
	$\chi'(z)$ is supported on $z\approx 1$ and
	\begin{align*}
		|p\chi'(\lambda p)b (\lambda p)|\les \frac{1}{\lambda}
		|z\chi'(z)(1+\log z)|\les \frac{1}{\lambda}.
	\end{align*}
	Further, one has $	|b'(\lambda p)|\les \f1\lambda$.
	To bound $F$, we note \eqref{Fat0} allows us to bound
	\begin{align}
		\int_0^{2\lambda_1} |\partial_\lambda F(\lambda, p,q)|
		\, d\lambda &\les  \int_0^{2\lambda_1}
		|p\chi'(\lambda p)b (\lambda p)|+|q\chi'(\lambda q) b(\lambda q)| \, d\lambda \\
		&+ \int_0^{2\lambda_1}
		|[\partial_\lambda b(\lambda p)]\chi(\lambda p)-
		[\partial_\lambda b(\lambda q)] \chi(\lambda q)|\, d\lambda.\label{Fboundint}
	\end{align}
	The first line is seen to be bounded by the previous
	discussion.  For the second line, we note that
	$$
		\partial_{\lambda}b(\lambda p)=\frac{1}{\pi \lambda}+ b_1p+ O(
		p(\lambda p)^{1-} )=\frac{1}{\pi \lambda}+O
		\bigg(\frac{p^{\f12}}{\lambda^{\f12}}
		\bigg).
	$$
	Thus, we can bound \eqref{Fboundint} by
	\begin{align*}
		\int_0^{2\lambda_1} \frac{1}{\lambda}[\chi(\lambda p)-\chi(\lambda q)]
		+ \chi(\lambda p) \frac{p^{\f12}}{\lambda^{\f12}}
		+ \chi(\lambda q) \frac{q^{\f12}}{\lambda^{\f12}}\, d\lambda
		\les k(x,y).
	\end{align*}
	The first integrand is bounded since
	$\chi(\lambda p)-\chi(\lambda q)$ is supported on the
	set $[\frac{\lambda_1}{2p},2\frac{\lambda_1}{q} ]$, while the
	remaining pieces follow by integration using that
	$\lambda \les p^{-1}$ on the support of 
	$\chi(\lambda p)$. 
	
	For the second derivative, we note that
	\begin{align*}
		|\partial_\lambda^2 F(\lambda, p,q)|&=
		|p^2\chi''(\lambda p)
		b (\lambda p)-q^2\chi''(\lambda q) b(\lambda q)
		+p[\partial_\lambda b(\lambda p)]\chi'(\lambda p)-
		q[ \partial_\lambda b(\lambda q)] \chi'(\lambda q)\\
		&+[\partial^2_\lambda b(\lambda p)]\chi(\lambda p)-
		[\partial^2_\lambda b(\lambda p)]  \chi(\lambda q)|
	\end{align*}
	By multiplying and dividing by $\lambda$, and using
	that $|z^2 b(z)|,|zb(z)|\les 1$ for the terms with
	two derivatives on $b$, we note that since
	$\lambda p \leq 2\lambda_1\ll 1$, we have
	\begin{align*}
		|\partial^2_\lambda b(\lambda p)|\les p^2 +p^2 |\log(\lambda p)|
		\les \frac{1}{\lambda^{2}}
		+\frac{\lambda^2 p^2|\log (\lambda p)|}{\lambda^2}\les \frac{1}{\lambda^2}.
	\end{align*}

\end{proof}

\begin{lemma}\label{lem:lowlow}

	Under the assumptions of Theorem~\ref{thm:evalcancel},
	we have the bound
	\begin{multline*}
		\sup_{x,y\in \R^4}\bigg| \int_{\R^{16}}
		\int_0^\infty e^{it\lambda^2}\lambda \chi(\lambda) \frac{R_0^+-R_0^-}{\lambda^2}(\lambda^2)(x,z_1)
		\chi(\lambda |x-z_1|)\\
		VP_eV[R_0^+(\lambda^2)-G_0](z_4,y) 
		\chi(\lambda |z_4-y|)
		\, d\lambda \, dz_1\, dz_2\, dz_3\, dz_4\bigg| \les |t|^{-2}.
	\end{multline*}

\end{lemma}

\begin{proof}

	Considering the $\lambda$ integral, and \eqref{eqn:PV1trick}, we can replace
	$[R_0^+-R_0^-](\lambda^2)(x,z_1)\chi(\lambda |x-z_1|)$
	with $\lambda^2 G(\lambda, p_1,q_1)$ with $p_1=|x-z_1|$ and
	$q_1=1+|x|$. We can also replace
	$[R_0^+(\lambda^2)-G_0](z_4,y) \chi(\lambda |z_4-y|)$
	with $\lambda^2 G(\lambda, p_2,q_2)+\lambda^2 F(\lambda, p_2,q_2)$ with $p_2=|z_4-y|$ and
	$q_2=1+|y|$.  Thus, we are lead to bound the integral
	\begin{align*}
		\int_0^\infty e^{it\lambda^2} \lambda^3 \chi(\lambda) G(\lambda, p_1, q_1)[F(\lambda, p_2,q_2)
		+G(\lambda, p_2,q_2)]\, d\lambda
	\end{align*}
	By the bounds in Lemma~\ref{lem:FG bounds}, we can
	express this integral as 
	\begin{align*}
		[k(z_4,y)+\la z_4\ra ] \la z_1\ra \int_0^\infty e^{it\lambda^2}
		\chi(\lambda) \widetilde O_2(\lambda^{4})\, d\lambda \les \frac{[k(z_4,y)+\la z_4\ra ]\la z_1\ra}{|t|^2}.
	\end{align*}
	The $\lambda$ smallness allows us to integrate by parts
	twice without boundary terms to gain the $|t|^{-2}$ 
	time decay.

	We close the argument by bounding the spatial integrals
	by
	$$
		\sup_{x,y \in \R^4}
		\| \la \cdot \ra V\|_{L^2} \||P_e|\|_{L^2\to L^2} 
		\|\la \cdot \ra k(\cdot, y) V\|_{L^2}\les 1.
	$$

\end{proof}

For when the Bessel functions are supported on a
large argument, we recall that asymptotics \eqref{JYasymp2} and define the functions
\begin{align}\label{Gtilde def}
	\widetilde G^\pm (\lambda, p, q)=\widetilde \chi(\lambda p) \tilde w_{\pm}(\lambda p)-
	e^{\pm i \lambda (p-q)}\widetilde \chi(\lambda q)
	\tilde w_{\pm}(\lambda q),	\qquad
	\tilde w_{\pm}(z)=\widetilde O(z^{-\f32}).
\end{align}
Here we have absorbed the $\lambda/|x-y|$ from 
\eqref{resolv def} to the asymptotic expansion.
This allows us to write
\begin{align*}
	[R_0^+-R_0^-](\lambda^2)(p)\widetilde \chi(\lambda p)
	-[R_0^+-R_0^-](\lambda^2)(q)\widetilde \chi(\lambda q)
	=\lambda^2 \bigg[e^{i\lambda p}\widetilde G^+(\lambda, p, q)+e^{-i\lambda p}\widetilde G^-(\lambda, p, q)
	\bigg].
\end{align*}

\begin{lemma}\label{lem:Gtilde}

	For any $0\leq \tau \leq 1$, we have the bounds
	\begin{align*}
		|\widetilde G^\pm (\lambda, p, q)|&\les
		(\lambda |p-q|)^{\tau} \bigg(\frac{\widetilde \chi(\lambda p)}{|\lambda p|^{\frac{3-\tau}{2}}}+\frac{\widetilde \chi(\lambda q)}{|\lambda q|^{\frac{3-\tau}{2}}}\bigg),\\
		|\partial_\lambda \widetilde G^\pm (\lambda, p, q)|&\les
		(|p-q|+\lambda^{-1} )\bigg(\frac{\widetilde \chi(\lambda p)}{|\lambda p|^{\f32}}+\frac{\widetilde \chi(\lambda q)}{|\lambda q|^{\f32}}\bigg),\\
		|\partial_\lambda^2 \widetilde G^\pm (\lambda, p, q)|&\les
		(|p-q|+\lambda^{-1})^2 \bigg(\frac{\widetilde \chi(\lambda p)}{|\lambda p|^{\f32}}+\frac{\widetilde \chi(\lambda q)}{|\lambda q|^{\f32}}\bigg).
	\end{align*}

\end{lemma}

\begin{proof}

	We prove the case for $\widetilde G^+$ and ignore the
	superscript.  The bound
	\begin{align}\label{eqn:Gtrivial}
		|\widetilde G(\lambda,p,q)|\les 
		\bigg(\frac{\widetilde \chi(\lambda p)}{|\lambda p|^{\f32}}+\frac{\widetilde \chi(\lambda q)}{|\lambda q|^{\f32}}\bigg)
	\end{align}
	is clear.  To gain $\lambda$ smallness, we 
	define the function
	$c(s)=\widetilde \chi(s) w(s)$ which satisfies
	$|c^{(k)}(s)|\les \widetilde \chi(s) |s|^{-\f32-k}$ to write
	\begin{align}
		\widetilde G(\lambda,p,q)=c(\lambda p)-c(\lambda q)
		+\big(1-e^{i\lambda (p-q)}\big)c(\lambda q).
	\end{align}
	The second summand is easily seen to be bounded by
	$$
		\lambda |p-q| |c(\lambda q)|\les \lambda |p-q|
		\frac{\widetilde \chi(\lambda q)}{|\lambda q|^{\f32}}.
	$$
	For the first term, without loss of generality we
	assume that $p>q$ to see
	\begin{align*}
		|c(\lambda p)-c(\lambda q)|=\bigg| \int_{\lambda q}^{\lambda p} c'(s)\, ds\bigg| \les 
		\int_{\lambda q}^{\lambda p} \widetilde \chi(s)|s|^{-\f52}\, ds.
	\end{align*}
	In the case that $1<\lambda q<\lambda p$, this integral
	is bounded by
	$$
		\lambda |p-q| \frac{\widetilde \chi(\lambda q)}
		{|\lambda q|^{\f32}}.
	$$
	In the case that $\lambda q<1<\lambda p$, we bound
	as
	\begin{align*}
		\int_{\lambda q}^{\lambda p} \widetilde \chi(s)|s|^{-\f52}\, ds &\les \widetilde \chi(\lambda p) \int_1^{\lambda p} s^{-\f52}\, ds
		\les \widetilde \chi(\lambda p) \frac{(\lambda p)^{\f32}-1}{|\lambda p|^{\f32}}
		\les \widetilde \chi(\lambda p) \frac{(\lambda p)^{\f32}-(\lambda q)^{\f32}}{|\lambda p|^{\f32}}\\
		&\les \widetilde \chi(\lambda p)\frac{\lambda |p-q|}{|\lambda p|},
	\end{align*}
	where the last bound follows by
	$$
		|b^{\f32}-a^{\f32}|\les \int_{a}^b \sqrt{s}\, ds
		\les |b-a| \sqrt b.
	$$
	Since $\lambda p, \lambda q\gtrsim 1$, the dominant
	term is the bound
	$$
		\lambda |p-q|\bigg(\frac{\widetilde \chi(\lambda p)}{|\lambda p|}+ \frac{\widetilde \chi(\lambda q)}{|\lambda q|}
		\bigg).
	$$
	Interpolating between this and the trivial bound
	\eqref{eqn:Gtrivial}, one obtains the desired bound.
	
	We now turn to derivatives, rather than rewriting 
	$\widetilde G$ we use \eqref{Gtilde def} directly to see
	\begin{align*}
		|\partial_\lambda \widetilde G(\lambda, p,q)|
		&=|pc'(\lambda p)-i(p-q)e^{i\lambda (p-q)}c(\lambda q)-e^{i\lambda (p-q)}qc'(\lambda q)|\\
		&\les  \frac{c_1(\lambda p)+c_1(\lambda q)}{\lambda}+ |p-q|c(\lambda q)
	\end{align*}
	Here $c_1(s):=sc'(s)$ satisfies the same bounds as
	$c(s)$.  This suffices to prove the desired bound for
	the first derivative.  For the second derivative, we
	again use \eqref{Gtilde def} to see
	\begin{align*}
		|\partial_\lambda^2 \widetilde G(\lambda, p,q)|
		&\les|p^2c''(\lambda p)+(p-q)^2c(\lambda q)+(p-q)qc'(\lambda q)+q^2c'(\lambda q)|\\
		&\les  \frac{c_2(\lambda p)+c_2(\lambda q)}{\lambda^2}+ |p-q|^2c(\lambda q)+
		\frac{|p-q|}{\lambda} c_1(\lambda q).
	\end{align*}	
	With $c_2(s):=s^2c''(s)$ satisfies the same bounds as
	$c(s)$.  This establishes the desired bound.

\end{proof}

\begin{lemma}\label{lem:highlow}

	Under the assumptions of Theorem~\ref{thm:evalcancel},
	we have the bound
	\begin{multline*}
	\sup_{x,y\in \R^4}\bigg| 
		\int_{\R^{16}}
		\int_0^\infty e^{it\lambda^2}\lambda 
		\chi(\lambda) \frac{R_0^+-R_0^-}{\lambda^2}(\lambda^2)(x,z_1)
		\widetilde \chi(\lambda |x-z_1|)\\
		VP_eV[R_0^+(\lambda^2)-G_0](z_4,y) 
		\chi(\lambda |z_4-y|)
		\, d\lambda \, dz_1\, dz_2\, dz_3\, dz_4\bigg| \les |t|^{-2}.
	\end{multline*}

\end{lemma}

\begin{proof}

	As in the proof of Lemma~\ref{lem:lowlow}, we employ
	the functions $F$, $G$ and $\widetilde G$ as needed.
	We will prove the bound for $F$ in place of 
	$[R_0^+(\lambda^2)-G_0]$ as this is larger in $\lambda$.
	We define $p_1:=\max(|x-z_1|, 1+|x|)$ and
	$p_2:=\min(|x-z_1|, 1+|x|)$.
	Accordingly, we see to bound
	\begin{align*}
		\int_0^\infty e^{it\lambda^2}
		\lambda^3 \chi(\lambda) e^{\pm i\lambda p_2} \widetilde G(\lambda, p_1, p_2)F(\lambda, q_1,q_2)\, d\lambda.
	\end{align*}
	The $\lambda$ smallness of the integrand allows us to
	integrate by parts once without boundary terms to 
	bound with
	\begin{align*}
	\frac{1}{t}	\int_0^\infty e^{it\lambda^2}
	\partial_\lambda \bigg[
	\lambda^2 \chi(\lambda) e^{\pm i\lambda p_2} \widetilde G(\lambda, p_1, p_2)F(\lambda, q_1,q_2)\bigg]\, d\lambda	
	=\frac{1}{t} \int_0^\infty e^{it\phi_\pm (\lambda)}
	a(\lambda) \, d\lambda	.
	\end{align*}
	Here $\phi_\pm(\lambda)=\lambda^2\pm \lambda p_2t^{-1}$.
	In Lemma~3.8 of \cite{EG2}, using Lemma~\ref{stat phase}
	it is proven that
	\begin{align*}
		\bigg| \int_0^\infty e^{it\phi_\pm (\lambda)}
			a(\lambda) \, d\lambda	\bigg| \les \frac{1}{|t|},
	\end{align*}
	provided
	\begin{align*}
		|a(\lambda)|\les \chi(\lambda) \lambda^{\f12} \bigg(
		\frac{\widetilde \chi(\lambda p_1)}{p_1^{\f12}}+
		\frac{\widetilde \chi(\lambda p_2)}{p_2^{\f12}}\bigg),
		\qquad
		|a'(\lambda)|\les \chi(\lambda) \lambda^{-\f12} \bigg(
		\frac{\widetilde \chi(\lambda p_1)}{p_1^{\f12}}+
		\frac{\widetilde \chi(\lambda p_2)}{p_2^{\f12}}\bigg).
	\end{align*}
	Since we have
	\begin{align*}
		a(\lambda)=
		e^{\mp i\lambda p_2}\partial_\lambda \bigg[
		\lambda^2 \chi(\lambda) e^{\pm i\lambda p_2} \widetilde G(\lambda, p_1, p_2)F(\lambda, q_1,q_2)\bigg].
	\end{align*}
	The bounds of Lemma~\ref{lem:FG bounds} and 
	\ref{lem:Gtilde} give us
	\begin{align*}
		a(\lambda)=p_2 \widetilde O_1(\lambda^2)\widetilde
		G(\lambda, p_2,p_1)+\widetilde O_1(\lambda )\widetilde G(\lambda, p_1,p_2)+\widetilde O(\lambda^2) \partial_\lambda \widetilde
		G(\lambda, p_2,p_1)
	\end{align*}
	satisfies the desired bounds. 
	
\end{proof}

\begin{lemma}\label{lem:highhigh}

	Under the assumptions of Theorem~\ref{thm:evalcancel},
	we have the bound
	\begin{multline*}
		\sup_{x,y\in \R^4}\bigg| 
		\int_{\R^{16}}
		\int_0^\infty e^{it\lambda^2}\lambda 
		 \chi(\lambda) \frac{R_0^+-R_0^-}{\lambda^2}(\lambda^2)(x,z_1)
		\widetilde \chi(\lambda |x-z_1|)\\
		VP_eV[R_0^+(\lambda^2)-G_0](z_4,y) 
		\widetilde\chi(\lambda |z_4-y|)
		\, d\lambda \, dz_1\, dz_2\, dz_3\, dz_4\bigg| \les |t|^{-2}.
	\end{multline*}

\end{lemma}

\begin{proof}

We note that the contribution of $G_0(z_4,y) \widetilde\chi(\lambda |z_4-y|)=\widetilde O_2(\lambda^2)$.
Thus, it's contribution may be bounded by
\begin{align*}
	\int_0^\infty e^{it\lambda^2}\lambda 
	 \chi(\lambda) \frac{R_0^+-R_0^-}{\lambda^2}(\lambda^2)(x,z_1)
	\widetilde \chi(\lambda |x-z_1|)\widetilde O_2(\lambda^2)  	\, d\lambda. 	
\end{align*}
This can be bounded by $|t|^{-2}$ as in the proof of
Lemma~\ref{lem:highlow} when the auxiliary function
$F(\lambda,p,q)$ is used.

Assume that $t>0$ and recall \eqref{resolv def} and \eqref{JYasymp2}
\begin{align*}
	R_0^\pm(\lambda^2)(x,y)\chi(\lambda |x-y|)
	&=\pm\frac{i}{4}\frac{\lambda}{2\pi |x-y|} H_1^\pm(\lambda|x-y|)\chi(\lambda |x-y|)\\
	&=\frac{\lambda}{|x-y|}e^{\pm i\lambda |x-y|}\omega_\pm (\lambda |x-y|) .
\end{align*}
While for the difference of resolvents we have both the `+' and `-' phases,
\begin{align*}
	[R_0^+-R_0^-](\lambda^2)(x,y)\chi(\lambda |x-y|)
	&=\frac{\lambda}{|x-y|}\bigg[e^{ i\lambda |x-y|}\omega_+ (\lambda |x-y|)+e^{- i\lambda |x-y|}\omega_- (\lambda |x-y|)
	\bigg] .
\end{align*}
To employ the auxiliary functions $\widetilde{G}$ we denote $p_1= \max(|x-z_1|,1+|x|)$, $p_2=\min(|x-z_1|,1+|x|)$, $q_1= \max(|y-z_4|,1+|y|)$ and $q_2=\min(|y-z_4|,1+|y|)$. Assuming  $p_1,p_2,q_1,q_2 > 0$ we can exchange $[R_0^+-R_0^-](\lambda^2)(x,z_1)$ by  sum of two terms 
$\lambda^2 e^{\pm i\lambda p_2} \widetilde{G}^\pm(\lambda,p_1,p_2)$ and $[R_0^+(\lambda^2)-G_0](z_4,y)$ by $\lambda^2 e^{i\lambda q_2}\widetilde{G}^+(\lambda,q_1,q_2)+\widetilde O_2(\lambda^2)$. As a result, we need to establish
\begin{align*}
	\int_{\R^{16}}
	\int_0^\infty e^{i t \phi_{\pm}(\lambda)}\lambda^3 
	 \chi(\lambda) \widetilde{G}^\pm(\lambda,p_1,p_2)
	VP_eV \widetilde{G}^+(\lambda,q_1,q_2)  \, d\lambda \,
	d\vec z\les t^{-2},
\end{align*}
where 
$$\phi_{\pm}(\lambda)=\lambda^2 + \lambda \frac{q_2\pm p_2 }{t}.$$ 

We consider first when the phase is $\phi_+$.
Note that the powers of $\lambda$ allow us to  integrate by parts once without boundary terms to obtain
\begin{align*}
	\frac{1}{t}	\int_0^\infty e^{it\lambda^2}
	\partial_\lambda \bigg[\lambda^2 \chi(\lambda) e^{ i\lambda (q_2 + p_2)}\widetilde{G}^\pm (\lambda,p_1,p_2)
	\widetilde{G}^+(\lambda,q_1,q_2)\bigg]\, d\lambda	
	=\frac{1}{t} \int_0^\infty e^{it\phi_\pm (\lambda)}
	a(\lambda) \, d\lambda.	
	\end{align*}
It is now enough to show 	
\begin{align} \label{whtp}
    \Big|  \int_0^\infty e^{it\phi_\pm (\lambda)}
	a(\lambda) \, d\lambda \Big| \les {\f 1 t}. 
\end{align}

To do that we need to determine the upper bounds for $|a(\lambda)|$ and $|a^\prime(\lambda)|$. 
We have 
    \begin{align} \label{a bound}
    \begin{split}
		|a(\lambda)| & \les  \partial_{\lambda} \big[\lambda^2  \chi(\lambda) e^{i\lambda (p_2+q_2)}\widetilde{G}^+(\lambda,p_1,p_2)\widetilde{G}^+(\lambda,q_1,q_2) \big] \\
		& \les \chi(\lambda) \la z_1 \ra \la z_4 \ra \bigg(
		\frac{\widetilde \chi(\lambda p_1)}{p_1^{\f12}}+
		\frac{\widetilde \chi(\lambda q_2)}{q_2^{\f12}}\bigg) \bigg(
		\frac{\widetilde \chi(\lambda q_1)}{q_1^{\f12}}+
		\frac{\widetilde \chi(\lambda q_2)}{q_2^{\f12}}\bigg).
    \end{split}	
    \end{align}
           
   If the derivative acts on one of the $\widetilde{G}(\lambda,p_1, p_2)$, using the 
   bounds of Lemma~\ref{lem:Gtilde} we have
   \begin{align*}
   		\chi(\lambda)\lambda^2 &(|p_1-p_2|+\lambda^{-1})
   		\bigg(
   		\frac{\widetilde \chi(\lambda p_1)}{|\lambda p_1|^{\f32}}+
   		\frac{\widetilde \chi(\lambda q_2)}{|\lambda p_2|^{\f32}}\bigg) \bigg(
   		\frac{\widetilde \chi(\lambda q_1)}{|\lambda q_1|^{\f32}}+
   		\frac{\widetilde \chi(\lambda q_2)}{|\lambda q_2|^{\f32}}\bigg)\\
   		&\les \chi(\lambda)\lambda \la z_1 \ra 
   		\bigg(
   		\frac{\widetilde \chi(\lambda p_1)}{|\lambda p_1|^{\f12}}+
   		\frac{\widetilde \chi(\lambda q_2)}{|\lambda p_2|^{\f12}}\bigg) \bigg(
   		\frac{\widetilde \chi(\lambda q_1)}{|\lambda q_1|^{\f12}}+
   		\frac{\widetilde \chi(\lambda q_2)}{|\lambda q_2|^{\f12}}\bigg),
   \end{align*} 
   where we used that $\lambda p_j, \lambda q_j \gtrsim 1$
   and $|p_1-p_2|\les \la z_1 \ra$.  The desired bound 
   then follows by cancelling the $\lambda$ in the numerator
   with $\lambda^{\f12}$ in the denominator of each factor.
   The argument is identical for $\widetilde{G}(\lambda,q_1, q_2)$.  Similar bounds hold
   if the derivative acts on $\lambda^2$ or the cut-off
   function $\chi(\lambda)$, noting that $|\chi'(\lambda)|\approx \lambda^{-1}$ on the range of
   $\lambda$ under consideration.
   
   On the other hand, if the derivative acts on $e^{i\lambda(p_2+q_2)}$ in \eqref{a bound} we use
\begin{multline*}
      \lambda^2 p_2 \ \widetilde{G}(\lambda,p_1,p_2)\widetilde{G}(\lambda,q_1,q_2) 
      = \lambda (\lambda p_2) \bigg(\frac{\widetilde \chi(\lambda p_1)}{|\lambda p_1|^{\f 32}}
      +\frac{\widetilde \chi(\lambda p_2)}{|\lambda p_2|^{\f32}}\bigg) 
      \bigg(\frac{\widetilde \chi(\lambda q_1)}{|\lambda q_1|^{\f 32}}
      +\frac{\widetilde \chi(\lambda q_2)}{|\lambda q_2|^{\f32}}\bigg) \\
       \les \lambda  
       \bigg(\frac{\widetilde \chi(\lambda p_1)}{|\lambda p_1|^{\f1 2}}
       +\frac{\widetilde \chi(\lambda p_2)}{|\lambda p_2|^{\f12}}\bigg) 
        \bigg(\frac{\widetilde \chi(\lambda q_1)}{|\lambda q_1|^{\f 12}}
       +\frac{\widetilde \chi(\lambda q_2)}{|\lambda q_2|^{\f12}}\bigg)      \\
       \les \bigg(\frac{\widetilde \chi(\lambda p_1)}{ p_1^{\f1 2}}
       +\frac{\widetilde \chi(\lambda p_2)}{ p_2^{\f12}}\bigg) 
       \bigg(\frac{\widetilde \chi(\lambda q_1)}{ q_1^{\f 12}}
       +\frac{\widetilde \chi(\lambda q_2)}{ q_2^{\f12}}\bigg) .                       
   \end{multline*}
Here we used $p_2\leq p_1$. An identical argument holds for
$q_2$.
Similarly, one can obtain the first derivative of $a(\lambda)$ as 
 \begin{align} \label{a derivative}
		|a'(\lambda)| \les\la z_1 \ra^2 \la z_4\ra ^2  \chi(\lambda) \lambda^{-1}\bigg(
		\frac{\widetilde \chi(\lambda p_1)}{p_1^{\f12}}+
		\frac{\widetilde \chi(\lambda q_2)}{q_2^{\f12}}\bigg) \bigg(
		\frac{\widetilde \chi(\lambda q_1)}{q_1^{\f12}}+
		\frac{\widetilde \chi(\lambda q_2)}{q_2^{\f12}}\bigg).
\end{align}	
This can be seen by noting that the bound in Lemma~\ref{lem:Gtilde} show that if $0<\lambda \ll 1$, 
\begin{align*}
	|\partial_\lambda^k \widetilde G^\pm (\lambda, p_1,p_2)|
	\les \bigg(\frac{\la p_1-p_2\ra}{\lambda} \bigg)^k
	\bigg(\frac{\widetilde \chi(\lambda p_1)}{|\lambda p_1|^{\f32}}+\frac{\widetilde \chi(\lambda p_2)}{|\lambda p_2|^{\f32}}\bigg), \qquad k=0,1,2.
\end{align*}
		
The desired bound, \eqref{whtp}, follows as in Lemma~3.10 of \cite{EG2}. 

We now turn to the case of $\phi_-(\lambda)$ which has
opposing phases.  We wish to reduce to previously considered
cases as much as possible.  We note that if $\f12 q_2> p_2$,
we have $q_2-p_2 \approx q_2$, this allows us effectively
reduce to an integral of the form
\begin{align*}
	\int_0^\infty e^{i t \lambda^2+i\lambda a_2}\lambda^3 
	 \chi(\lambda) \widetilde{G}^-(\lambda,p_1,p_2)
	 \widetilde{G}^+(\lambda,q_1,q_2)  \, d\lambda.
\end{align*}
This can be controlled as in the proof of
Lemma~\ref{lem:highlow} to get the desired bound since
the constant (in $\lambda$)
$a_2$ satisfies the same bounds as $q_2$.  The
bounds on $\widetilde G^-(\lambda, p_1, p_2)$ are
bounded by those
used in $F(\lambda, p_1, p_2)$ in this proof.

In the case that $q_2< \f12 p_2$, we have $q_2-p_2\approx -p_2$, this allows us effectively
reduce to an integral of the form
\begin{align*}
	\int_0^\infty e^{i t \lambda^2-i\lambda b_2}\lambda^3 
	 \chi(\lambda) \widetilde{G}^-(\lambda,p_1,p_2)
	 \widetilde{G}^+(\lambda,q_1,q_2)  \, d\lambda.
\end{align*}  
This also can be controlled as in the proof of
Lemma~\ref{lem:highlow} to get the desired bound using that
the constant
$b_2$ satisfies the same bounds as $p_2$.

We now consider the final case in which 
$q_2\approx p_2$, 
and we cannot effectively reduce to the previous cases.
In this case, we need to bound an integral of the form
\begin{align*}
	\int_0^\infty e^{i t \lambda^2+i\lambda (q_2-p_2)}\lambda^3 
	 \chi(\lambda) \widetilde{G}^-(\lambda,p_1,p_2)
	 \widetilde{G}^+(\lambda,q_1,q_2)  \, d\lambda.
\end{align*} 
In the previous cases, we do not integrate by parts twice
to avoid spatial weights.  In this case, we use the
$\lambda$ smallness to integrate by parts twice.
We need to
bound
\begin{align}\label{eqn:IBP nasty}
	\int_0^\infty e^{it\lambda^2} e^{i\lambda (q_2-p_2)}
	b(\lambda) \, d\lambda.
\end{align}
Using the bounds in Lemma~\ref{lem:Gtilde},
$b(\lambda)$ is a function that is supported on 
$[0,2\lambda_1)$ that satisfies
\begin{align}\label{eqn:IBP b bounds}
	|\partial_\lambda^k b(\lambda)| \les \lambda^{3-k}
	\bigg(\frac{\widetilde \chi(\lambda p_1)}{|\lambda p_1|^{\f32}}+\frac{\widetilde \chi(\lambda p_2)}{|\lambda p_2|^{\f32}}\bigg)
	\bigg(\frac{\widetilde \chi(\lambda q_1)}{|\lambda q_1|^{\f32}}+\frac{\widetilde \chi(\lambda q_2)}{|\lambda q_2|^{\f32}}\bigg), \qquad k=0,1,2.
\end{align}
Thus, upon integrating by parts twice, we have
\begin{align*}
	|\eqref{eqn:IBP nasty}|&\les \frac{1}{t^2}+
	\frac{1}{t^2} \int_0^\infty \bigg|\bigg(\frac{\partial}{\partial \lambda}
	\frac{1}{\lambda} \bigg)^{2}
	\bigg( e^{i\lambda (q_2-p_2)} b(\lambda) \bigg) 
	\bigg| \, d\lambda.
\end{align*}
The boundary term occurs if the first derivative when
integrating by parts acts on $b(\lambda)$, then to set up
the second integration by parts there is an effective
loss of three powers of $\lambda$.  We then note that
\eqref{eqn:IBP b bounds} gives us that
$|\lambda^{-2}b'(\lambda)|\les 1$.  We now move to control
the integral.  By absorbing the division by $\lambda$ into
$b(\lambda)$, we have to bound
\begin{align}
	\int_0^\infty \sum_{k=0}^2 |p_2-q_2|^k \lambda^{k-1}
	\chi(\lambda)
	\bigg(\frac{\widetilde \chi(\lambda p_1)}{|\lambda p_1|^{\f32}}+\frac{\widetilde \chi(\lambda p_2)}{|\lambda p_2|^{\f32}}\bigg)
	\bigg(\frac{\widetilde \chi(\lambda q_1)}{|\lambda q_1|^{\f32}}+\frac{\widetilde \chi(\lambda q_2)}{|\lambda q_2|^{\f32}}\bigg)\, d\lambda.
\end{align}
When $k=0$, the integral is seen to be bounded
by  
\begin{align*}
	\int_0^\infty \lambda^{-1}
	\chi(\lambda)
	\bigg(\frac{\widetilde \chi(\lambda p_1)}{|\lambda p_1|^{\f32}}+\frac{\widetilde \chi(\lambda p_2)}{|\lambda p_2|^{\f32}}\bigg)
	\bigg(\frac{\widetilde \chi(\lambda q_1)}{|\lambda q_1|^{\f32}}+\frac{\widetilde \chi(\lambda q_2)}{|\lambda q_2|^{\f32}}\bigg)\, d\lambda
	\les \int_{\R} \frac{\widetilde \chi(\lambda p_2)}{\lambda^{\f 52} p_2^{\f32}}\, d\lambda \les 1.
\end{align*}
Here we used the crude bound of a constant for all the
terms involving $q_1,q_2$, and using that $p_2\leq p_1$.
If
$k=1$, we recall that $p_2\leq p_1$, $q_2\leq q_1$ and
$p_2\approx q_2$.  We use the bound $|p_2-q_2|\les p_2$,
and bound the terms containing $q_1,q_2$ with a constant,
to see
\begin{multline*}
	\int_0^\infty |p_2-q_2|
	\chi(\lambda)
	\bigg(\frac{\widetilde \chi(\lambda p_1)}{|\lambda p_1|^{\f32}}+\frac{\widetilde \chi(\lambda p_2)}{|\lambda p_2|^{\f32}}\bigg)
	\bigg(\frac{\widetilde \chi(\lambda q_1)}{|\lambda q_1|^{\f32}}+\frac{\widetilde \chi(\lambda q_2)}{|\lambda q_2|^{\f32}}\bigg)\, d\lambda\\
	\les \int_0^\infty \frac{p_2\widetilde \chi(\lambda p_1)}{|\lambda p_1|^{\f32}}+
	\frac{\widetilde \chi(\lambda p_2)}
	{\lambda^{\f32} |p_2|^{\f12}}\, d\lambda
	\les \int_0^\infty 
	\frac{\widetilde \chi(\lambda p_2)}
	{\lambda^{\f32} |p_2|^{\f12}}\, d\lambda \les 1.
\end{multline*}
Finally, if $k=2$, we seek to bound
\begin{align*}
	\int_0^\infty |p_2-q_2|^2 \lambda
	\chi(\lambda)
	\bigg(\frac{\widetilde \chi(\lambda p_1)}{|\lambda p_1|^{\f32}}+\frac{\widetilde \chi(\lambda p_2)}{|\lambda p_2|^{\f32}}\bigg)
	\bigg(\frac{\widetilde \chi(\lambda q_1)}{|\lambda q_1|^{\f32}}+\frac{\widetilde \chi(\lambda q_2)}{|\lambda q_2|^{\f32}}\bigg)\, d\lambda.
\end{align*}
This time, we may not ignore
any terms with a crude bound of a constant.  Instead,
we use the dominating terms and $p_2\approx q_2$
to bound with
\begin{align*}
	\int_0^\infty p_2q_2 \lambda
	\bigg(\frac{\widetilde \chi(\lambda p_2)}{|\lambda p_2|^{\f32}}\frac{\widetilde \chi(\lambda q_2)}{|\lambda q_2|^{\f32}}\bigg)\, d\lambda
	\les \int_0^\infty \frac{1}{\lambda^2}
	\bigg(\frac{\widetilde \chi(\lambda p_2)}{p_2^{\f12}}\frac{\widetilde \chi(\lambda q_2)}{ q_2^{\f12}}\bigg)\, d\lambda\les \int_{\R} \frac{\widetilde \chi(\lambda p_2)}
	{\lambda^2 p_2}\, d\lambda \les 1.
\end{align*}
 
\end{proof}

It is now a simple matter to prove
\begin{lemma}\label{lem:highlow2}

	Under the assumptions of Theorem~\ref{thm:evalcancel},
	we have the bound
	\begin{multline*}
		\sup_{x,y\in \R^4}\bigg| 
		\int_{\R^{16}}
		\int_0^\infty e^{it\lambda^2}\lambda 
		\chi(\lambda) \frac{R_0^+-R_0^-}{\lambda^2}(\lambda^2)(x,z_1)
		\chi(\lambda |x-z_1|)\\
		VP_eV[R_0^+(\lambda^2)-G_0](z_4,y) 
		\widetilde \chi(\lambda |z_4-y|)
		\, d\lambda \, dz_1\, dz_2\, dz_3\, dz_4\bigg| \les |t|^{-2}.
	\end{multline*}

\end{lemma}

We can now prove Proposition~\ref{lem:awful3}.

\begin{proof}[Proof of Proposition~\ref{lem:awful3}]

	The bound follows from the bounds in 
	Lemmas~\ref{lem:lowlow}, \ref{lem:highlow2}, 
	\ref{lem:highlow} and \ref{lem:highhigh}.  The ample
	decay of $V$ more than suffices to ensure that the
	spatial integrals are bounded as in Theorem~\ref{thm:res1}.

\end{proof}

We now proceed to the proof of Theorem~\ref{thm:evalcancel}.
We note that the assumption $P_eV1=0$ is already
satisfied when there is a resonance of the second kind at zero. 
\begin{proof}[Proof of Theorem~\ref{thm:evalcancel}]
As in the previous sections, we need to understand the contribution of  \eqref{eq:M inverse cancel} to the Stone formula. Thanks to the algebraic fact \eqref{alg fact} we have three cases to consider; the case when the `+/-' difference acts on $M^{\pm}(\lambda)^{-1}$, the case when the difference acts on an inner resolvent, and the case when the difference acts on the leading/lagging resolvent. 

We first consider when the `+/-' difference acts on
$M^\pm(\lambda)^{-1}$.  By Proposition~\ref{prop:Minvcanc},
we have
 $$ M^{+}(\lambda)^{-1} - M^{-}(\lambda)^{-1}= \lambda^2 M_6 + \widetilde O _2(\lambda ^{2+}) $$
for an absolutely bounded operator $M_6$. The $\lambda$
smallness this brings to
$$
R_0^+VR_0^+vM^+(\lambda)^{-1}vR_0^+VR_0^+
- R_0^-VR_0^-vM^-(\lambda)^{-1}vR_0^-VR_0^-
$$
allows us to consider this under the framework of the
Born Series.  The analysis in Lemma~3.8 of \cite{GGeven}
can be applied to show that the contribution of this
difference to the Stone formula can be bounded by
$|t|^{-2}$ uniformly in $x,y$.

When the `+/-' difference acts on an inner resolvent,
with some work, we can again reduce this to integrals
bounded in the analysis of the Born series.  Consider,
the following as a representative term
\begin{align}\label{innerdiff}
	R_0^-(\lambda^2)V[R_0^+-R_0^-](\lambda^2)vM^{+}(\lambda)^{-1}vR_0^+(\lambda)VR_0^{+}(\lambda^2)
\end{align}
Here we note that we can express $[R_0^{+}(\lambda^2)-R_0^{-}(\lambda^2)]=c\lambda^2 + \widetilde{O}_2(\lambda^4 |z_j-z_{j+1}|^2)$, by using the
small argument expansion of the Bessel function
\eqref{J0 def}, while if $\lambda |z_j-z_{j+1}|\gtrsim 1$, one employs \eqref{R0high} with $\alpha=\f72-k$ for
$k=0,1,2$.  Then, writing
$M^+(\lambda)^{-1}=-D_2/\lambda^2+\widetilde O(1)$, we 
have
\begin{align*}
	[R_0^+-R_0^-](\lambda^2)vM^{+}(\lambda)^{-1}&=
	(c\lambda^2 + \widetilde{O}_2(\lambda^4 |z_j-z_{j+1}|^2)) v\big[-\frac{D_2}{\lambda^2}+
	\widetilde{O}(1)\big]\\
	&=-c1vD_2+\la z_j\ra^{2} \la z_{j+1}\ra^{2}
	\widetilde O_2(\lambda^2)=\la z_j\ra^{2} \la z_{j+1}\ra^{2}\widetilde O_2(\lambda^2).
\end{align*}
The last equality holds due to the identity $vD_2=vS_1D_2S_1=VP_ew$ we have $1vD_2=1VP_ew=0$.
The growth in the inner spatial variables $z_j,z_{j+1}$
can be absorbed by the decay of the potential functions
$V$ and $v$ respectively.
This again allows us to use the analysis of the Born
series in Lemma~3.8 of \cite{GGeven} to bound its
contribution to the Stone formula by $|t|^{-2}$ 
uniformly in $x,y$.  

Finally, we consider when the `+/-' difference acts
on a leading free resolvent.  By symmetry, the calculations are identical if the difference acts on
the lagging free resolvent.  We first note that
by Proposition~\ref{prop:Minvcanc}, we have
$M^+(\lambda)^{-1}=-D_2/\lambda^2 +\widetilde O_2(1)$.  
When the error term $\widetilde O_2(1)$
is substituted into
\begin{align}
	[R_0^+-R_0^-]VR_0^+vM^+(\lambda)^{-1}vR_0^+VR_0^+
\end{align}
the desired bound again falls under the framework of
the analysis of the Born series, this time using
Lemma~3.6 of \cite{GGeven}.  We now consider only the
contribution of $-D_2/\lambda^2$.

We need only consider the contribution of
\begin{align}
	\frac{R_0^+-R_0^-}{\lambda^2}&VR_0^+vD_2vR_0^+VR_0^+
	\nn\\
	&=\frac{R_0^+-R_0^-}{\lambda^2}V[R_0^+-G_0]vD_2v[R_0^+-G_0]VR_0^+\label{likebs}\\
	&+\frac{R_0^+-R_0^-}{\lambda^2}V[R_0^+-G_0]vD_2vG_0VR_0^+\label{A}\\
	&+\frac{R_0^+-R_0^-}{\lambda^2}VG_0vD_2v[R_0^+-G_0]VR_0^+\label{B}\\
	&+\frac{R_0^+-R_0^-}{\lambda^2}VG_0vD_2vG_0VR_0^+
	\label{C}
\end{align}
The smallness of $R_0^+-G_0$ occuring at `inner resolvents' in \eqref{likebs} allows us to bound this
term as in the Born series.
The remaining terms are bounded by
Lemmas~\ref{lem:awful}, \ref{lem:awful2} and
Proposition~\ref{lem:awful3}.

\end{proof}

\section{The Klein-Gordon Equation with a Potential}\label{sec:KG}

In this section we prove the bounds in Theorem~\ref{thm:KG} that control the solution of a perturbed Klein-Gordon equation.  We note that much of our
analysis is greatly simplified due to the expansions and analysis performed in previous
sections in the context of the Schr\"odinger operator.  Much of the analysis of the oscillatory integrals
proceeds similarly with the multipliers 
$\frac{\sin(t\sqrt{\lambda^2+m^2})}{\sqrt{\lambda^2+m^2}}\, \lambda$
and $\cos(t\sqrt{\lambda^2+m^2})\, \lambda$ in place of the multiplier $e^{it\lambda^2}\lambda$.

We employ the following consequence of
the classical Van der Corput lemma, see for example \cite{Stein}.
\begin{lemma}\label{lem:vdc}
	
	If $\phi:[0,1]\to \R$ obeys the bound 
	$|\phi^{(2)}(\lambda)|\geq c t>0$ for all $\lambda\in[0,1]$, and if $\psi:[0,1]\to \mathbb C$ is
	such that $\psi '\in L^1([0,1])$, then
	$$
	\bigg|\int_0^1 e^{i\phi(\lambda)}\psi(\lambda) \, d\lambda\bigg| \les t^{-\f12} \bigg\{
	|\psi(1)|+\int_0^1 |\psi^\prime (\lambda)|\, 
	d\lambda 	\bigg\}.
	$$

\end{lemma}

In particular, we employ this lemma with the phase
$\phi(\lambda)=\pm t\sqrt{\lambda^2+m^2}+\lambda \nu$
for some $\nu\in \R$.  In this case, we have
$$
|\phi''(\lambda)|=t\frac{m^2}{(m^2+\lambda^2)^{\f32}}
\geq \frac{t}{m}, \qquad m>0.
$$
Accordingly, this lemma is quite useful when analyzing
the Klein-Gordon with non-zero mass $m^2$, but an
alternative approach is required for the  wave equation,
when $m^2=0$.

Recalling \eqref{resolvid1}, we need separate analysis
for the contribution of the finite Born series, 
\eqref{bs finite}, and the singular portion of the
expansion which is sensitive to the existence of 
zero-energy resonances and eigenvalues, \eqref{bs tail}.

To control the first term in the Born series, we note
\begin{lemma}\label{KGbs1}
	
	One has the bound
	\begin{align*}
		\bigg|\int_0^\infty \frac{\sin(t\sqrt{\lambda^2+m^2})}
		{\sqrt{\lambda^2+m^2}}
		\lambda \chi(\lambda)
		[R_0^+-R_0^-](\lambda^2)(x,y)\, d\lambda\bigg| 
		\les |t|^{-\f32}, \\
		\bigg|\int_0^\infty \cos(t\sqrt{\lambda^2+m^2})
		\lambda \chi(\lambda)
		[R_0^+-R_0^-](\lambda^2)(x,y)\, d\lambda\bigg| 
		\les |t|^{-\f32},
	\end{align*}
	uniformly in $x,y$.
	
\end{lemma}

This reflects the natural dispersive decay rate of
$|t|^{-\f32}$ for a  wave-like equation in 
$\R^4$.  

\begin{proof}
	
	The bound is established by integrating by parts
	once, then using Lemma~\ref{lem:vdc}.  We need
	to consider two cases, based on the size of
	$\lambda |x-y|$.  We consider the first integral
	by writing $\sin(z)=\frac{1}{2i}(e^{iz}-e^{-iz})$,
	the second integral follows similarly using that
	$\sqrt{\lambda^2+m^2}=\widetilde O(1)$ on the support of
	$\chi(\lambda)$.
	
	In the first case, if $\lambda |x-y|\ll 1$, we have
	from \eqref{resolv wtd},
	\begin{align}\label{R0diffeps}
		R_0^+-R_0^-(\lambda^2)(x,y)=c\lambda^2 +\widetilde O_2(\lambda^2 (\lambda|x-y|)^\epsilon), \qquad
		0\leq \epsilon <2.
	\end{align}
	Using $\partial_\lambda e^{\pm it\sqrt{\lambda^2+m^2}}
	=e^{\pm it\sqrt{\lambda^2+m^2}}\frac{\pm it\lambda}{\sqrt{\lambda^2+m^2}},$ 
	we need to 
	control
	\begin{multline*}
		\bigg|\int_0^\infty 
		\frac{e^{\pm it\sqrt{\lambda^2+m^2}}}
		{\sqrt{\lambda^2+m^2}}
		\lambda \chi(\lambda)
		[R_0^+-R_0^-](\lambda^2)(x,y) \chi(\lambda |x-y|)
		\, d\lambda\bigg|\\
		=\frac{1}{t}\bigg|\int_0^\infty e^{\pm it\sqrt{\lambda^2+m^2}}
		\partial_\lambda \big(
		\chi(\lambda)
		[R_0^+-R_0^-](\lambda^2)(x,y)\big)\, d\lambda\bigg|.
	\end{multline*}
	From the expansion \eqref{R0diffeps}, there are no
	boundary terms when integrating by parts.  
	Using \eqref{R0diffeps}, we note that differentiation
	in $\lambda$ is comparable to division by $\lambda$
	we can apply Lemma~\ref{lem:vdc} with
	$\psi(\lambda)=\chi(\lambda)[2c\lambda+\widetilde O_1(\lambda (\lambda|x-y|)^\epsilon)]+\chi^\prime(\lambda) [c\lambda^2 +\widetilde O_2(\lambda^2 (\lambda|x-y|)^\epsilon)]$, then
	$\psi(\lambda)=\chi(\lambda)
	\widetilde O_1(\lambda)$, and $\psi(1)=0$,
	$\psi^\prime \in L^1([0,1])$.
	
	When $\lambda|x-y|\gtrsim 1$, we do not employ any
	of the cancellation between $R_0^+$ and
	$R_0^-$, instead we use the representation
	\eqref{JYasymp2} to bound 
	\begin{align*}
		\bigg|\int_0^\infty 
		\frac{e^{\pm it\sqrt{\lambda^2+m^2}}}
		{\sqrt{\lambda^2+m^2}}
		\lambda \chi(\lambda)
		\frac{\lambda e^{i\lambda |x-y|}\omega_+(\lambda |x-y|)}{|x-y|}\, d\lambda\bigg|
	\end{align*}
	Here we consider the `+' phase in \eqref{JYasymp2},
	the `-' phase is handled identically.  As before,
	we can integrate by parts once without boundary 
	terms and bound
	$$
	\frac{1}{t}\int_0^\infty 
	e^{\pm it\sqrt{\lambda^2+m^2}+i\lambda |x-y|}
	a(\lambda)\, d\lambda, \qquad
	a(\lambda)=e^{-i\lambda |x-y|}\partial_\lambda \bigg(\chi(\lambda)
	\frac{\lambda e^{i\lambda |x-y|}\omega_+(\lambda |x-y|)}{|x-y|}\bigg)
	$$
	One can see that $|a(\lambda)|\les \lambda^{\f12}|x-y|^{-\f12}$ and 
	$|a^\prime(\lambda)|\les (\lambda|x-y|)^{-\f12}$.  Thus, using Lemma~\ref{lem:vdc},
	we can bound with
	$$
	  \frac{1}{|t|^{\f32}} \int_0^\infty |a'(\lambda)|\, d\lambda \les 
	  \frac{1}{|t|^{\f32}} \int_0^1 \frac{\lambda^{-\f12}}{|x-y|^{\f12}} \, d\lambda
	  \les \frac{1}{|t|^{\f32} |x-y|^{\f12}}
	$$
	Noting that, on the support of $\chi(\lambda)$ if
	$\lambda |x-y|\gtrsim 1$, then $|x-y|\gtrsim 1$.  
	
\end{proof}

For the remaining terms of the Born series, we have
the following bound.

\begin{prop}\label{KGbsprop}
	
	If $|V(x)|\les \la x\ra^{-\f52-}$, then 
	for any $\ell \in \mathbb N\cup \{0\}$, we have 
	\begin{multline*}
		\sup_{x,y\in \R^4}\bigg|
		\int_0^\infty  \bigg(\cos(t\sqrt{\lambda^2+m^2})+\frac{\sin(t\sqrt{\lambda^2+m^2})}{\sqrt{\lambda^2+m^2}}
		\bigg)\\\lambda \chi(\lambda)\bigg[\sum_{k=0}^{2\ell+1}(-1)^k\big\{
		R_0^+(VR_0^+)^k
		-R_0^-(VR_0^-)^k\big\}\bigg](\lambda^2)(x,y)\,
		d\lambda\bigg| \les |t|^{-\f32}.
	\end{multline*}
	
\end{prop}

\begin{proof}
	
	We note that the case of $k=0$ was handled separately
	in Lemma~\ref{KGbs1}.  To handle $k\geq1$, we recall \eqref{H0} and the asymptotic expansions, 
	\eqref{Y0 def}, \eqref{Y0 def2} and \eqref{JYasymp2} to write
	\begin{align*} 
		R_0^{\pm}(\lambda^2)(x,y)=\left\{\begin{array}{ll}
			\frac{c}{|x-y|^2}+\widetilde O_2(\lambda^{\f32}|x-y|^{-\f12})
			& \lambda |x-y|\ll 1\\
			\frac{\lambda}{|x-y|}(e^{ i\lambda |x-y|} \omega_+(\lambda|x-y|)+
			e^{- i\lambda |x-y|} \omega_-(\lambda|x-y|)) & \lambda |x-y|\gtrsim 1
		\end{array}
		\right.
	\end{align*}
	When we encounter $[R_0^+-R_0^-](\lambda^2)(x,y)$, there is cancellation 
	when $\lambda |x-y|\ll 1$, but no useful cancellation can be found when
	$\lambda |x-y|\gtrsim 1$.  Thus when $\lambda |x-y|\gtrsim 1$, the same
	asymptotics apply, with slightly different functions $\omega_\pm$ that
	satisfy the same bounds.  
	
	Recall that we have  
		\begin{align}
		[R_0^+-R_0^-](\lambda^2)(x,y)&=c\lambda^2 +\widetilde O_2(\lambda^2(\lambda |x-y|)^{\epsilon}) \qquad 0\leq \epsilon <2, \qquad \lambda |x-y|\ll 1 \nn\\
		&=\widetilde O_2(\lambda^{\f 32}|x-y|^{-\f12})
	\end{align}

	Let $J\cup J^*=\{1,2\, \dots, k+1\}$ be a partition.
	Omitting the potentials for the moment, we 
	need to control the contribution of
	\begin{align*}
		(R_0^{\pm})^{k+1}=\prod_{j\in J}\frac{\lambda}{r_j}\left( e^{i\lambda r_j}
		\omega_{+}(\lambda r_j)
		+e^{- i\lambda r_j}\omega_{-}(\lambda r_j) \right)
		\prod_{i\in J^*}\left(\frac{1}{r_i^2}+\widetilde O_2(\lambda^{\f32}r_i^{-\f12})\right)
	\end{align*}
	where $r_j=|z_{j-1}-z_j|$ are the differences of the inner
	spatial variables with $z_0=x$ and $z_{k+1}=y$.  We note that for $j\in J$ we have 
	the support condition that $\lambda r_j\gtrsim 1$ and
	for $i\in J^*$, we have $\lambda r_i \ll 1$.
	As the different phases for the
	large $\lambda r_j$ contributions do not matter for
	our analysis, we will abuse notation slightly and write $e^{\pm i \prod_J \lambda r_j}$ to 
	indicate a sum over all
	possible combinations of positive and negative phases in 
	the product.
	
	We note that we need only use the difference of the `+'
	and `-' phases when $J=\{\emptyset\}$.  For the remaining
	cases, we can estimate each term separately without
	relying on any cancellation.
	We first consider when $J\neq \{\emptyset\}$.  We wish
	to bound 
	\begin{align*}
		\int_0^\infty e^{\pm i t\sqrt{\lambda^2+m^2}}
		\bigg(\lambda+ \frac{\lambda}{\sqrt{\lambda^2+m^2}}\bigg)
		\chi(\lambda) \prod_{j\in J}\frac{\lambda}{r_j}\left( e^{\pm i\lambda r_j}\omega_{\pm}(\lambda r_j)\right)
		\prod_{i\in J^*}\left(\frac{1}{r_i^2}+\widetilde O_2(\lambda^{\f32}r_i^{-\f12})\right)
		\, d\lambda .
	\end{align*}
	Since $J\neq \{\emptyset\}$, we can integrate by parts 
	once without boundary terms to bound
	\begin{multline*}
		\frac{1}{t}\int_0^\infty e^{\pm i t\sqrt{\lambda^2+m^2}}\partial_\lambda \bigg\{
		\big(1+ \sqrt{\lambda^2+m^2}\big)
		\chi(\lambda) \prod_{j\in J}\frac{\lambda}{r_j}\left( e^{\pm i\lambda r_j}\omega_{\pm}(\lambda r_j)\right)\\
		\prod_{i\in J^*}\left(\frac{1}{r_i^2}+\widetilde O_2(\lambda^{\f32}r_i^{-\f12})\right)\bigg\}
		\, d\lambda .
	\end{multline*}
	Thus, we wish to control integrals of the form
	\begin{align}\label{kgbsintegral}
		\frac{1}{t}\int_0^\infty e^{\pm i t\sqrt{\lambda^2+m^2}\pm i \lambda \prod_J r_j}
		\psi(\lambda)\, d\lambda.
	\end{align}
	To control $\psi$, we note the following bounds:
	\begin{align*}
		\chi(\lambda), \sqrt{\lambda^2+m^2}&=\widetilde{O}(1), \qquad
		\bigg|\frac{e^{\pm i \lambda r_j}\lambda}{r_j}
		\omega_{\pm}(\lambda r_j)\bigg| \les
		\frac{\lambda^{\f12}}{r_j^{\f32}}, \qquad
		\bigg|\frac{1}{r_i^2}+\widetilde O_2(\lambda^{\f32}r_i^{-\f12})\bigg| \les \frac{1}{r_i^2},\\
		\bigg|\partial_\lambda \bigg(\frac{e^{\pm i \lambda r_j}\lambda}{r_j}
		\omega_{\pm}(\lambda r_j)\bigg)\bigg| &\les
		\frac{\lambda^{\f12}}{r_j^{\f12}},\qquad
		\bigg|\partial_\lambda \bigg(\frac{1}{r_i^2}+\widetilde O_2(\lambda^{\f32}r_i^{-\f12})\bigg)\bigg|  =\widetilde O_1\bigg(\frac{\lambda^{\f12}}{r_i^{\f12}}\bigg).
	\end{align*}
	Where we used that $O_2(\lambda^{\f32}r_i^{-\f12})\les
	r_i^{-2}$ when $\lambda r_i \ll 1$.
	So that 
	\begin{align}
		|\psi(\lambda)|&\les \sum_{\ell=1}^{k+1} 
		\frac{\lambda^{\f12}}{r_\ell^{\f12}} \bigg[
		\prod_{\ell\neq j\in J} \frac{\lambda^{\f12}}{r_j^{\f32}} \prod_{i\in J^*}
		\frac{1}{r_i^2}  +	\prod_{j\in J} \frac{\lambda^{\f12}}{r_j^{\f32}} 
		\prod_{\ell \neq i\in J^*}\frac{1}{r_i^2}
		\bigg].
	\end{align}
	Since we have factored out the high energy phases
	$e^{\pm i\lambda r_j}$, differentiation of $\psi$ is
	comparable to division by $\lambda$, so
	\begin{align}
		|\partial_\lambda \psi(\lambda)|&\les \sum_{\ell=1}^{k+1} 
		\frac{\lambda^{-\f12}}{r_\ell^{\f12}} \bigg[
		\prod_{\ell\neq j\in J} \frac{\lambda^{\f12}}{r_j^{\f32}} \prod_{i\in J^*}
		\frac{1}{r_i^2}  +	\prod_{j\in J} \frac{\lambda^{\f12}}{r_j^{\f32}} 
		\prod_{\ell \neq i\in J^*}\frac{1}{r_i^2}
		\bigg].
	\end{align}
	Lemma~\ref{lem:vdc} shows that
	$$
	\eqref{kgbsintegral}\les \frac{1}{|t|^{\f32}}\sum_{\ell=1}^{k+1}
	\frac{1}{r_\ell^{\f12}} \bigg[
	\prod_{\ell\neq j\in J} \frac{1}{r_j^{\f32}} \prod_{i\in J^*}
	\frac{1}{r_i^2}  +	\prod_{j\in J} \frac{1}{r_j^{\f32}} \prod_{\ell \neq i\in J^*}
	 \frac{1}{r_i^2}\bigg]
	$$
	when $J\neq \{\emptyset\}$.
	
	On the other hand, when $J=\{\emptyset\}$ we need to 
	use the difference of + and - resolvents to be able to
	integrate by parts without boundary terms.  In this
	case, we need to control integrals of the form
	\begin{align*}
		\frac{1}{t}\int_0^\infty e^{\pm i t\sqrt{\lambda^2+m^2}}\partial_\lambda \bigg\{
		\big(1+ \sqrt{\lambda^2+m^2}\big)
		\chi(\lambda) \sum_{\ell=1}^{k+1}	
		\widetilde O_2(\lambda^{\f32}r_\ell^{-\f12})\prod_{\ell \neq i}
		\left(\frac{1}{r_i^2}+\widetilde O_2(\lambda^{\f32}r_i^{-\f12})\right)\bigg\}
		\, d\lambda 
	\end{align*}
	Again, we gain the extra $|t|^{-\f12}$ decay by using
	Lemma~\ref{lem:vdc}.  In this case, again using that 
	$\widetilde O_2(\lambda^{\f32}r_i^{-\f12})\les r_i^{-2}$ when $\lambda r_i \ll 1$, we have
	\begin{align*}
		|\psi(\lambda)|&\les \sum_{\ell=1}^{k+1} \frac{\lambda^{\f12}}{r_\ell^{\f12}}
		\sum_{\ell \neq i}\frac{1}{r_i^2},\qquad
		|\partial_\lambda\psi(\lambda)|\les \sum_{\ell=1}^{k+1} \frac{\lambda^{-\f12}}{r_\ell^{\f12}}
		\sum_{\ell \neq i}\frac{1}{r_i^2}.
	\end{align*}
	So that, in this case,
	$$
	\eqref{kgbsintegral}\les \frac{1}{|t|^{\f32}}\sum_{\ell=1}^{k+1} \frac{1}{r_\ell^{\f12}}
	\sum_{\ell \neq i}\frac{1}{r_i^2}.
	$$
	To see the necessary decay assumptions on the potential,
	we have to ensure that
	\begin{align}\label{spatial int}
		\bigg|
		\int_{\R^{4k}}\sum_{\ell=1}^{k+1}
		\frac{1}{r_\ell^{\f12}} \bigg[
		\prod_{\ell\neq j\in J} \frac{1}{r_j^{\f32}} \prod_{i\in J*}
		\frac{1}{r_i^2}  +	\prod_{j\in J} \frac{1}{r_j^{\f32}} \prod_{\ell \neq i\in J*}\frac{1}{r_i^2}
		\bigg]	\prod_{m=1}^k V(z_m) \, d\vec z \, \bigg |,
	\end{align}
	with $d\vec z=dz_1\, dz_2\, \dots\, dz_k$ is bounded
	uniformly in $x,y$ for every choice of partitions
	$J$ and $J^*$.  Using
	Lemma~\ref{EG:Lem}, we note that
	\begin{multline}\label{zoneint}
		\int_{\R^4}\frac{\la z_\ell\ra^{-\f52-}}{|z_{\ell-1}-z_{\ell}|^{\f12}}
		\bigg(\frac{1}{|z_{\ell}-z_{\ell+1}|^{\f32}}+
		\frac{1}{|z_{\ell}-z_{\ell+1}|^{2}}\bigg) \, dz_\ell\\
		\les \la z_{\ell-1}-z_{\ell+1}\ra^{-\f 12}
		\les \frac{1}{|z_{\ell-1}-z_{\ell+1}|^{\f12}}.
	\end{multline}
	So that, upon integrating in $z_\ell$, we pass forward
	a decay of size $|z_{\ell-1}-z_{\ell+1}|^{-\f12}$, which
	allows us to iterate the bound in \eqref{zoneint} until
	we have 
	$$
	\sup_{x,y\in \R^4}  \eqref{spatial int} \les 
	\sup_{x,y\in \R^4} \la x-y\ra^{-\f 12}\les 1,
	$$
	as desired.

\end{proof}

We note that, in our application, we need only establish
the bound in Proposition~\ref{KGbsprop} for $\ell=1$.
That is, we need only bound the first three terms of
the Born series.  To accomplish this, one can lower
the assumptions on the potential to $|V(x)|\les \la x\ra^{-\f94-}$, 
which is slightly less restrictive than
we assume in the statement of Proposition~\ref{KGbsprop}.

\begin{proof}[Proof of Theorem~\ref{thm:KG}]

We need only bound
\begin{multline}
	\sup_{x,y\in \R^4}\bigg|
	\int_0^\infty  \bigg(\cos(t\sqrt{\lambda^2+m^2})+\frac{\sin(t\sqrt{\lambda^2+m^2})}{\sqrt{\lambda^2+m^2}}
	\bigg)\lambda \chi(\lambda)\\
	[R_0^{+}(\lambda^2)VR_0^+(\lambda^2)vM^+(\lambda)^{-1}vR_0^+(\lambda^2)VR_0^{+}(\lambda^2)\\
	-R_0^{-}(\lambda^2)VR_0^-(\lambda^2)vM^-(\lambda)^{-1}vR_0^-(\lambda^2)VR_0^{-}(\lambda^2)
	]\, d\lambda\bigg|.
\end{multline}
We proceed as in the proofs of Theorems~\ref{thm:res1}, 
\ref{thm:res2} and \ref{thm:res3}, using the oscillatory
bounds in Lemmas~\ref{lem:KGlog}, \ref{lem:KGalpha},
\ref{KGlog decay2} and \ref{KGno lambda} in place of
Lemmas~\ref{log decay}, \ref{log decay2}, and
\ref{lem:fauxIBP} respectively.

Again, a sharper bound requires an interpolation with the
results in \cite{EGG}.  In considering the wave equation,
we have growth at a rate of $t/\log t$ for the sine
operator due to the bound
$$
\bigg|\int_0^\infty \sin(t\lambda)
\chi(\lambda)
\mathcal E(\lambda)\, d\lambda\bigg| 
\les \frac{t}{\log t}, \qquad t>2
$$	
when $\mathcal E(\lambda)=\widetilde O_1((\lambda \log\lambda)^{-2})$.  This can be replaced with the
bound of Lemma~\ref{lem:KGlog} for the desired bound
of $1/(\log t)$ for the Klein-Gordon using the ideas
and methods illustrated above. 

\end{proof}


	
	

For the wave equation, we need to bound integrals of the
form
\begin{align*}
	\int_0^\infty e^{\pm i t\lambda}\lambda
	\bigg(1+ \frac{1}{\lambda}\bigg)
	\chi(\lambda) \prod_{j\in J}\frac{\lambda}{r_j}\left( e^{\pm i\lambda r_j}\omega_{\pm}(\lambda r_j)\right)
	\prod_{i\in J^*}\left(\frac{1}{r_i^2}+\widetilde O_2(\lambda^{\f32}r_i^{-\f12})\right).
	\, d\lambda 
\end{align*}
Here, one cannot use the Van der Corput
lemma, Lemma~\ref{lem:vdc}.  Instead, one must use an
integration by parts argument and case analysis based
on the size of $t\pm r_j$ compared to $t$.  This
can be done as in the case when $n=2$ considered in
Section~4 of \cite{Gr2}.  The most
delicate case will be when $1$ or $k+1\in J$.  In that
case, we can safely integrate by parts once.  
Without loss of generality, we assume $t>0$.  Consider
the worst case for the index $j=1$, when the high
energy contributes the `-' phase.  The analysis for $j=k+1$ is identical.  When
$t-r_1\leq t/2$, we have that $t\leq r_1$, and we need to
bound
\begin{align*}
	\frac{1}{t}\int_0^\infty \bigg| 
	\chi(\lambda) \frac{\lambda^{\f12}}{r_1^{\f12}} \prod_{1\neq j\in J}\frac{\lambda^{\f12}}{r_j^{\f32}}
	\prod_{i\in J^*}\left(\frac{1}{r_i^2}+\widetilde O_2(\lambda^{\f32}r_i^{-\f12})\right)\bigg|
	\, d\lambda \les \frac{1}{t^{\f32}}\prod_{1\neq j\in J}\frac{1}{r_j^{\f32}}
	\prod_{i\in J^*}\frac{1}{r_i^2}.
\end{align*}
On the other hand, if $t-r_1\geq t/2$, we can integrate by
parts against the phase $e^{i\lambda (t-r_1)}$ to gain
a time decay of $|t|^{-2}$.  Interpolating between that
bound and the bound of $|t|^{-1}$ one can get from the
above integral gets the desired $|t|^{-\f32}$ decay rate.  If $j=1,k+1\notin J$, one can
integrate by parts twice without a case analysis to get
a bound of $|t|^{-2}$.

This analysis will require that $|V(x)|\les \la x\ra^{-3-}$.  As in the analysis of the Born series
for the Schr\"odinger evolution, this
can be seen by examing the case when two derivatives 
act on a phase $e^{\pm i\lambda r_j}$ when integrating by
parts, and we need to bound the integral
$$
\int_{\R^4}\frac{\la z_j \ra^{\f12} V(z_j)}{|z_{j-1}-z_j|^{\f32}}\, dz_j.
$$
This can be bounded by a constant, uniformly in $z_{j-1}$ provided
$|V(x)|\les \la x\ra^{-3-}$.  We leave
the remaining details to the interested reader.

We conclude this section by remarking that the bounds
proven for the Klein-Gordon allow us to conclude similar
bounds for wave equation with the unfortunate growth 
in $t$ for the sine operator as seen from the estimate
above used in \cite{EGG}.  We do not investigate the
high-energy dispersive bounds, as this requires a much
different approach and requires smoothness
on the potential and initial data, see \cite{CV}.  We suspect that high
energy bounds for the Klein-Gordon should follow from
the bounds for the wave equation.  Similar issues in, for example, Kato
smoothing estimates are discussed in \cite{D}.
High energy weighted $L^2$ bounds for the Klein-Gordon
were proven in \cite{KK}
in two spatial dimensions, we believe a similar analysis can be
performed in four spatial dimensions.  
Our low energy $L^1\to L^\infty$ bounds imply the weighted $L^2$ estimates 
on $L^{2,\sigma}\to L^{2,-\sigma'}$ for any $\sigma, \sigma'>2$.

\section{Spectral Theory and Integral Estimates}\label{sec:spec}
We repeat the characterization of the spectral subspaces
of $L^2(\R^4)$ and their relation to the invertibility
of operators in our resolvent expansions performed in \cite{EGG} for completeness.  
We omit the proofs.
The results below are essentially
Lemmas 5--7 of \cite{ES2} modified to suit four
spatial dimensions.  In addition, we give proofs of some integral estimates that are
used in the preceding analysis.

\begin{lemma}\label{lem:S1char}

	Suppose $|V(x)|\les \la x\ra^{-4}$.  Then 
	$f\in S_1L^2\setminus\{0\}$
	if and only if $f=wg$ for some $g\in L^{2,0-}\setminus\{0\}$ such that
	$$
		(-\Delta+V)g=0
	$$
	holds in the sense of distributions.
	
\end{lemma}

Recall that $S_2$ is the projection onto the kernel of $S_1PS_1$.
Note that for $f\in S_2L^2$, since $S_1,S_2$ and $P$ are projections
and hence self-adjoint we have
\begin{align}\label{PS2}
	0=\la S_1PS_1 f, f\ra=\la Pf,Pf\ra=\|Pf\|^2_2
\end{align}
Thus $PS_2=S_2P=0$.

\begin{lemma}\label{lem:resonance}

	Suppose $|V(x)|\les \la x\ra^{-4-}$.  Then 
	$f\in S_2L^2\setminus\{0\}$
	if and only if $f=wg$ for some $g\in L^{2}\setminus\{0\}$ such that
	$$
		(-\Delta+V)g=0
	$$
	holds in the sense of distributions.

\end{lemma}

\begin{corollary}\label{ranks}
Suppose $|V(x)|\les \la x\ra^{-4-}$. Then 
$$\textrm{Rank}(S_1)\leq \textrm{ Rank}(S_2)+1. $$
\end{corollary}

\begin{lemma}\label{vG1v kernel}

	If $|V(x)|\les \la x\ra^{-5-}$, then the kernel of 
	$S_2vG_1vS_2=\{0\}$ on $S_2L^2$.

\end{lemma}

\begin{lemma} \label{lem:eigenspace} 
 
    The projection onto the eigenspace at zero is $G_0vS_2[S_2vG_1vS_2]^{-1}S_2vG_0$.

\end{lemma}

\subsection{Oscillatory Integral Estimates}

We have the following oscillatory integral
bounds which prove useful in the preceding analysis.
Some of these Lemmas along with their proofs appear 
in Section~6 of \cite{GGodd} or Section~5 of \cite{GGeven},
accordingly we state them without proof.

\begin{lemma}\label{lem:IBP}

	If $k\in \mathbb N_0$, we have the bound
	$$
		\bigg|\int_0^\infty e^{it\lambda^2} \chi(\lambda)
		\lambda^k \, d\lambda\bigg| \les |t|^{-\frac{k+1}{2}}.
	$$

\end{lemma}

\begin{lemma}\label{lem:fauxIBP}

	For a fixed $\alpha>-1$,
	let $f(\lambda)=\widetilde O_{k+1}(\lambda^\alpha)$ be supported on the 
	interval $[0,\lambda_1]$ for some $0<\lambda_1\les 1$.
	Then, if $k$ satisfies $-1<\alpha-2k< 1$ we have
	\begin{align*}
		\bigg|\int_0^\infty e^{it\lambda^2} 
		f(\lambda)\, d\lambda\bigg|&\les |t|^{-\frac{\alpha+1}{2}}.
	\end{align*}
	
\end{lemma}

In addition, we make use of the following oscillatory
integral estimates that allow us to bound the Klein-Gordon and wave equations.

\begin{lemma}\label{lem:KGlog}

 	If $\mathcal E(\lambda)=\widetilde O_1((\lambda \log \lambda)^{-2})$, then  for $m>0$
 	$$
 		\bigg|\int_0^\infty \frac{\sin(t\sqrt{\lambda^2+m^2})}
 		{\sqrt{\lambda^2+m^2}}
 		\lambda \chi(\lambda)
 		\mathcal E(\lambda)\, d\lambda\bigg| 
 		\les \frac{1}{|\log t|}, \qquad t>2.
 	$$
 	Further,
 	$$
 		\bigg|\int_0^\infty \cos(t\sqrt{\lambda^2+m^2})
 		\lambda \chi(\lambda)
 		\mathcal E(\lambda)\, d\lambda\bigg| 
 		\les \frac{1}{|\log t|}, \qquad t>2.
 	$$

\end{lemma}

\begin{proof}

	We prove the first bound, the second bound follows
	similarly.
	We divide the integral into two pieces.  First, on
	$0<\lambda<t^{-1}$, we cannot use the oscillation.
	Instead, we use $0<m\leq \sqrt{\lambda^2+m^2}$
	to bound with
 	$$
 		\bigg|\int_0^{t^{-1}} 
 		\frac{1}{\lambda (\log\lambda)^{-2} }
 		\, d\lambda\bigg| 
 		\les \frac{1}{|\log t|}.
 	$$
 	On the remaining piece, where $\lambda \geq t^{-1}$,
 	we write 
 	$$
 		\sin(t\sqrt{\lambda^2+m^2})
 		=-\frac{\sqrt{\lambda^2+m^2}}{\lambda t}
 		\partial_\lambda \cos(t\sqrt{\lambda^2+m^2})
 	$$
 	to facilitate an integration by parts.  So we need
 	to bound
 	\begin{multline*}
 		\int_{t^{-1}}^\infty \frac{\sin(t\sqrt{\lambda^2+m^2})}
 	 	{\sqrt{\lambda^2+m^2}}
 	 	\lambda \chi(\lambda)
 	 	\mathcal E(\lambda)\, d\lambda	=\frac{1}{t}
 	 	\int_{t^{-1}}^\infty -\partial_\lambda \cos(t\sqrt{\lambda^2+m^2}) \chi(\lambda) \mathcal E(\lambda)
 	 	\, d\lambda\\
 	 	=\frac{-\cos(t\sqrt{\lambda^2+m^2})\chi(\lambda)
 	 	\mathcal E(\lambda)}{t} \bigg|^\infty_{t^{-1}}+
 	 	\int_{t^{-1}}^\infty \cos(t\sqrt{\lambda^2+m^2}) \chi(\lambda)\mathcal E(\lambda) \, d\lambda\\
 	 	\les \frac{1}{t}\frac{t}{(\log t)^2}+\int_{t^{-1}}^\infty \big| \partial_\lambda
 	 	\mathcal E(\lambda)\big| \, d\lambda
 	 	\les \frac{1}{(\log t)^2}.
 	\end{multline*}
 	The final integral is bounded as in the proof of
 	Lemma~\ref{log decay}.  For the cosine integral,
 	one uses that $\sqrt{\lambda^2+m^2}\les 1$ on the
 	support of $\chi(\lambda)$.

\end{proof}

\begin{lemma}\label{lem:KGalpha}

 	If $\mathcal E(\lambda)=\widetilde O_2(\lambda^\alpha)$ for some $-1<\alpha<1$, then  for $m>0$
 	$$
 		\bigg|\int_0^\infty \frac{\sin(t\sqrt{\lambda^2+m^2})}
 		{\sqrt{\lambda^2+m^2}}
 		\lambda \chi(\lambda)
 		\mathcal E(\lambda)\, d\lambda\bigg| 
 		\les t^{-1-\f \alpha 2}, \qquad t>2.
 	$$
 	Further,
 	$$
 		\bigg|\int_0^\infty \cos(t\sqrt{\lambda^2+m^2})
 		\lambda \chi(\lambda)
 		\mathcal E(\lambda)\, d\lambda\bigg| 
 		\les t^{-1-\f \alpha 2}, \qquad t>2.
 	$$

\end{lemma}

\begin{proof}

	We prove the first bound, the second bound follows
	similarly.
	We divide the integral into two pieces.  First, on
	$0<\lambda<t^{-\f12}$, we cannot use the oscillation.
	Instead, we use $0<m\leq \sqrt{\lambda^2+m^2}$
	to bound with
 	$$
 		\bigg|\int_0^{t^{-\f12}} 
 		\lambda^{1+\alpha}
 		\, d\lambda\bigg| 
 		\les t^{-1-\f \alpha 2}.
 	$$
 	On the remaining piece, where $\lambda \geq t^{-\f12}$,
 	we need to integration by parts twice.  So we need
 	to bound
 	\begin{multline*}
 		\frac{\mathcal E(t^{-\f12})}{t}+
 		\frac{\partial_\lambda \mathcal E(t^{-\f12})}{t^{-\f32}}+\frac{1}{t^2}
 		\int_{t^{-\f12}}^\infty \partial_\lambda\bigg(
 		\frac{\sqrt{\lambda^2+m^2}}{\lambda}\partial_\lambda \mathcal E(\lambda)
 		\bigg)\, d\lambda\\ \les t^{-1-\f \alpha 2}+
 		\frac{1}{t^2}\int_{t^{-\f12}}^\infty \lambda^{\alpha -3}\, d\lambda
 		 \les t^{-1-\f \alpha 2}
 	\end{multline*}
 	as desired.

\end{proof}

\begin{lemma}\label{KGlog decay2}

	If $\mathcal E(\lambda)=\widetilde O_2((\log \lambda)^{-k})$, then  for $k\geq 2$,
	$$
		\bigg|\int_0^\infty \frac{\sin(t\sqrt{\lambda^2+m^2})}
		 		{\sqrt{\lambda^2+m^2}}
		\lambda \chi(\lambda)
		\mathcal E(\lambda)\, d\lambda\bigg| 
		\les \frac{1}{t(\log t)^{k-1}}, \qquad t>2.
	$$
   Further,
 	$$
 		\bigg|\int_0^\infty \cos(t\sqrt{\lambda^2+m^2})
 		\lambda \chi(\lambda)
 		\mathcal E(\lambda)\, d\lambda\bigg| 
 		\les \frac{1}{t(\log t)^{k-1}}, \qquad t>2.
 	$$

\end{lemma}

\begin{proof}

	The proof follows as in the proof of Lemma~\ref{log decay2} with the modifications made above in
	Lemmas~\ref{lem:KGlog} and \ref{lem:KGalpha}.

\end{proof}

For completeness, we include the following bound which
follows from a simple integration by parts.

\begin{lemma}\label{KGno lambda}

	One has the bounds
	$$
		\bigg|\int_0^\infty \frac{\sin(t\sqrt{\lambda^2+m^2})}
		 		{\sqrt{\lambda^2+m^2}}
		\lambda \chi(\lambda)\, d\lambda\bigg| , 
		\bigg|\int_0^\infty \cos(t\sqrt{\lambda^2+m^2})
		 \lambda \chi(\lambda)\, d\lambda\bigg|
		\les \frac{1}{t}.
	$$

\end{lemma}

\subsection{Spatial Integral Estimates}

The following bound is needed to show that certain
operators are bounded.  The proof is straight-forward
and can be found in, for example \cite{GV}.

\begin{lemma}\label{EG:Lem}

	Fix $u_1,u_2\in\R^n$ and let $0\leq k,\ell<n$, 
	$\beta>0$, $k+\ell+\beta\geq n$, $k+\ell\neq n$.
	We have
	$$
		\int_{\R^n} \frac{\la z\ra^{-\beta-}}
		{|z-u_1|^k|z-u_2|^\ell}\, dz
		\les \left\{\begin{array}{ll}
		(\frac{1}{|u_1-u_2|})^{\max(0,k+\ell-n)}
		& |u_1-u_2|\leq 1\\
		\big(\frac{1}{|u_1-u_2|}\big)^{\min(k,\ell,
		k+\ell+\beta-n)} & |u_1-u_2|>1
		\end{array}
		\right.
	$$

\end{lemma}

\end{document}